\documentclass[10pt]{article}
\usepackage{fancyhdr}
\usepackage{extramarks}
\usepackage{amsmath}
\usepackage{amsthm}
\usepackage{amsfonts}
\usepackage{siunitx}
\usepackage{tikz}
\usepackage[plain]{algorithm}
\usepackage{algpseudocode}
\usepackage{multirow}
\usepackage{booktabs}
\usepackage{palatino}
\usepackage{graphicx}
\usepackage[colorlinks,linkcolor=black,anchorcolor=black,citecolor=black,urlcolor=blue]{hyperref}
\usepackage{amsmath,bm}
\usepackage{booktabs}
\usepackage{mathtools}
\usepackage{amssymb}
\usepackage{caption}
\usepackage{capt-of}
\usepackage{mciteplus}
\usepackage{cite}
\usepackage{mathrsfs}
\usepackage{parskip}
\usepackage{soul}
\usepackage{textcomp}
\usepackage[colaction]{multicol}
\usepackage[switch]{lineno}
\usepackage{lipsum}
\usepackage{etoolbox}
\usepackage{longtable}
\usepackage{array}
\usepackage{tablefootnote}
\usepackage{ragged2e}
\usepackage{lscape}

\usepackage[all,cmtip]{xy}
\usepackage{subcaption}
\usepackage{float}
\usepackage{paralist}
\usepackage[misc]{ifsym}

\newcolumntype{C}[1]{>{\centering\arraybackslash}p{#1}}
\usetikzlibrary{automata,positioning}
\newtheorem{theorem}{Theorem}[section]

\newtheorem{proposition}[theorem]{Proposition}

\theoremstyle{definition}
\newtheorem{definition}[theorem]{Definition}

\theoremstyle{definition}
\newtheorem{defin}{Definition}
\newtheorem{example}{Example}

\begin{document}
		
	\title{Persistent Dirac of Path and Hypergraph}
	\author{Faisal Suwayyid$^{{\href{mailto:faisal.suwayyid@kfupm.edu.sa}{\textrm{\Letter}}}1,2}$
		and Guo-Wei Wei$^{{\href{mailto:weig@msu.edu}{\textrm{\Letter}}}2,3,4}$ \\
		$^1$Department of Mathematics,\\
		King Fahd University of Petroleum and Minerals, Dhahran 31261, KSA.\\
		$^2$Department of Mathematics,\\
		Michigan State University, MI 48824, USA.\\
		$^3$Department of Electrical and Computer Engineering,\\
		Michigan State University, MI 48824, USA.\\
		$^4$Department of Biochemistry and Molecular Biology,\\
		Michigan State University, MI 48824, USA.
	}
	
	\date{\today} 
	
	\maketitle

		\begin{abstract}
			This work introduces the development of path Dirac and hypergraph Dirac operators, along with an exploration of their persistence. These operators excel in distinguishing between harmonic and non-harmonic spectra, offering valuable insights into the subcomplexes within these structures. The paper showcases the functionality of these operators through a series of examples in various contexts. An essential facet of this research involves examining the operators' sensitivity to filtration, emphasizing their capacity to adapt to topological changes.
			The paper also explores a significant application of persistent path Dirac and persistent hypergraph Dirac in molecular science, specifically in analyzing molecular structures. The study introduces strict preorders derived from molecular structures, which generate graphs and digraphs with intricate path structures. The depth of information within these path complexes reflects the complexity of different preorder classes influenced by molecular structures. This characteristic underscores the effectiveness of these tools in the realm of topological data analysis.
		\end{abstract}
	
		Keywords: Persistent  Hypergraph   Dirac, Persistent  Digraph   Dirac, topological data analysis, spectral data analysis, simultaneous geometric and topological analyses.
		%
		{\setcounter{tocdepth}{4} \tableofcontents}
		\newpage
		
		\setcounter{page}{1}
		\renewcommand{\thepage}{{\arabic{page}}}
		

			\section{Introduction}
			
			Exploring topological features is a significant aspect of various disciplines, including mathematics, statistics, and data science. This involves examining data’s inherent shapes and patterns to discover hidden attributes. Topology has long been integrated into mathematics, physics, biology, and engineering, as highlighted in the book by Kaczynski et al . (2004) \cite{kaczynski2004computational}. However, its recent applications in revealing data properties and structures have become increasingly notable \cite{cang2017topologynet,grbic2022aspects,chen2023persistent}.
			The growing interest in this area has spurred the development of numerous topological tools specifically designed for data analysis. Prominent among these is persistent homology, a key element of topological data analysis, elaborated in \cite{zomorodian2004computing, edelsbrunner2008persistent, bubenik2017persistence}. This method generates a series of topological spaces from a dataset through a process known as filtration, aiming to capture the dataset’s topological characteristics across various spatial scales. However, persistent homology cannot capture non-topological changes in shape or geometry. This drawback can be addressed with topological Laplacians. One of the oldest topological Laplacians is the Hodge Laplacian or de Rham-Hodge theory from differential geometry. This theory employs differential forms and the Hodge Laplacian boundary map to characterize the cohomology of a closed, oriented Riemannian manifold \cite{dodziuk1977rham}. However, this technique does not incorporate multiscale analysis, which limits its applicability in data science. In 2019, the evolutionary de Rham-Hodge theory was developed to incorporate multiscale analysis into the traditional de Rham-Hodge theory \cite{chen2021evolutionary}, creating persistent Hodge Laplacians. However, applying this advanced theory to Riemannian manifolds poses a significant computational challenge. To overcome this challenge, the persistent spectral graph (PSG) concept was introduced in the same year by integrating filtration techniques with combinatorial Laplacians \cite{wang2020persistent}. PSG, also referred to as the persistent combinatorial Laplacian  \cite{wang2020persistent} or persistent Laplacian (PL) \cite{memoli2022persistent}, represents an expansion of persistent homology into the non-harmonic spectral analysis. This development has proven highly advantageous in topological data analysis with tens of large datasets \cite{qiu2023persistent,chen2022persistent}.
			The PSG approach involves converting a point cloud into a sequence of simplicial complexes created through filtration. Its harmonic spectra align with the topological persistence observed in persistent homology. In contrast, the non-harmonic spectra are instrumental in capturing the homotopic shape evolution of the data throughout the filtration process. This dual capacity of PSG  to reflect both harmonic and non-harmonic elements enhances its utility in understanding and analyzing topological data, as discussed in \cite{meng2021persistent, chen2022persistent,qiu2023persistent}. A software package has been developed for persistent Laplacian \cite{wang2021hermes}.  
			PSG approach, akin to persistent homology, treats all data points equally. However, to better handle point cloud data with labeled information and facilitate multiscale analysis, Persistent Sheaf Laplacian (PSL) was introduced, as noted in \cite{wei2021persistent}. This method extends the framework of cellular sheaves, a concept detailed in earlier works in \cite{shepard1985cellular, hansen2019toward}, and can embed non-geometric information into the topological invariants and spectral representations. 
			Despite their advancements, the methods described so far have a common limitation: they are not sensitive to asymmetry or directed relationships in data. This shortcoming restricts their capability to encode structures that contain directed information effectively. To address this issue, a novel approach was developed that applies filtration to path complexes derived from directed graphs   \cite{grigor2012homologies}. This approach, encompassing the path Laplacian and persistent path Laplacian, as proposed in \cite{wang2022persistent}, serves as a topological Laplacian tailored for analyzing the spectral geometry and topology of data.
			Additionally, to tackle a similar limitation, the persistent hyperdigraph Laplacian was introduced. This approach differs in its use of hyperdigraph homology and persistent hyperdigraph homology, as discussed in \cite{chen2023persistent}. Unlike the previous methods, this approach focuses on capturing more intricate directional information inherent in the data, thereby offering a more nuanced understanding the topological and geometrical aspects of data. 
			
			In topological data analysis, recent advancements have introduced a new category of fundamental operators known as Dirac operators, which are essentially formal square roots of Laplacian operators, as detailed in \cite{ameneyro2022quantums}. These Dirac operators are unique because they encapsulate the same information as Laplacian operators and provide additional insights about subchain complexes. One of their fundamental properties is the ability to represent the homologies of their complexes as distinct subspaces within their kernels, enhancing their capability to detect subtle topological changes during various stages of data filtration.
			A quantum computing method and algorithm have been developed to utilize persistent Dirac operators for computing persistent Betti numbers. These numbers capture key topological features within data at multiple scales \cite{ameneyro2022quantums}. Additionally, two innovative algorithms have been introduced to address the challenge of exponential increases in computation time and memory demand with larger datasets, particularly in time series data \cite{ameneyro2022quantum}. These include a quantum Takens’s delay embedding technique that transforms time series data into point clouds in higher-dimensional spaces.
			Dirac operators also play a pivotal role in analyzing topological signals. They are instrumental in describing locally interconnected topological signals and uncovering unique properties such as explosive behavior and hysteresis loops in Dirac synchronization on fully connected networks, as discussed in \cite{calmon2022dirac}. Moreover, these operators are utilized in developing Dirac signal processing techniques, effectively filtering noisy topological signals defined on various elements of simplicial complexes, like nodes, links, and triangles, as outlined in \cite{calmon2023dirac}.
			Finally, a computational framework has been developed for molecular representation using the concept of the persistent Dirac operator, as presented in \cite{wee2023persistent}. This framework systematically investigates the properties of the spectrum of discrete Dirac matrices. These properties are leveraged to understand the geometric and topological characteristics of eigenvectors, both in terms of homology and non-homology, associated with actual molecular structures. The framework also examines how different weighting schemes affect the information encoded in Dirac eigenspectra, providing deeper insights into molecular structures.
			
			The present study presents a new series of operators, named path Dirac operators and hypergraph Dirac operators, along with their persistence attributes, aimed at extracting valuable information from path complexes and hypergraphs. These operators can identify both harmonic and non-harmonic spectra, thus offering a comprehensive view of the underlying substructures within path complexes. Their functionality and behavior are showcased through various examples in different contexts. A significant aspect of the work is the examination of the operators’ sensitivity to filtration processes, which underscores their ability to adapt to changes in topological structures.
			Furthermore, the research extends to applying persistent path Dirac operators and persistent hypergraph Dirac in the biological domain, particularly in studying molecular structures. It introduces a concept of naturally induced strict preorders derived from molecular structural properties. These preorders lead to the formation of complex graphs and digraphs enriched with intricate path complexes. The depth of information contained within these path complexes reflects the diversity and complexity inherent in the different classes of the preorders. As a result, these operators and complexes represent significant advancements in topological data analysis, particularly in their ability to provide nuanced insights into molecular and biological data.

			\section{Preliminaries}
			
			
			\subsection{Modules, chain complexes and homology}
			
			Chain complexes are fundamental in mathematics and crucial tools in various mathematical fields. They facilitate the connection of algebraic structures to mathematical objects, like simplicial complexes in algebraic topology. This connection is instrumental in extracting pertinent information to differentiate between distinct mathematical entities. In this section, we provide a collection of definitions and foundational concepts pertaining to chain complexes.
			Firstly, a ring is an abelian group with a distributive multiplication operation over addition. In this article, we denote $R$ as a ring, explicitly focusing on commutative rings with unity. Examples of such rings are the ring of integers, represented as $\mathbb{Z}$, and the ring of fractions. Rings in which every non-zero element has an inverse are known as fields, with the sets of real and complex numbers being classic examples.
			A module $M$ over the ring $R$ is an abelian group equipped with a multiplication operation involving elements from the ring $R$. This multiplication conforms to the distributive law with addition within the module $M$. Typical examples of modules are vector spaces, specifically vector spaces over real and complex numbers.
			With these definitions in place, the article will formally define chain complexes over commutative rings, delving deeper into these algebraic entities’ structural and functional aspects.
			
			
			\begin{definition}[Chain Complex]  \label{def:1}
				A chain complex of $R$-modules $(C_\bullet, d_\bullet)$ is a sequence of $R$-modules equipped with $R$-linear maps $d_n:C_n\longrightarrow C_{n-1}$ such that $d_{n}d_{n+1} = 0$ for every $n\in \mathbb{Z}$. The $n$-th homology of $(C_\bullet, d_\bullet)$ is defined to be $\text{ker}(d_n)/\text{Im}(d_{n+1})$, and is denoted by $H_n(C_\bullet)$, or simply $H_n$ if there is no ambiguity.
			\end{definition}
			
			
			The elements of $\text{ker}(d_n)$ are called cycles, and the elements of $\text{Im}(d_{n+1})$ are called boundaries. A morphism $\phi: (A_\bullet, d^A_\bullet) \longrightarrow (B_\bullet, d^B_\bullet)$ of chain complexes is a sequence of $R$-linear maps $\phi_n: A_n\longrightarrow B_n$ such that the following diagram
			\begin{equation*}
				\xymatrix{
					A_n \ar[r]^{d^A_n} \ar[d]_{\phi_n} & A_{n-1} \ar[d]^{\phi_{n-1}}\\
					B_n \ar[r]^{d^B_n} & B_{n-1}
				}
			\end{equation*}
			commutes for all $n\in \mathbb{Z}$. A subcomplex $(A_\bullet, d_\bullet|_{A_\bullet})$ of a chain complex $(C_\bullet, d_\bullet)$, denoted by $A_\bullet \subseteq C_\bullet$,  is a sequence of modules $\{A_n\}$ such that $A_n$ is a submodule of $C_n$ and that $d_n(A_n)\subseteq A_{n-1}$ for every $n\in \mathbb{Z}$. 
			
			Several techniques are available for generating chain complexes based on existing ones. In the subsequent definition, we provide procedures for creating subcomplexes derived from a particular complex as defined in \cite{bressan2016embedded}. We will frequently employ these constructions throughout this work.
			
			
			\begin{definition}[Infimum and Supremum Chain Complexes]  \label{def:2}
				Given a chain complex $(C_\bullet, d_\bullet)$, and a sequence of submodules $\{A_n|A_n\subseteq C_n\}$, the infimum chain complex $\text{Inf}_\bullet(\{A_m\})$ is the subcomplex
				\begin{equation}
					\text{Inf}_n(\{A_m\})= \bigcup_{C'_\bullet \subseteq C_\bullet,\, C'_n\subseteq A_n} C'_n .
				\end{equation}
				Similarly, the supremum chain complex $\text{Sup}_\bullet(\{A_m\})$ is the subcomplex
				\begin{equation}
					\text{Sup}_n(\{A_m\}) =  \bigcap_{C'_\bullet \subseteq C_\bullet,\, A_n\subseteq C'_n} C'_n.
				\end{equation}
			\end{definition}
			
			
			We observe that $\text{Inf}_n(\{A_m\})\leq A_n \leq \text{Sup}_n(\{A_m\}) $ for all $n\in \mathbb{N}_0$.
			\begin{equation*}
				\xymatrix@C=3em{
					\text{Inf}_{n+1}(\{A_m\}) \ar[r]^{d_{n+1}} \ar[d]&\text{Inf}_n(\{A_m\})\ar[r]^{d_{n}} \ar[d] & \text{Inf}_{n-1}(\{A_m\})\ar[d]& \\
					A_{n+1}  \ar@{_{(}.>}[d]  & A_{n} \ar@{_{(}.>}[d] & A_{n-1} \ar@{_{(}.>}[d] & \\
					\text{Sup}_{n+1}(\{A_m\}) 	\ar[r]^{d_{n+1}}&\text{Sup}_n(\{A_m\}) \ar[r]^{d_{n}} & \text{Sup}_{n-1}(\{A_m\})  &
				}
			\end{equation*}
			It can be easily shown that 
			\begin{equation}
				\text{Inf}_n(\{A_m\})= A_n \cap d^{-1}_{n}(A_{n-1}),
			\end{equation} 
			and
			\begin{equation}
				\text{Sup}_n(\{A_m\}) = A_n + d_{n+1}(A_{n+1}).
			\end{equation}
			Therefore, 
			\begin{equation}
				H_n(\text{Inf}_\bullet(\{A_m\})) \cong H_n(\text{Sup}_\bullet(\{A_m\})) ,
			\end{equation}
			and is called the embedded homology of $\{A_n\}$. 
			
			Chain complexes of the form
			\begin{equation}	
				\cdots \longrightarrow C_1 \longrightarrow C_0 \longrightarrow 0
			\end{equation} 
			are commonly utilized across numerous mathematical domains. Consequently, unless specified otherwise, we assume all chain complexes discussed in this paper conform to this particular form.
			
			\begin{example} \label{exp:1}
				Chain complexes are abundant, and creating abstract examples is relatively straightforward. This article follows an abstract method for constructing chain complexes that will be repeatedly employed. Let $V$ be a nonempty set, $R$ a commutative ring with unity, and $p$ be a non-negative integer. Elements of $V^{p+1}$ are called elementary $p$-th path. Let 
				\begin{equation}
					C_p = \displaystyle \bigoplus_{v \in V^{p+1}} R
				\end{equation}
				be the free $R$-module generated by the elementary $p$-path . For $v=(v_0,v_1\cdots,v_p)\in V^{p+1}$, and $i\in \{0,1,\cdots, p\}$, let $v^{(i)}$ be the elementary $(p-1)$-path obtained from $v$ by dropping the $i$-th element in the sequence. Define the $(p+1)$-th boundary $R$-map $\partial_{p+1}:C_{p+1}\longrightarrow C_p$ to be the map that sends $v$ to $\sum_{i=0}^{p} (-1)^iv^{(i)}$, and $\partial_0$ to be the zero map. Then, $\partial_p\partial_{p+1}=0$ for all non-negative integers $p$. Consequently, $(C_\bullet, \partial_\bullet)$ is a chain complex.
				
				For an empty subset of $V^{p+1}$, we assign the zero submodule of $C_p$, and for a nonempty subset $P$ of $V^{p+1}$, we assign the submodule generated by $P$ elements. Then, a sequence of subsets $\{P_{p}|P_p\subseteq V^{p+1}\text{, } p\in \mathbb{N}_0\}$ induces a sequence of submodules $\{\Lambda_p|\Lambda_p \text{ is induced by }P_p\text{, } p\in \mathbb{N}_0\}$, which induces an infimum chain complex $\{\Omega_{p}|\Omega_p=	\text{Inf}_p(\{A_m\})\text{, } p\in \mathbb{N}_0\}$ as defined in Definition \ref{def:2}.
			\end{example}
			
			
			\subsection{Laplacian and Dirac of chain complexes}
			
			When applied over subrings of the complex number field, the Laplacian and Dirac operators exhibit numerous algebraic and computational characteristics. For example, when considering chain complexes defined over real numbers or all complex numbers, the dimension of the $n$-th homology can be succinctly expressed as the rank of the $n$-th Laplacian operator. This section provides the Laplacian and Dirac operator definitions for chain complexes composed of vector spaces over the ring $R$, where $R$ can be either real or complex numbers. In the remaining part of this article, we assume that $R$ is either $\mathbb{R}$ or $\mathbb{C}$, and we use the notation $\mathbb{K}$ instead of $R$.
			
			We first define the standard inner products to define Laplacian and Dirac operators. Given a basis $\mathbb{B}=\{v_{\alpha}\}_{\alpha \in J}$ of a vector space $\mathbb{V}$, we define the standard inner product to be
			
			\begin{equation*}
				\langle   v_{i},  v_{j}\rangle =\left\{
				\begin{array}{ll}
					1, & \hbox{$i=j$;} \\
					0, & \hbox{otherwise.}
				\end{array}
				\right.
			\end{equation*}
			for all $i,j \in J$, an index set.
			
			
			\begin{defin}[Laplacian Operators] \label{def:3}
				Let $(C_\bullet, d_\bullet)$ be a chain complex consisting of vector spaces over $\mathbb{K}$. Then, for a fixed basis for each $C_n$ equipped with the standard inner product, e.g., the basis consists of orthonormal elements, the $n$-th Laplacian operator is defined to be
				\begin{equation}
					\Delta_n = \Delta_n^{\rm Up}+\Delta_n^{\rm Down},
				\end{equation}
				where $\Delta_n^{\rm Up} = d_{n+1}d_{n+1}^*$, $\Delta_n^{\rm Down} = d^*_nd_n$, and $d^*_n$ is the adjoint operator of $d_n$.
			\end{defin}
			

			Laplacian operators possess several noteworthy algebraic characteristics with significant computational implications, whether applied to real or complex numbers. Initially, Laplacian operators exhibit the properties of being Hermitian, self-adjoint, and non-negative semi-definite operators. Additionally, all of their eigenvalues are real and non-negative. Furthermore, Laplacian operators demonstrate a range of decompositional and homological properties that hold considerable computational importance. We recommend referring to \cite{wee2023persistent} for details. For a linear operator $T$ between two finite dimensional vector spaces, we denote its nullity by $\eta(T)$.
			
			
			\begin{proposition} \label{prop:1}
				Let $(C_\bullet, d_\bullet)$ be a chain complex over $\mathbb{K}$. Then,
				for all $p\in \mathbb{N}_0$, we have:
				\begin{enumerate}
					\item $C_{p}=\ker( \Delta_{p})\oplus \mathrm{Im}d_{p+1}\oplus\mathrm{Im}(d_{p}^{\ast})$. 
					\item \label{prop:1:3} $\ker \Delta_{p}=\ker (d_{p})\cap \ker (d_{p+1}^{\ast})\cong H_{p}(C_\bullet)$.
				\end{enumerate} 
			\end{proposition}
			
			
			Property \ref{prop:1:3} in the previous proposition states that to calculate Betti numbers, we determine the dimension of the null spaces of Laplacian operators. Moving forward, we will introduce the definitions of Dirac operators.
			
			
			\begin{defin} \label{def:4}
				Let $(C_\bullet, d_\bullet)$ be a chain complex consisting of vector spaces over $\mathbb{K}$. Then, for a fixed basis for each $C_n$ equipped with the standard inner product, the $p$-th Dirac operator  $D_p$ is defined to be
				\begin{equation}
					\label{eqn:dirac}
					\renewcommand{\arraystretch}{1.5}
					D_p = 
					\begin{bmatrix}
						\textbf{0}_{n_0\times n_0}    & \mathbf{B}_1    & \textbf{0}_{n_0\times n_2}     & \cdots      & \textbf{0}_{n_0\times n_{p}} & \textbf{0}_{n_0\times n_{p+1}}  \\
						
						\mathbf{B}_1^*  & \textbf{0}_{n_1\times n_1}      & \mathbf{B}_2    & \cdots      & \textbf{0}_{n_1\times n_{p}} & \textbf{0}_{n_1\times n_{p+1}}  \\
						
						\textbf{0}_{n_2\times n_0}      & \mathbf{B}_2^*  & \textbf{0}_{n_2\times n_2}       & \cdots    & \textbf{0}_{n_2\times n_{p}} & \textbf{0}_{n_2\times n_{p+1}}   \\
						
						\vdots & \vdots & \vdots & \ddots & \vdots & \vdots  \\
						
						\textbf{0}_{n_{p}\times n_0}       & \textbf{0}_{n_{p}\times n_1}       & \textbf{0}_{n_{p}\times n_2}       & \cdots       &  \textbf{0}_{n_{p} \times n_{p}}       & \mathbf{B}_{p+1}\\
						
						\textbf{0}_{n_{p+1}\times n_0}       & \textbf{0}_{n_{p+1}\times n_1}       & \textbf{0}_{n_{p+1}\times n_2}       & \cdots       & \mathbf{B}_{p+1}^*  & \textbf{0}_{n_{p+1}\times n_{p+1}}
					\end{bmatrix},
				\end{equation}
				where $\mathbf{B}_n $ is the matrix representation of $d_n$, and $	\mathbf{B}_n^*$ is its adjoint matrix.
				
			\end{defin}
			
			
			Similar to Laplacian operators, Dirac operators exhibit numerous algebraic properties and possess certain connections with Laplacian operators, as discussed in \cite{wee2023persistent}. Firstly, for a non-negative integer $p$, the operator $D_p$ is Hermitian and self-adjoint. Secondly,
			\begin{equation} \label{eq:5}
				D_p^2 =\begin{bmatrix}
					L_0    & \textbf{0}_{n_0\times n_1}    & \textbf{0}_{n_0\times n_2}    & \cdots      & \textbf{0}_{n_0\times n_p}  & \textbf{0}_{n_0\times n_{p+1}} \\
					
					\textbf{0}_{n_1\times n_0}   & L_1      & \textbf{0}_{n_1\times n_2}    & \cdots      & \textbf{0}_{n_1\times n_p} & \textbf{0}_{n_1\times n_{p+1}} \\
					
					\textbf{0}_{n_2\times n_{0}}   & \textbf{0}_{n_2\times n_{1}}   & L_2      & \cdots    & \textbf{0}_{n_2\times n_{p}}  & \textbf{0}_{n_2\times n_{p+1}}  \\
					
					\vdots & \vdots & \vdots & \ddots & \vdots & \vdots  \\
					
					\textbf{0}_{n_{p}\times n_{0}}         & \textbf{0}_{n_{p}\times n_{1}}         & \textbf{0}_{n_{p}\times n_{2}}       & \cdots       & L_p      & \textbf{0}_{n_{p}\times n_{p+1}} \\
					
					\textbf{0}_{n_{p+1}\times n_{0}}     & \textbf{0}_{n_{p+1}\times n_{1}}       &\textbf{0}_{n_{p+1}\times n_{2}}     & \cdots       & \textbf{0}_{n_{p+1}\times n_{p}}  & L^{\rm Down}_{p+1}
				\end{bmatrix},
			\end{equation}
			where $L_i$ is the matrix representation of $\Delta_i$, and $L^{\rm Down}_{p+1}$ is the representation matrix for $\Delta_{p+1}^{\rm Down}$.  Therefore,
			\begin{equation}
				\text{ker}(D_p)=\text{ker}(D_p^2)\cong(\bigoplus_{i=0}^{p}\text{ker}(L_i))\bigoplus\text{ker}(L_{p+1}^{\rm Down})			
			\end{equation}
			and
			\begin{equation}\label{eq:1}
				\eta(D_p)=\eta(L_{p+1}^{\rm Down})+\sum_{i=0}^{p}\eta(L_i),
			\end{equation}
			which implies that $\eta(D_p)\geq \sum_{i=0}^{p}\eta(L_i)$ with equality if and only if $ \eta(L_{p+1}^{\rm Down})=0$, that is , if and only if $d_{p+1}$ is injective. We next have the following proposition.
			
			\begin{proposition}\label{prop:2}
				Let $(C_\bullet, d_\bullet)$ be a chain complex over $\mathbb{K}$. Then,
				\begin{enumerate}
					\item  For all $p\in \mathbb{N}_0$, all eigenvalues of $D_p$ are real numbers.
					\item For all $p\in \mathbb{N}_0$, if $\lambda$ is an eigenvalue of $D_p$, then $-\lambda$ is an eigenvalue of $D_p$.
				\end{enumerate}
			\end{proposition}
			
			Further, if $\lambda$ is an eigenvalue of $D_p$, then $\lambda^2$ is an eigenvalue of $L_i$ for some $0\leq i\leq p$ or of $L^{\rm Down}_{p+1}$. Conversely, if $\lambda$ is an eigenvalue of $L_i$ for some $0\leq i\leq p$ or of $L^{\rm Down}_{p+1}$, then $\pm \sqrt{\lambda}$ is an eigenvalue of $D_p$.


			
			\section{Path Dirac}
			
			Path homology and hypergraphs are emerging as powerful mathematical tools, especially for encapsulating higher-order relationships. They are increasingly utilized in diverse fields such as computer science, bioinformatics, and data mining. Extracting relevant information from paths and hypergraphs, particularly those with geometric interpretations, heavily relies on chain complexes. Over the past few decades, various chain complexes, including hypergraph complexes and simplicial complexes, have been associated with hypergraphs, as discussed in \cite{bressan2016embedded}.
			Additionally, path complexes have been introduced more recently as an extension of these existing complexes. This development is highlighted in \cite{grigor2012homologies, grigor2020path}. Path complexes have proven particularly valuable in data clustering and various other applications, as evidenced by studies like \cite{wang2022persistent, chen2023persistent}.
			This section will provide an overview of the fundamental concepts related to path homology and path complexes. Our focus will extend to illustrating results that support the construction of associated path complexes for both hypergraphs and directed graphs (digraphs). This discussion aims to comprehensively understand how these mathematical tools can effectively represent and analyze complex relational structures in various applications.
			
			\subsection{Path complex and path homology}

			Let $V$ be a nonempty set. A path complex $P$ over $V$ is a disjoint union of collection of subsets $ \{P_p|P_p\subset V^{p+1},\, P_0\neq \emptyset\}_{p\in \mathbb{N}_0}$ such that for every non-negative integer $p$, and $v=(v_0, v_1,\cdots,v_{p+1})\in P_{p+1}$, $v^{(0)}=(v_1,\cdots,v_{p+1})$ and $ v^{(p+1)}=(v_0, v_1,\cdots,v_p)\in P_p$. A path complex induces a chain complex as in Example \ref{exp:1}. We explain the construction thoroughly here. Let $V$ be a nonempty set and $P$ be a path complex over $V$. Let $p$ be a non-negative integer, and 
			\begin{equation*}
				\Lambda_p =  \text{span} \{ v=(v_0, v_1,\cdots,v_{p})\in P_{p}\} 
			\end{equation*} 
			be the vector space over the field $\mathbb{K}$ generated by all elementary $p$-th path in $P$. Let $(C_\bullet, \partial_\bullet)$ be the chain complex induced by $V$ as in Example \ref{exp:1}, where $\partial_{p+1}:C_{p+1}\longrightarrow C_p$ is defined by 
			$$\partial_{p+1}(v_0,v_1,\cdots, v_{p+1}) = \sum_{i=0}^{p} (-1)^i(v_0,\cdots, \hat{v_i},\cdots, v_{p+1})$$
			for all $p\geq 0$, and $\partial_0$ is the zero map, where $\hat{v_i}$ means $v_i$ is omitted from the sequence. Here $\Lambda_\bullet = (\Lambda_p)_{p\geq 0}$ may not be chain complex with boundary maps of $(C_\bullet, \partial_\bullet)$. However, it contains an infimum subchain complex $(\Omega_\bullet, \partial_\bullet)$ as in Definition \ref{def:2},
			\begin{equation}
				\cdots \Omega_p \stackrel{\partial_p} \longrightarrow \Omega_{p-1} \stackrel{\partial_{p-1}} \longrightarrow \cdots \stackrel{\partial_1} \longrightarrow  \Omega_0 \stackrel{\partial_0} \longrightarrow  0,
			\end{equation}
			where
			$$\Omega_{p} = \{v\in \Lambda_{p}| \partial_{p}v\in \Lambda_{p-1}\},$$ 
			whose elements are called $\partial$-invariant $p$-paths, and whose homology groups $(H_p)_{p\geq 0}$ are called the path homology groups of the path complex $P$.
			

			\subsection{Path homology of directed graphs}
			
			Recall that a digraph $G$ is a pair $G=(V, E)$ of sets, where $V$ is a nonempty set whose elements are called vertices of $G$, and $E$ is a collection of subsets of $V\times V$ whose elements are called directed edges. In this case, $G$ is called a digraph on $V$. A {\it walk} in $G$ is an alternating sequence of vertices and edges of the form $(v_0,e_1,v_1,\cdots,e_n,v_n)$ such that $v_{i-1}\neq v_{i}$, and $e_i = (v_{i-1},v_{i})$ for $1\leq i \leq n$, where $n$ is called the length of the walk. The sequence $(v_0,v_1,\cdots,v_n)$ is called the anchor sequence of the walk. Sequences of the form $(v_0)$ are called trivial walks of length zero. Digraph anchor sequences can naturally induce path complexes. In \cite{grigor2012homologies, grigor2020path}, a general definition of path homologies on digraphs is given. In this work, we consider the path homologies of density two, which gives an extended version of digraph homologies, which we will keep calling digraph homologies. 
			
			\begin{defin} \label{def:7}
				Let $G=(V, E)$ be a digraph on $V$, and $p$ be a non-negative integer. Let $P_p$ be the collection of all anchor sequences of length $p$ in $G$, and $P= \{P_p\subset V^{p+1}\}_{p\in \mathbb{N}_0}$ be the set of these collections.
				Then, $P$ is the path complex induced by $G$.
			\end{defin}
			
			Since digraphs have many properties, the first three components of path complexes can be generated by the characterization in Proposition \ref{prop:5}. 
			
			\begin{proposition}[\cite{grigor2012homologies, grigor2020path}] \label{prop:5}
				Let $G$ be a digraph. Then, the zero-degree component of the chain complex of its path complex is generated by the vertices, the first-degree one is generated by $(v_0,v_1)$ such that $v_0$ is adjacent to $v_1$, and the second-degree one is generated by anchor sequences of the form $(v_0,v_1,v_2)$ such that $v_0$ is adjacent to $v_2$, $v_0\neq v_2$, and of the form $(v_0,v_1,v_2)-(v_0,v_1',v_2)$, with possibly $v_0= v_2$.
			\end{proposition}
			
			Proposition \ref{prop:5} asserts that path complexes induced by digraphs effectively represent triangles and squares within directed graphs. However, various complexities arise when describing the generators of the higher-degree components, necessitating additional definitions and fundamental findings to address them.
			
			Every digraph $G$ has an underlying graph $G'$, obtained by forgetting the orientation of the edges, for which both induce the same zero-degree homology. Therefore, the rank of the first homology $H_0$ of $G$ over principal ideal domains can be computed using the formula of graph cases. Hence, $ H_0\cong R^{|C|}$, where $C$ is the set of the connected components of the underlying graph. We have a natural embedding from the complex induced by a digraph to the complex induced by its underlying graph (see hypergraphs). Therefore, the rank can be given explicitly  by
			\begin{equation} \label{eq:3}
				\text{rank}(H_1) = |E_{G'}|-|V_{G'}|+|C_{G'}|+|S|-\text{rank}(\text{Im}d_2),
			\end{equation}
			where $S=\{\{(v_0,v_1), (v_1,v_0)\}\subset E_G\}$. Hence, an upper limit can be given by
			\begin{equation}
				\text{rank}(H_1)\leq |E_{G'}|-|V_{G'}|+|C_{G'}|+|S|.
			\end{equation}


			\subsection{Path Dirac operators}
			
			When we equip $C_p$ with the inner product defined by 
			\begin{equation*}
				\langle   x,  y\rangle =\left\{
				\begin{array}{ll}
					1, & \hbox{$x=y$;} \\
					0, & \hbox{otherwise.}
				\end{array}
				\right.
			\end{equation*}
			for all sequences $x,y$ of elements in $V$ of length $p+1$, this inner product restricts to an inner product on $\Omega_p$. By fixing an orthonormal basis for each $\Omega_p$, we can define the $p$-th path Dirac operator to be
			the $p$-th Dirac operator  $D_p$ defined as
			\begin{equation}
				\label{eqn:dirac}
				\renewcommand{\arraystretch}{1.5}
				D_p = 
				\begin{bmatrix}
					\textbf{0}_{n_0\times n_0}    & \mathbf{B}_1    & \textbf{0}_{n_0\times n_2}     & \cdots      & \textbf{0}_{n_0\times n_{p}} & \textbf{0}_{n_0\times n_{p+1}}  \\
					
					\mathbf{B}_1^*  & \textbf{0}_{n_1\times n_1}      & \mathbf{B}_2    & \cdots      & \textbf{0}_{n_1\times n_{p}} & \textbf{0}_{n_1\times n_{p+1}}  \\
					
					\textbf{0}_{n_2\times n_0}      & \mathbf{B}_2^*  & \textbf{0}_{n_2\times n_2}       & \cdots    & \textbf{0}_{n_2\times n_{p}} & \textbf{0}_{n_2\times n_{p+1}}   \\
					
					\vdots & \vdots & \vdots & \ddots & \vdots & \vdots  \\
					
					\textbf{0}_{n_{p}\times n_0}       & \textbf{0}_{n_{p}\times n_1}       & \textbf{0}_{n_{p}\times n_2}       & \cdots       &  \textbf{0}_{n_{p} \times n_{p}}       & \mathbf{B}_{p+1}\\
					
					\textbf{0}_{n_{p+1}\times n_0}       & \textbf{0}_{n_{p+1}\times n_1}       & \textbf{0}_{n_{p+1}\times n_2}       & \cdots       & \mathbf{B}_{p+1}^*  & \textbf{0}_{n_{p+1}\times n_{p+1}}
				\end{bmatrix},
			\end{equation}
			
			where $\mathbf{B}_i $ is the matrix representation of $\partial_i$, and $\mathbf{B}_i^*$ is the adjoint matrix of $\mathbf{B}_i $ for all $0\leq i\leq p+1$.
			
			In general, for an operator $\partial: X\rightarrow Y$, with a matrix representation $\mathbf{B}$, the matrix representation $\mathbf{B}^*$ of its adjoint operator may not be the adjoint matrix $\overline{\mathbf{B}}^t$. In fact, it is the matrix that satisfies 
			$$ \overline{\mathbf{B}^*}^tO_X = O_Y\mathbf{B},$$
			where $O_X$ and $\ O_Y$ are the matrix representation of the inner products on $X$ and $ Y$ respectively, and the bar over $\mathbf{B}^*$ indicates we take the complex conjugate. Choosing an orthonormal basis, however, yields that $\mathbf{B}^*$ is the adjoint matrix of $\mathbf{B}$. Therefore, $D_p$ is symmetric and has real eigenvalues whose spectrum is symmetric. That is, if 
			\[
			\text{Spectra}_+(D_p) = \{(\lambda_1)_p, (\lambda_2)_p, \cdots, (\lambda_N)_p  \}
			\]
			is the collection of all the non-negative eigenvalues of $D_p$, then
			\[
			\text{Spectra}_-(D_p) = \{(-\lambda_1)_p, (-\lambda_2)_p, \cdots, (-\lambda_N)_p  \}.
			\]
			Further, recall that, by \ref{prop:2}, 
			\begin{equation}\label{eq:4}
				\eta(D_p)=\beta_0 + \beta_1+\cdots + \beta_p+ \eta(L_{p+1}^{\rm Down}),
			\end{equation}
			where $L_{p+1}^{\rm Down}$ is the $p+1$-th down Laplacian of the path homology.
			
			In the following examples, we follow the same notation used in Example \ref{exp:1} to keep the presentation consistent, where the notation is summarized in Table \ref{table:1}. 

			\begin{figure}[t] \label{fig:4}
				\centering
				\begin{subfigure}{0.24\textwidth}
					\centering
					\includegraphics{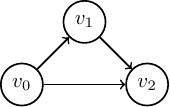}
					\caption{} 
					\label{fig:4a}
				\end{subfigure}%
				\begin{subfigure}{0.24\textwidth}
					\centering
					\includegraphics{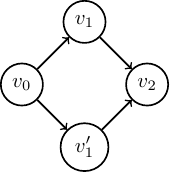}
					\caption{}
					\label{fig:4b}
				\end{subfigure}%
				\begin{subfigure}{0.31\textwidth}
					\centering
					\includegraphics{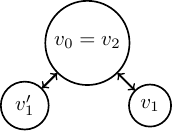}
					\caption{}
					\label{fig:4c}
				\end{subfigure}%
				\caption{The essential subgraphs that generate the second-degree component of the chain complex of path complex induced by digraphs. {(A)} shows an essential triangle, {(B)} an essential square where $v_1\neq v_1'$, and {(C)} an essential fork-like in digraphs.} 
			\end{figure}
			
			
			
			\begin{table}[t!]
				\centering
				\begin{tabular}{@{}ll}
					\hline
					$V$ & nonempty set of vertices of a digraph \\
					$(v_0, v_1, \cdots, v_p)$ & anchor sequence of length $p$\\
					$C_p$       & the vector space generated by all sequences of vertices\\
					&  of length $p+1$\\
					$P_p$     & set of all anchor sequences of length $p$ in the digraph\\
					$\Lambda_p$ & the vector space spanned by $P_p$ elements             \\
					$\Omega_{p}$          & the $\partial$-invariant $p$-paths  \\
					\hline
				\end{tabular}
				\caption{Summary of the notations used for chain complexes of path complexes constructions induced by a digraph in the article.}	\label{table:1}
			\end{table}
			

			\begin{example}\label{exp:2}
				
				This example demonstrates the computation of path complexes of digraphs and the application of Proposition \ref{prop:5}. It further emphasizes the impact of graphs’ topological and geometric features on Dirac operators. To illustrate this point, we consider two directed graphs, as shown in Figure \ref{fig:2}. These graphs have the same structural components, comprising three edges and three vertices. The sole difference between them lies in the orientation of one of their edges. This comparison highlights how a minor variation in graph structure, such as edge direction, can affect the analysis and interpretation of path complexes and Dirac operators.
				
				
				\begin{figure}[t]
					\centering
					\begin{subfigure}{.3\textwidth}
						\centering
						\includegraphics{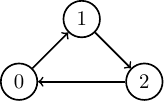}
						\caption{} \label{fig:2a}
					\end{subfigure}%
					\begin{subfigure}{.3\textwidth}
						\centering
						\includegraphics{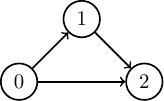}
						\caption{} \label{fig:2b}
					\end{subfigure}%
					\caption{Two digraphs to demonstrate the Dirac operator's ability to distinguish between them.}
					\label{fig:2}
				\end{figure}
				

				The first three sets of anchor sequences of both directed graphs are
				\begin{eqnarray*}
					P^1_0=P^2_0&=&\{(0), (1), (2)\}, \\
					P^1_1&=&\{(0, 1), (1, 2), (2, 0)\}, \\
					P^2_1&=&\{(0, 1), (1, 2), (0, 2)\}, \\
					P^1_2&=&\{(0, 1, 2), (1, 2, 0), (2, 0, 1)\}, \\
					P^2_2&=& \{(0, 1, 2)\},
				\end{eqnarray*}
				where for $i\in\{0,1,2\}$, $P^1_i$ is for Figure \ref{fig:2a} digraph, and $P^2_i$ is for Figure \ref{fig:2b} digraph. Applying Proposition \ref{prop:5}, we find that the components of the digraph complexes are
				\begin{eqnarray*}
					\Omega^1_0 = \Omega^2_0&=&\text{span}<\{(0), (1), (2)\}>, \\
					\Omega^1_1&=&\text{span}<\{(0, 1), (1, 2), (2, 0)\}>, \\
					\Omega^2_1&=&\text{span}<\{(0, 1), (1, 2), (0, 2)\}> ,\\
					\Omega^1_p &=& 0\  \forall p\geq 2,\\
					\Omega^2_2&=&\text{span}<\{(0, 1,2))\}> ,\\
					\Omega^2_p&=& 0\ \forall p\geq 3,
				\end{eqnarray*}
				There are no triangles or squares of the types characterized in the proposition in digraph \ref{fig:2a}; however, there is one triangle in digraph \ref{fig:2b}. Therefore, all components of a degree higher than one are zero for the first digraph. We compute the boundary maps, Laplacian, and Dirac operators of Figure \ref{fig:2a} digraph. In general, $\mathbf{B}_0$ is zero, and $\mathbf{B}_1$ can be computed as follows:
				\begin{eqnarray*}
					\partial_{1}\left(
					\begin{array}{c}
						(0,1)\\ (1,2)\\ (2,0)\\
					\end{array}
					\right)&=&\left(
					\begin{array}{ccccc}
						-1& 1& 0\\
						0& -1& 1\\
						1& 0& -1
					\end{array}
					\right)\left(
					\begin{array}{cc}
						(0)\\ (1)\\ (2) \\
					\end{array}
					\right).
				\end{eqnarray*}
				
				Since higher components are zero spaces, we compute only the zeroth Dirac operator $D_0$. $D_0$ will contains only $\mathbf{B}_1$ and $\mathbf{B}^*_1$, and therefore 
				
				\begin{center}
					$ D_0=\left(\begin{array}{cccccc}
						0& 0& 0& -1& 0& 1\\
						0& 0& 0& 1& -1& 0\\
						0& 0& 0& 0& 1& -1\\
						-1& 1& 0& 0& 0& 0\\
						0& -1& 1& 0& 0& 0\\
						1& 0& -1& 0& 0& 0 
					\end{array}\right)$
				\end{center}
				
				We also compute $L_0$ and $L_1$ since $\mathbf{B}_1$ is calculated.
				
				\begin{center}
					
					$L_0=\left( \begin{array}{ccc}
						2& -1& -1\\
						-1&  2& -1\\
						-1& -1&  2
					\end{array}\right)$\hspace{1cm}
					$L_1= \left( \begin{array}{ccc}
						2& -1& -1\\
						-1&  2& -1\\
						-1& -1&  2
					\end{array}\right)$
				\end{center}
				
				When we compute $D_0^2$, we obtain
				\begin{center}
					$ D_0^2=\left(\begin{array}{cccccc}
						2& -1& -1& 0& 0& 0\\
						-1& 2& -1& 0& 0& 0\\
						-1& -1& 2& 0& 0& 0\\
						0& 0& 0& 2& -1& -1\\
						0& 0& 0& -1& 2& -1\\
						0& 0& 0& -1& -1& 2
					\end{array}\right)$
				\end{center}
				validating equation \ref{eq:5}.
				Table \ref{table:2} and Table \ref{table:3} present characteristics of both directed graphs as identified through Laplacian and Dirac operators.
				
				
				\begin{table}[t] 
					\centering
					\begin{tabular}{@{}lllll}
						\hline
						n & Betti Number $\beta_n$ &  $\text{Spec}(L_n)$ & Dirac Nullity & $\text{Spec}(D_n)$ \\ [0.5ex] 
						\hline
						0 & 1& 0, 3, 3   & 2 & $0, 0, \pm\sqrt{3},\pm\sqrt{3} $\\ 
						1 & 1&0, 3, 3   & 3 & $0, 0, 0, \pm\sqrt{3},\pm\sqrt{3} $\\
						\hline
					\end{tabular}
					\caption{Captured features by both Laplacian and Dirac operators on the digraph \ref{fig:2a}.}
					\label{table:2}
				\end{table}
				
				
				
				\begin{table}[t] 
					\centering
					\begin{tabular}{@{}lllll}
						\hline
						n & Betti Number $\beta_n$ &  $\text{Spec}(L_n)$ & Dirac Nullity & $\text{Spec}(D_n)$ \\ [0.5ex] 
						\hline
						0 & 1 & 0, 3, 3   & 2 & $0, 0, \pm\sqrt{3},\pm\sqrt{3} $\\ 
						1 & 0 & 3, 3, 3   & 1 & $0, \pm\sqrt{3},\pm\sqrt{3}, \pm\sqrt{3} $\\
						\hline
					\end{tabular}
					\caption{Captured features by both Laplacian and Dirac operators on the digraph \ref{fig:2b}.}
					\label{table:3}
				\end{table}
				
				
			\end{example}
			

			
			\section{Hypergraph Dirac}
			Hypergraphs represent an expansion from traditional graph structures, characterized by the inclusion of hyperedges, which connect more than two vertices, unlike standard graph edges. This fundamental difference enhances hypergraphs’ representation and analytical capabilities, offering a more extensive framework for capturing and understanding complex relationships.
			This section delves into an adjusted version of the typical hypergraph complexes. We introduce an operator termed the hypergraph Dirac operator. This operator is designed to extract and analyze information about hypergraphs’ connectivity and geometric properties effectively. Its introduction advances the ability to examine intricate relationships within hypergraphs, offering deeper insights into their structural and connective complexities.
			
			\subsection{Hypergraph complexes}
			Similar to digraphs, a hypergraph $G$ is a pair $G=(V, E)$ of sets, where $V$ is a nonempty set whose elements are called vertices of $G$, and $E$ is a collection of subsets of $V$ whose elements are called hyperedges. In this case, $G$ is a hypergraph on $V$. Walks in hypergraphs are defined analogously to the digraph case; however, there are subtle differences. A walk in $G$ is an alternating sequence of vertices and edges of the form $(v_0,e_1,v_1,\cdots,e_n,v_n)$ such that $v_{i-1}\neq v_{i}$, and $\{v_{i-1},v_{i}\}\subseteq e_i$ for $1\leq i \leq n$. Its anchor sequence is then $(v_0,v_1,\cdots,v_n)$ of length $n$. This article considers an adjusted version of hypergraph homologies induced by their anchor sequences analogous to digraphs. Further, we have the following characterization of the first three components.
			
			\begin{figure}[t] \label{fig:1}
				\centering
				\begin{subfigure}{0.24\textwidth}
					\centering
					\includegraphics{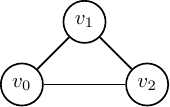}
					\caption{} 
					\label{fig:1a}
				\end{subfigure}%
				\begin{subfigure}{0.24\textwidth}
					\centering
					\includegraphics{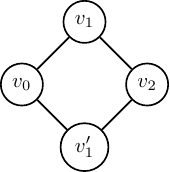}
					\caption{}
					\label{fig:1c}
				\end{subfigure}%
				\begin{subfigure}{0.31\textwidth}
					\centering
					\includegraphics{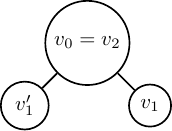}
					\caption{}
					\label{fig:1e}
				\end{subfigure}%
				\caption{The essential subgraphs that contribute to generating the second-degree hypergraph complex component. {(A)} shows an essential triangle, {(B)} an essential square where $v_1\neq v_1'$, and {(C)} an essential fork-like in hypergraphs. } 
			\end{figure}
			
			
			
			\begin{proposition}[\cite{grigor2012homologies, grigor2020path}] \label{prop:3}
				Let $G$ be a hypergraph. Then, the zero-degree component of the chain complex of its path complex is generated by the vertices, the first-degree one is generated by $(v_0,v_1)$ such that $v_0$ is adjacent to $v_1$, and the second-degree one is generated by anchor sequences of the form $(v_0,v_1,v_2)$ such that $v_0$ is adjacent to $v_2$, $v_0\neq v_2$, and of the form $(v_0,v_1,v_2)-(v_0,v_1',v_2)$, with possibly $v_0= v_2$.
			\end{proposition}
			
			
			Therefore, like path homologies, Proposition \ref{prop:3} asserts that hypergraph complexes effectively represent triangles and squares within hypergraphs. 
			
			As the anchor sequences influence hypergraph complexes, it is essential to note that the maximal hyperedges entirely determine the hypergraph complex of a hypergraph, which are the ones that are not part of any other hyperedges, for instance. Let $G$ be a hypergraph. Consider the graph $\overline{G}$ (see Figure \ref{fig:5}), called the essential graph of $G$, sharing the same vertex set with $G$, and has a set of edges satisfying the rule for every two distinct vertices $v_0$ and  $v_1$, $v_0$ is adjacent to $v_1$ in $\overline{G}$ if and only if there is a hyperedge $e$ in $G$ containing both. Notice then that $G$ and $\overline{G}$ have the same induced complex. Therefore, the rank of the first homology $H_0$ of $G$ over principal ideal domains can be computed using the formula of graph cases. Hence, $ H_0\cong R^{|C|}$, where $C$ is the set of the connected components of $\overline{G}$.
			
			
			\begin{figure}[t!]
				\centering
				\begin{subfigure}{.20\textwidth}
					\centering
					\includegraphics[width=.99\linewidth]{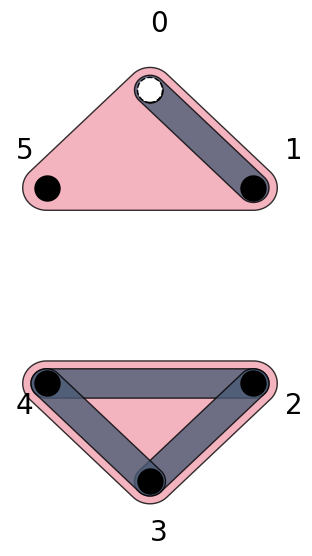}
					\caption{}
					\label{fig:5a}
				\end{subfigure}%
				\hspace{0.5cm}
				\begin{subfigure}{.20\textwidth}
					\centering
					\includegraphics[width=.99\linewidth]{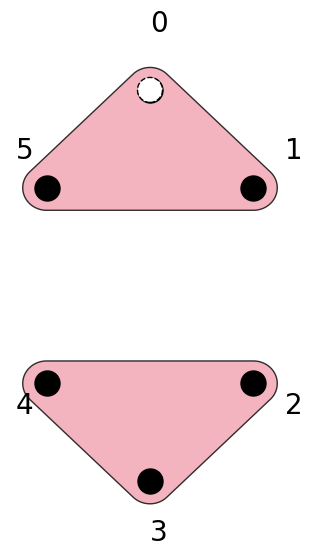}
					\caption{}
					\label{fig:5b}
				\end{subfigure}%
				\hspace{0.5cm}
				\begin{subfigure}{.20\textwidth}
					\centering
					\includegraphics[width=.99\linewidth]{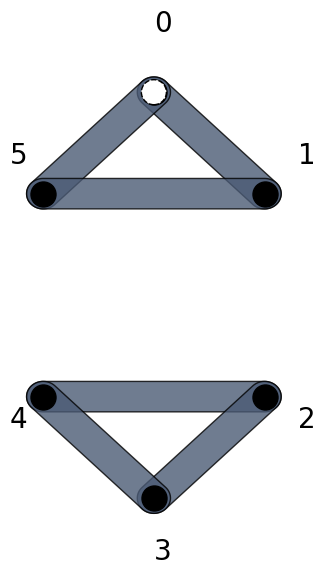}
					\caption{}
					\label{fig:5c}
				\end{subfigure}%
				\caption{Three hypergraphs inducing the same complex: {(A)} a hypergraph $G$, {(B)} a sub-hypergraph obtained by maximal hyperedges only, and {(C)} $G$ essential graph.}
				\label{fig:5}
			\end{figure}
			
			
			Next, we obtain a similar equation to $H_1 \cong R^{|E|-|V|+|C|}$ of graph homologies.
			
			
			\begin{proposition} \label{prop:4}
				Let $G$ be a hypergraph. and $\overline{G}$ be the graph obtained from $G$ as in the previous discussion. Let $R=k$ be a field. Then, 
				\begin{equation}
					H_1 \cong R^{2|E_{\overline{G}}|-|V_G|+|C_{\overline{G}}|-\rm{dim}( \rm{Imd}_2)},
				\end{equation}
				where $E_{\overline{G}}$ is the set of edges of $\overline{G}$, and $C_{\overline{G}}$ is the set of connected components. $d_2$ is the boundary map from the hypergraph complex component generated by paths of length three to the hypergraph complex component generated by the edges.
			\end{proposition}
			
			
			\begin{proof}
				First, notice that for an edge ${v_0, v_1}$ in $\overline{G}$, two anchors sequences are induced, namely, $e=(v_0, v_1)$ and $e'=(v_1,v_0)$. Obviously then, $e+e'\in \text{ker}(d_1)$. We say $e$ and $e'$ are dual to each other. Let $v$ be an element in the first component of the hypergraph complex, then $v=\sum c_e e+c_{e'}e'=\sum (c_e-c_{e'})e+c_{e'}(e+e')$, where the sum runs over all edges of $\overline{G}$. By correctly writing the edges for simplicial complexes and graph homologies, we notice that the subspace $W_1$ generated by all $e+e'$ is disjointed from the subspace $W_2$ generated by $e$. Moreover, $v \in \text{ker}(d_1)$ if and only if $\sum (c_e-c_{e'})e \in \text{ker}(d_1)$. Therefore, $$\text{ker}(d_1) = \text{ker}(d_1|_{W_1})\bigoplus \text{ker}(d_1|_{W_2}).$$ By graph homology theory, 
				\begin{equation} \label{eq:2}
					\text{dim}(\text{ker}(d_1|_{W_2})) = |E_{\overline{G}}|-|V_{\overline{G}}|+|C_{\overline{G}}|,
				\end{equation}
				and clearly, $\text{dim}(\text{ker}(d_1|_{W_1}))=|E_{\overline{G}}|$. Hence, we reach the desired conclusion since $V_G=V_{\overline{G}}$.
			\end{proof}
			
			In case $R$ is a principal ideal domain, then proposition \ref{prop:3} can be restated as
			\begin{equation}
				\text{rank}(H_1) = 2|E_{\overline{G}}|-|V_G|+|C_{\overline{G}}|-\text{rank}(\text{Im}d_2).
			\end{equation}
			
			We obtain several consequences from the previous proposition. First, the previous proposition provides an upper limit for the rank of the first-degree homology. That is,
			\begin{equation}
				\text{rank}(H_1) \leq  2|E_{\overline{G}}|-|V_G|+|C_{\overline{G}}|.
			\end{equation}
			

			
			\begin{example}\label{exp:5}
				Consider the hypergraph in Example \ref{exp:3}. The essential graph shown in Figure \ref{fig:5} has six vertices, six edges, and two connected components. Thus, 
				\begin{equation}
					2|E_{\overline{G}}|-|V_{\overline{G}}|+|C_{\overline{G}}|=2\cdot 6-6+2 = 8.
				\end{equation} 
				The rank of the boundary map $d_2$ is also eight; hence, the rank (dimension) of the homology $H_1$ is zero, which coincides with Example \ref{exp:3}.
				
			\end{example}
			
			
			\subsection{Hypergraph Dirac operators}
			
			Given a hypergraph $G=(V, E$), we can construct its hypergraph complex $(\Omega_\bullet, \partial_\bullet)$ embedded in $(C_\bullet, \partial_\bullet)$ generated by $V$ sequences as before. By equipping $(C_\bullet, \partial_\bullet)$ with an orthonormal basis consisting of these sequences and restricting the resulted inner product on $(\Omega_\bullet, \partial_\bullet)$, we can define Dirac operators $(D_p)_{p\geq 0}$ on $(\Omega_\bullet, \partial_\bullet)$ which we call hypergraph Dirac operators. Unlike path Dirac operators on digraphs, hypergraph Dirac operators generally have large sizes when $G$ is large with complicated connections. This is evident because a single triangle induces three generators in the second-degree components. Therefore, they are difficult to express in this article's examples. Accordingly, we represent hypergraph Dirac matrices by images instead of the traditional matrix representations.
			
			\begin{example} \label{exp:3}
				In this example, we demonstrate the cases of hypergraphs. Generally, large hypergraphs induce complexes with nontrivial components for all degrees. Therefore, only the zeroth degree Dirac operator, denoted as $D_0$, is visually depicted in Figures \ref{HypergraphExample1:1}. Consider the hypergraph shown in Figure \ref{fig:3}.
				\begin{figure}[t!]
					\centering
					\includegraphics[width=.6\linewidth]{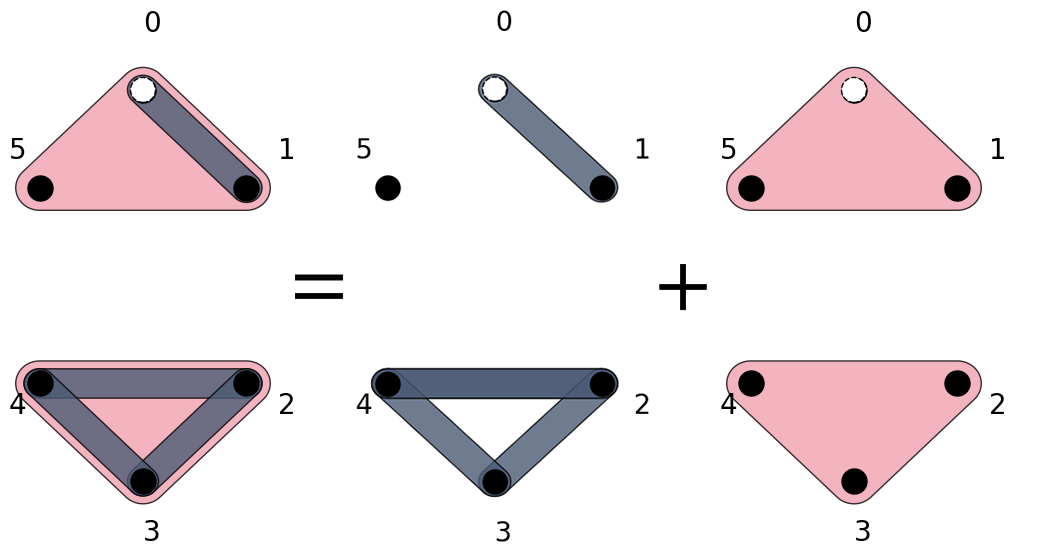}
					\caption{A hypergraph to demonstrate the Dirac operator's ability to capture features of hypergraphs.}
					\label{fig:3}
				\end{figure}
				We calculate the first three hypergraph complex components and the Dirac operators. Many anchor sequences exist within sequences of three or fewer lengths. The vertices span the first component, and all directed edges in the hypergraph span the second component. 
				Our algorithms indicate that the first four Betti numbers are as follows: $\beta_0=2$, $\beta_1=\beta_2=\beta_3=0$. Simultaneously, the nullities of the initial four Dirac operators are $\eta(D_0)=10$, $\eta(D_1)=12$, $\eta(D_2)=16$, and $\eta(D_3)=18$. Due to multiple edges, the dimensions of the hypergraph's complex components are significant. Therefore, the Laplacian and Dirac operator matrices are also quite large. Consequently, Table \ref{table:4} provides information on the eigenvalues and their multiplicities for the initial four hypergraph Dirac operators.
				
				
				\begin{table}[t] 
					\centering
					\begin{tabular}{ccc} 
						\hline
						n &   $\text{nullity}(D_n)$ & $\text{Spec}(D_n)$\\ [0.5ex] 
						\hline 
						& & \\
						0 &  10 & $0(\times 10)$, $\pm\sqrt{6}(\times 4)$,\\ [3.3ex] 
						1 & 12 & $0(\times 12)$, $\pm1(\times 2)$, $\pm\sqrt{6}(\times 4)$,\\[1.3ex] 
						&  & $\pm\sqrt{7}(\times 4)$, $\pm 3(\times 2)$\\[3.3ex] 
						2 &  16 & $0(\times 16)$, $\pm1(\times 2)$, $\pm2(\times 4)$, \\[1.3ex]
						&   &  $\pm\sqrt{6}(\times 4)$, $\pm\sqrt{7}(\times 4)$, $\pm 3(\times 2)$, \\[1.5ex]
						& &  $\pm 3.06(\times 6)$\\[3.3ex]
						3 & 18 & $0(\times 18)$, $\pm0.93(\times 2)$, $\pm1(\times 2)$,  \\[1.3ex]
						&   &  $\pm2(\times 4)$, $\pm\sqrt{6}(\times 4)$, $\pm\sqrt{7}(\times 10)$,  \\[1.5ex]
						& &  $\pm 2.89(\times 4)$, $\pm 3(\times 2)$, $\pm 3.06(\times 6)$,\\[1.5ex]
						& &  $\pm 3.24(\times 2)$\\[3.3ex]
						\hline
					\end{tabular}
					\caption{The spectrum of Dirac operators on the hypergraph in Figure \ref{fig:3}.}
					\label{table:4}
					
				\end{table}
				
				
				
				\begin{figure}[t]
					\centering
					\begin{subfigure}{.30\textwidth}
						\centering
						\includegraphics[width=.99\linewidth]{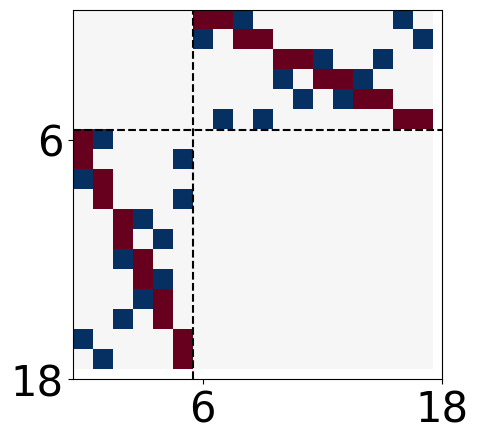}
						\caption{$D_0$}
						\label{HypergraphExample1:1}
					\end{subfigure}
					\caption{$D_0$ Dirac matrix visualized as an image, where blue = -1, white = 0, and red = 1.}
				\end{figure}
				
				
			\end{example}
			
			
			
			\section{Persistent path Dirac and persistent hypergraph Dirac}

			This section describes thoroughly the construction of persistent path Dirac; however, the same construction can be applied to hypergraphs. In topological data analysis, extracting information from data sets and point clouds across various scales may involve establishing a sequence of topological structures, resulting in a complex chain sequence. This procedure is commonly referred to as filtration. More formally, the filtration of a digraph can be defined in two ways:
			Firstly, it can be defined as a function $\phi$ that assigns a digraph for each real number in such a way that for every $n \leq m$, the vertices of $\phi(n)$ are a subset of the vertices of $\phi(m)$, and the edges of $\phi(n)$ are a subset of the edges in $\phi(m)$. That is, a filtration is a functor $\mathcal{P}:(\mathbb{R},\leq)\to \mathbf{Dig}$ from $(\mathbb{R},\leq)$ to the category of digraphs. For real numbers $a\leq b$, the $(a,b)$-persistent homology of  a persistent path homology is given by
			\begin{equation*}
				H_{p}^{a,b}(\mathcal{P};\mathbb{K})=\text{Im}(H_{p}(\mathcal{P}(a);\mathbb{R})\to H_{p}(\mathcal{P}(b);\mathbb{R})).
			\end{equation*}
			Alternatively, filtration can be seen as a sequence of digraphs. The same construction can be applied to hypergraphs. This filtration process leads to the establishment of embeddings of chain complexes. 
			
			\subsection{Persistent path/hypergraph   Dirac operators}
			
			Given an embedding of chain complexes $i_\bullet: (A_\bullet, d^A_\bullet) \longrightarrow (B_\bullet, d^B_\bullet)$, we compute an auxiliary complex $(C_\bullet, d^C_\bullet:=d^B_\bullet|_{C_\bullet})$ defined by
			\begin{equation}
				C_n = (d^B_n)^{-1}(i_{n-1}(A_{n-1}))=\{v\in B_n|d_n^B(v)\in i_{n-1}(A_{n-1})\}
			\end{equation}
			for $n\in \mathbb{N}_0$. 
			\begin{equation}
				\xymatrix@C=3em{
					\cdots \ar[r] &A_{n}\ar[r]^{d^{A}_{n}} \ar[d] & A_{n-1} \ar[r]^{d^{A}_{n-1}}\ar[d]& \cdots \ar[r]^{d^{A}_{2}}\ar[d]& A_{1} \ar[r]^{d^{A}_{1}}\ar[d]& A_0 \ar[r]\ar[d]& 0\\
					\cdots \ar[r]&  C_{n} \ar[r]^{d^{C}_{n}}\ar[ur]^{d^{B,A}_{n}}\ar@{_{(}.>}[d] & C_{n-1}\ar[r]^{d^{C}_{n-1}}\ar[ur]^{d^{B,A}_{n-1}}\ar@{_{(}.>}[d] & \cdots \ar[r]^{d^{C}_{2}} \ar[ur]^{d^{B,A}_{2}}\ar@{_{(}.>}[d] & C_1\ar[r]^{d^{C}_{1}} \ar[ur]^{d^{B,A}_{1}}\ar@{_{(}.>}[d] & C_0 \ar@{_{(}.>}[d] \ar[r] & 0 \\
					\cdots \ar[r] & B_{n}\ar[r]^{d^{B}_{n}} & B_{n-1} \ar[r]^{d^{B}_{n-1}}& \cdots \ar[r]^{d^{B}_{2}}& B_{1} \ar[r]^{d^{B}_{1}}& B_0 \ar[r]& 0 	
				}
			\end{equation}
			
			The persistent path $n$-th Laplacian operator \cite{wang2022persistent} is defined to be
			\begin{equation}
				\Delta_n^{B,A} = (d^A_n)^*d^A_n+d^{B,A}_{n+1}(d^{B,A}_{n+1})^*,
			\end{equation}
			and the persistent path $n$-th Dirac operator $D^{B,A}_n $ is defined as the $n$-th Dirac operator of the auxiliary complex shown above. Notice that the persistent path Laplacian operator $\Delta_n^{B,A}$ corresponds to the $n$-th Laplacian operator of the chain complex
			\begin{equation}
				\xymatrix@C=3em{
					\cdots \ar[r] & C_{n+1}	\ar[r]^{d^{B,A}_{n+1}} &A_{n}\ar[r]^{d^{A}_{n}}  & A_{n-1} \ar[r]^{d^{A}_{n-1}}& \cdots \ar[r]^{d^{A}_{1}}& A_0 \ar[r]& 0.
				}
			\end{equation}
			
			Since $i_\bullet$ is an embedding, then $A_n\subseteq C_n$, and thus over principal ideal domains, rank$(A_n)\leq \text{ rank}(C_n)$, for all $n\in \mathbb{N}_0$. Moreover, $ \text{Im}(d_{n+1}^{A})\subseteq \text{Im}(d_{n+1}^{B,A})\subseteq \text{ker}(d^A_n)$, and hence 
			\begin{equation}
				H_n(A_\bullet)=\text{ker}(d^A_n)/\text{Im}(d_{n+1}^{A}) \longrightarrow  \text{ker}(d^A_n)/\text{Im}(d_{n+1}^{B,A})\longrightarrow 0
			\end{equation}
			and consequently, $\eta(\Delta^A_n)\geq  \eta(\Delta^{B,A}_n)$. We also observe that $\text{ker}(d^A_n)\subseteq \text{ker}(d^C) \subseteq \text{ker}(d^B_n)$, and $\text{Im}(d_{n+1}^{C}) = \text{Im}(d_{n+1}^{B,A})$ . Therefore, 
			\begin{equation}
				0 \longrightarrow \text{ker}(d^A_n)/\text{Im}(d_{n+1}^{B,A}) \longrightarrow H_n(C_\bullet)
			\end{equation}
			and subsequently, $\eta(\Delta^C_n)\geq  \eta(\Delta^{B,A}_n)$.
			
			This section explores the practical applications and examples of the filtration technique. We also demonstrate the features of the persistent path Dirac and persistent hypergraph Dirac operators captured.
			Consider filtration denoted as $\phi$ and real numbers or integers $n \leq m$. Let $H^n$ and $H^m$ represent the corresponding hypergraphs or digraphs under the filtration $\phi$, and $H^n_p$ and $H^m_p$ be the $p$-th degree components of their associated chain complexes. We compute the first-degree persistent path Dirac operators, denoted as $D_1^{(n,m)}$, and analyze selected characteristic features for demonstration purposes. According to the definitions of the operators and the diagram below,
			\begin{equation}
				\xymatrix@C=3em{
					\cdots \ar[r]^{d^{H^n}_{2}}\ar[d]& H^n_{1} \ar[r]^{d^{H^n}_{1}}\ar[d]& H^n_0 \ar[r]\ar[d]& 0\\
					\cdots \ar[r]^{d^{C}_{2}} \ar[ur]^{d^{H^m,H^n}_{2}}\ar@{_{(}.>}[d] & C_1\ar[r]^{d^{C}_{1}} \ar[ur]^{d^{H^m,H^n}_{1}}\ar@{_{(}.>}[d] & C_0 \ar@{_{(}.>}[d] \ar[r] & 0 \\
					\cdots \ar[r]^{d^{H^m}_{2}}& H^m_{1} \ar[r]^{d^{H^m}_{1}}& H^m_0 \ar[r]& 0 	
				}
			\end{equation}
			we can write
			\begin{equation}
				\label{eqn:dirac}
				\renewcommand{\arraystretch}{1.5}
				D_1^{(n,m)} = 
				\begin{bmatrix}
					\textbf{0}_{n_0\times n_0}    & \mathbf{d^{H^m}_1}    & \textbf{0}_{n_0\times n_2}       \\
					
					(\mathbf{d^{H^m}_1})^*  & \textbf{0}_{n_1\times n_1}      & \mathbf{d_2^C}      \\
					
					\textbf{0}_{n_2\times n_0}      & (\mathbf{d_2^C})^*  & \textbf{0}_{n_2\times n_2}          
				\end{bmatrix}.
			\end{equation}
			
			Equation \ref{eq:4} states that
			\begin{align}
				\eta(D^{(n,m)}_1)&= \beta^{(n,m)}_0+\beta^{(n,m)}_1 +\eta((\mathbf{d_2^C})^* \mathbf{d_2^C} )\\
				&=\beta^{(n,m)}_0+\beta^{(n,m)}_1 + \eta( \mathbf{d_2^C} ), \notag
			\end{align}
			where $\beta^{(n,m)}_0$ and $\beta^{(n,m)}_1$ are the first two Betti numbers of the auxiliary complex $C_\bullet$. In various situations, $H^n_2=H^m_2=C_2=0$, and hence
			\begin{equation}
				\label{eqn:dirac}
				\renewcommand{\arraystretch}{1.5}
				D_1^{(n,m)} = 
				\begin{bmatrix}
					\textbf{0}_{n_0\times n_0}    & \mathbf{d^{H_1^m}_1}    & \textbf{0}_{n_0\times n_2}       \\
					
					(\mathbf{d^{H_1^m}_1})^*  & \textbf{0}_{n_1\times n_1}      & \textbf{0}_{n_1\times n_2}        \\
					
					\textbf{0}_{n_2\times n_0}      & \textbf{0}_{n_2\times n_1}     & \textbf{0}_{n_2\times n_2}          
				\end{bmatrix}.
			\end{equation}
			
			$\beta^{(n,m)}_0$ is the same as the zeroth Betti number of $H^m_\bullet$, and when $C_2=0$, then $\beta^{(n,m)}_1$ is the same as the first Betti number of $H^m_\bullet$. That is, $D_1^{(n,m)}$ might have the same features for many values of $n$ for a fixed $m$. We consider three features, namely, the nullity of the operators, the mean of their positive eigenvalues defined as
			\begin{equation}
				\overline{\lambda} = \frac{1}{n}\sum_{1}^{n}\lambda_i
			\end{equation}
			that runs over all the positive eigenvalues, where $n$ is the number of the positive eigenvalues counting multiplicities; otherwise, it assigns zero. Another feature of interest is the generalized mean of the positive eigenvalues defined as 
			\begin{equation}
				\hat{\lambda} = \frac{1}{n}\sum_{1}^{n}|\lambda_i-\overline{\lambda}|
			\end{equation}
			that runs over all positive eigenvalues and assigns otherwise zero. This is useful, especially when the eigenvalues are not equal to their mean. The generalized mean is utilized in the context of mean absolute deviation to remove the direct reliance of the aggregate absolute deviation of each eigenvalue from the mean eigenvalue on the size of the molecule, as explained in \cite{consonni2008new}, and to assist in the development of quantitative structural models for a diverse range of molecules \cite{alhevaz2020generalized, nguyen2019agl}. For a linear operator $T$, we will denote its mean by $\overline{\lambda}(T)$, and its generalized mean by $\hat{\lambda}(T)$. Other quantities such as the minimum of the nonzero positive eigenvalues, the maximum among them, the sum of positive ones, and the standard deviation of the positive ones are used in the literature as well \cite{wang2022persistent, chen2023persistent,wee2023persistent}.
			
			In the following examples, we calculate the features discussed at the beginning of this section captured by persistent path Dirac operators to demonstrate their capabilities.
			
			\begin{figure}[t!]
				\centering
				\begin{subfigure}{.20\textwidth}
					\centering
					\includegraphics[width=.99\linewidth]{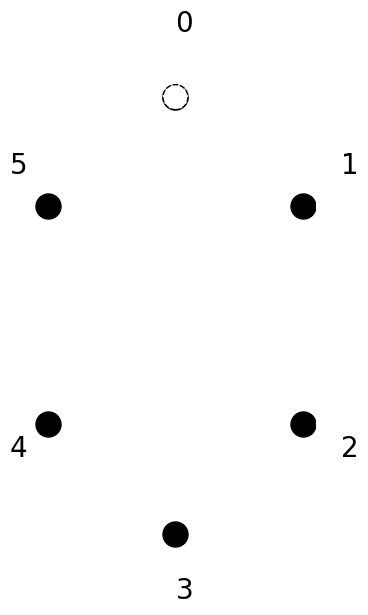}
					\caption{$H^1$}
				\end{subfigure}%
				\begin{subfigure}{.20\textwidth}
					\centering
					\includegraphics[width=.99\linewidth]{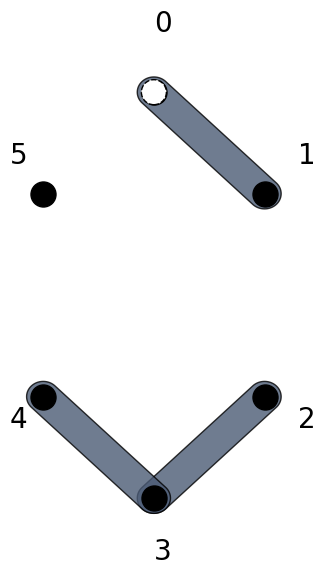}
					\caption{$H^2$}
				\end{subfigure}%
				\begin{subfigure}{.20\textwidth}
					\centering
					\includegraphics[width=.99\linewidth]{H3}
					\caption{$H^3$}
				\end{subfigure}%
				\begin{subfigure}{.20\textwidth}
					\centering
					\includegraphics[width=.99\linewidth]{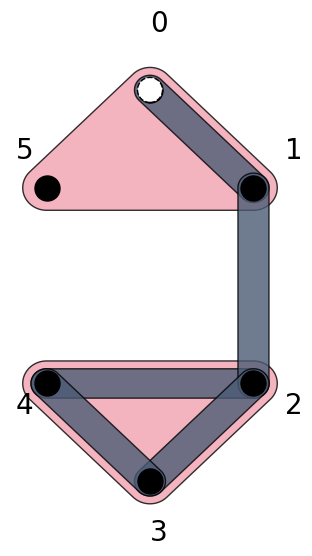}
					\caption{$H^4$}
				\end{subfigure}%
				\begin{subfigure}{.20\textwidth}
					\centering
					\includegraphics[width=.99\linewidth]{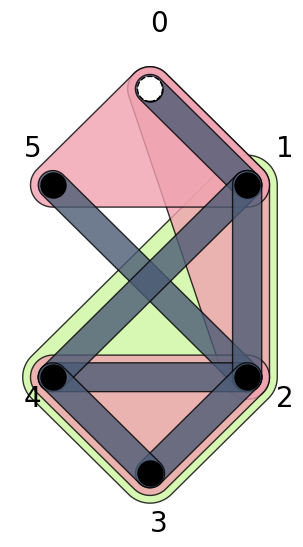}
					\caption{$H^5$}
				\end{subfigure}%
				\caption{Filtration of $H^5$.}
				\label{hypergraph_filtration:1}
			\end{figure}
			
			
			\begin{example}\label{exp:9}
				This example considers a filtration of a hypergraph shown in Figure \ref{hypergraph_filtration:1}. The more incident edges a hypergraph has, the larger the dimensions of the higher-degree components of the hypergraph complex. This is evident by Proposition \ref{prop:3}, since the hypergraph will have many generators of the form in Figure \ref{fig:1e}. Therefore, it is difficult to discuss and verify values. However, we use our algorithms to compute the first-degree persistent hypergraph Dirac operators $D_1^{(n,m)}$, and compute their nulities $\eta(D_1^{(n,m)})$, their means $\overline{\lambda}(D_1^{(n,m)})$, and generalized means $\hat{\lambda}(D_1^{(n,m)})$. Their values are shown in Figure \ref{HypergraphExample2}.

				
				\begin{figure}[t!]
					\centering
					\begin{subfigure}{.33\textwidth}
						\centering
						\includegraphics[width=.99\linewidth]{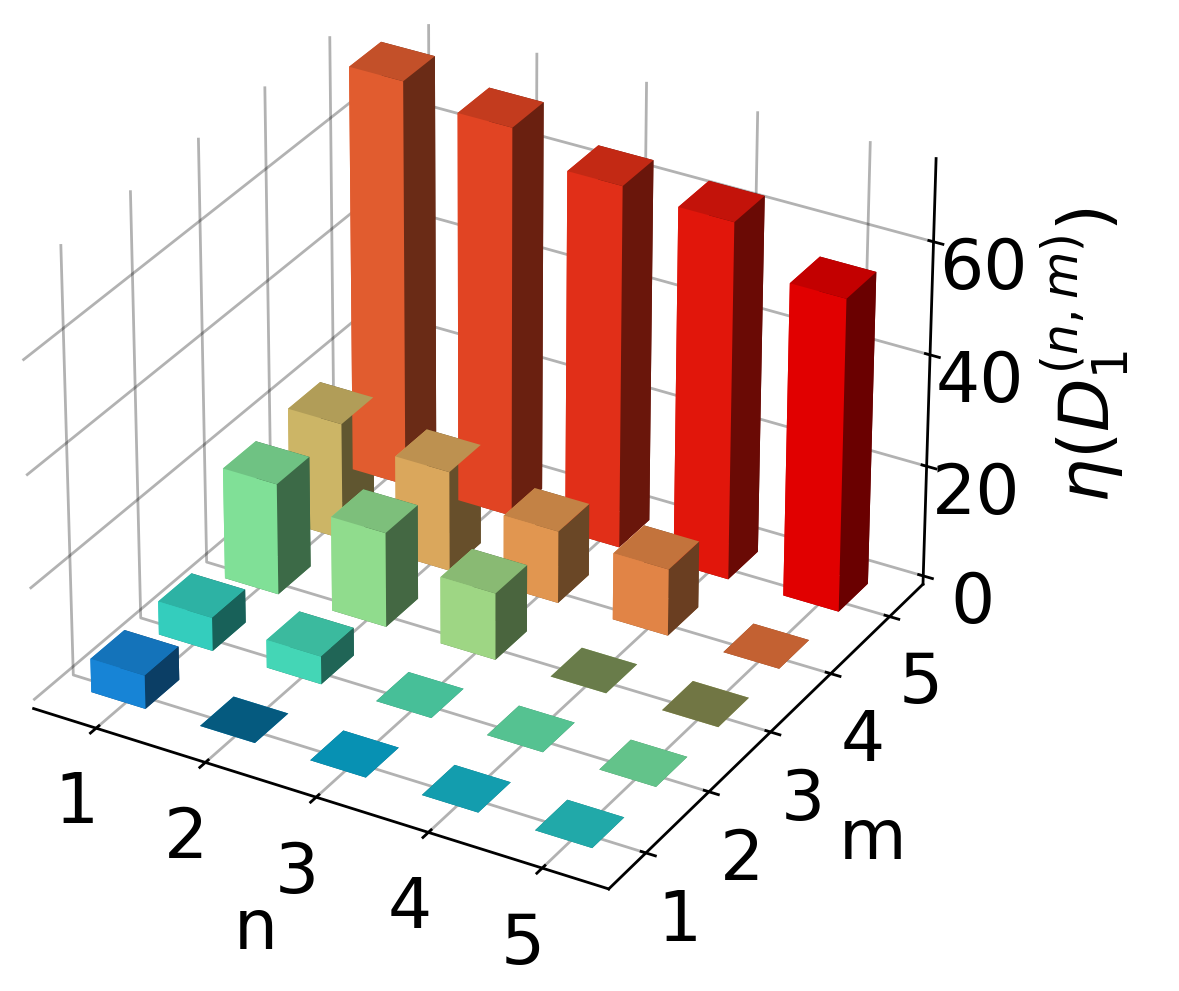}
						\caption{}
						\label{HypergraphExample2:1}
					\end{subfigure}%
					\begin{subfigure}{.33\textwidth}
						\centering
						\includegraphics[width=.99\linewidth]{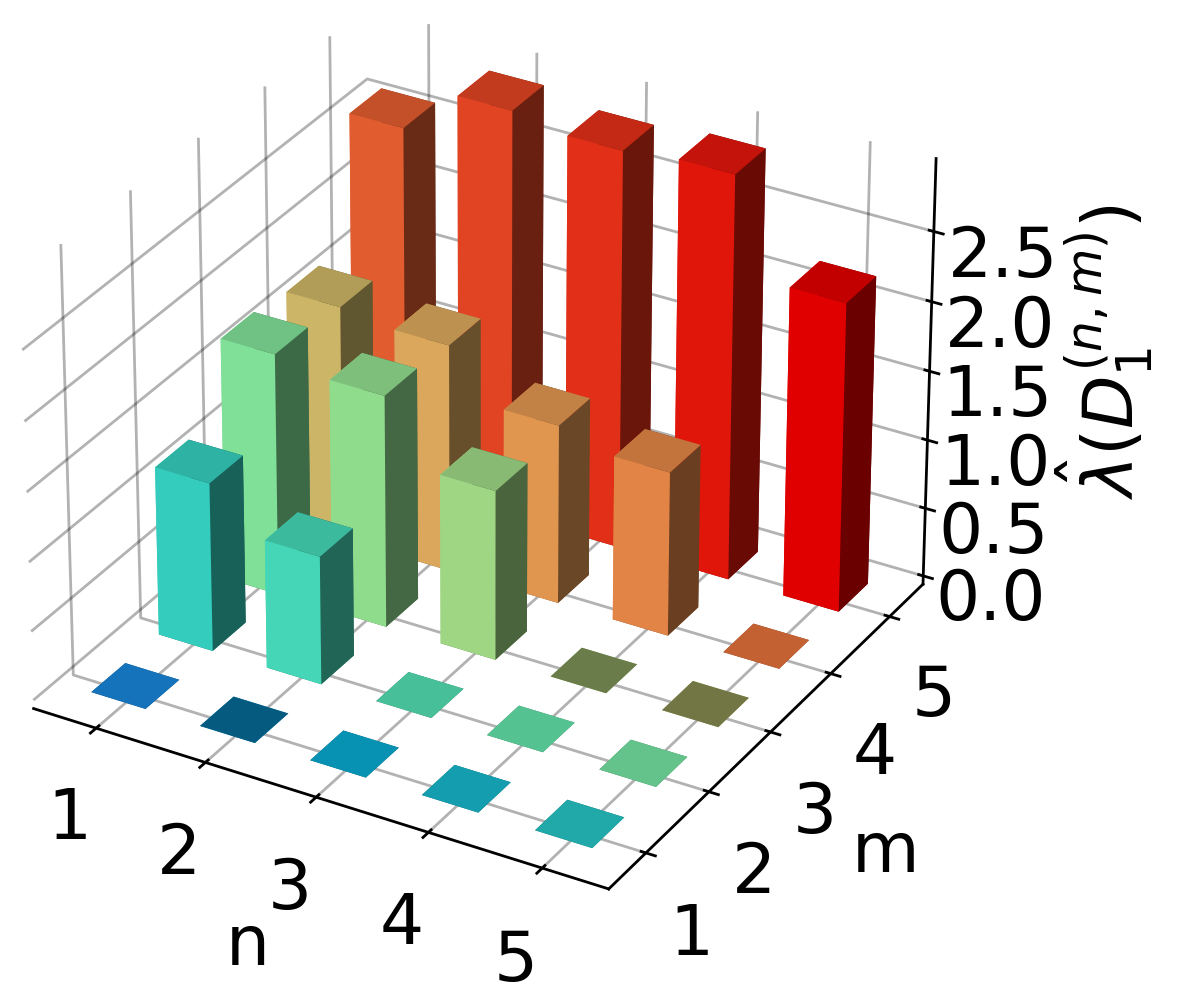}
						\caption{}
						\label{HypergraphExample2:2}
					\end{subfigure}%
					\begin{subfigure}{.33\textwidth}
						\centering
						\includegraphics[width=.99\linewidth]{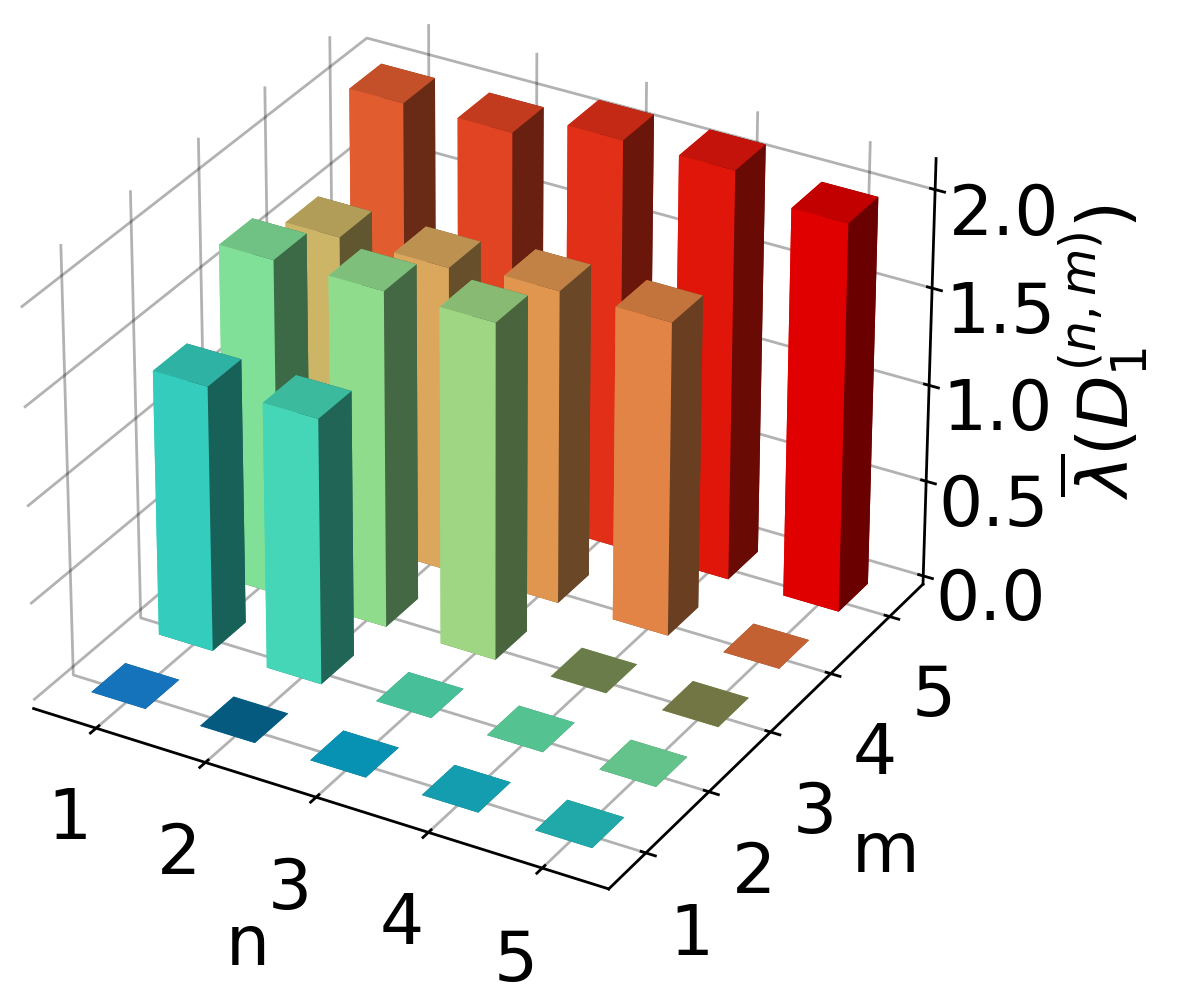}
						\caption{}
						\label{HypergraphExample2:3}
					\end{subfigure}
					
					\begin{subfigure}{.33\textwidth}
						\centering
						\includegraphics[width=.99\linewidth]{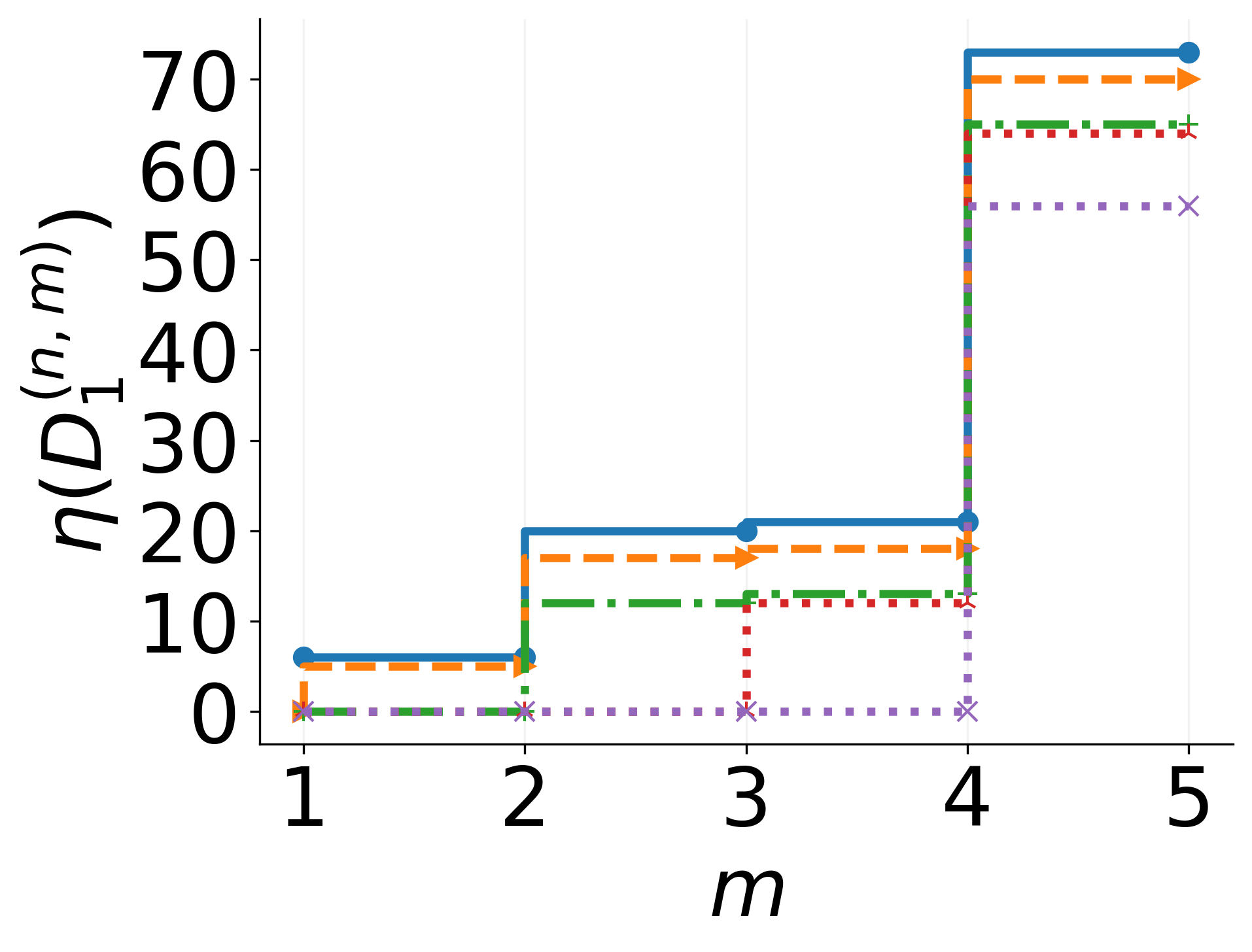}
						\caption{}
						\label{HypergraphExample2:4}
					\end{subfigure}%
					\begin{subfigure}{.33\textwidth}
						\centering
						\includegraphics[width=.99\linewidth]{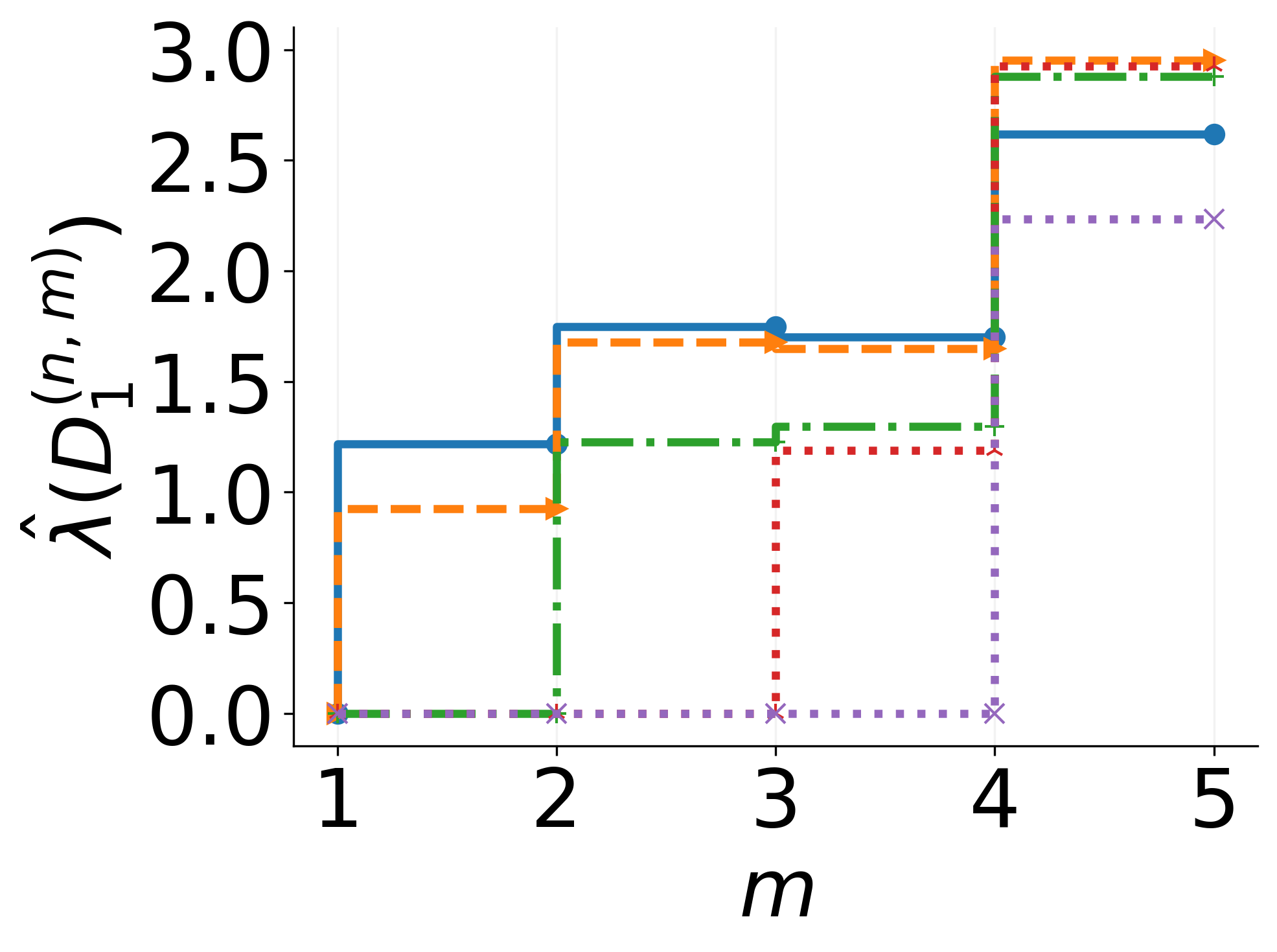}
						\caption{}
						\label{HypergraphExample2:5}
					\end{subfigure}%
					\begin{subfigure}{.33\textwidth}
						\centering
						\includegraphics[width=.99\linewidth]{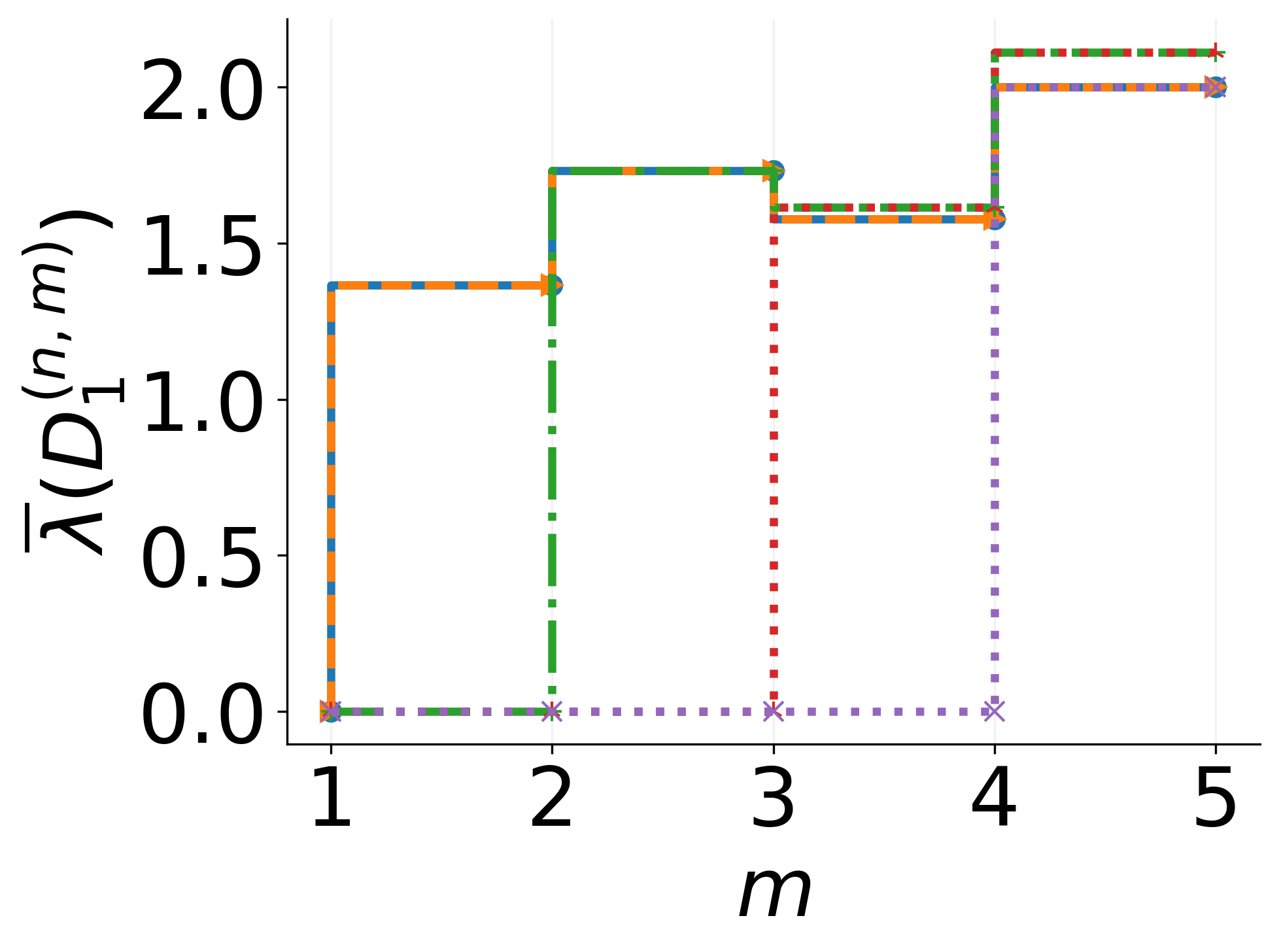}
						\caption{}
						\label{HypergraphExample2:6}
					\end{subfigure}
					
					\begin{subfigure}{.70\textwidth}
						\centering
						\includegraphics[width=.99\linewidth]{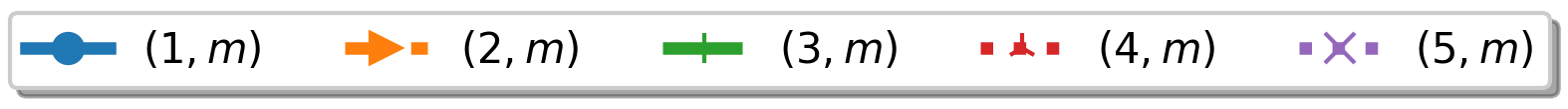}
					\end{subfigure}		
					
					\caption{Captured features by persistent hypergraph Dirac operators.  {(A)}, {(B)} and {(C)} show $\eta(D_1^{(n,m)\\
						})$, $\hat{\lambda}(D_1^{(n,m)})$ and $\overline{\lambda}(D_1^{(n,m)})$ of the filtration in Figure \ref{hypergraph_filtration:1} respectively. {(D)}, {(E)}, and {(F)} show their projections on the $mz$-plane.}
					\label{HypergraphExample2}
				\end{figure}
				
				
			\end{example}
			
			
			
			\begin{example} \label{exp:7}
				In this example, we showcase how persistent path Dirac operators exhibit sensitivity to different filtrations. We examine two distinct filtrations of a directed graph, depicted in Figure \ref{digraph_filtration:1}. Figure \ref{digraph_filtration_features:1} displays distinctive characteristics captured by the persistent path Dirac operators for both filtrations, including the nullities of the operators, generalized mean, and mean. It becomes evident that these different filtrations result in significantly varied captured features, underscoring the sensitivity of persistent path Dirac operators. Further clarification of the values in Figure \ref{digraph_filtration_features:2a} will be provided. There are only the vertices for $(n,m)=(1,1)$. 
				
				\begin{figure}[t!]
					\centering
					\begin{subfigure}{.17\textwidth}
						\centering
						\includegraphics[width=.99\linewidth]{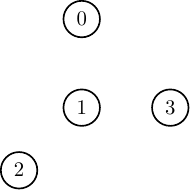}
						\caption{}
					\end{subfigure}%
					\hspace{0.2cm}
					\begin{subfigure}{.17\textwidth}
						\centering
						\includegraphics[width=.99\linewidth]{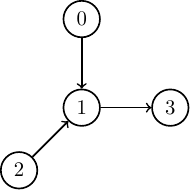}
						\caption{}
					\end{subfigure}%
					\hspace{0.2cm}
					\begin{subfigure}{.17\textwidth}
						\centering
						\includegraphics[width=.99\linewidth]{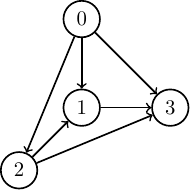}
						\caption{} 
					\end{subfigure}%
					\hspace{0.2cm}
					\begin{subfigure}{.17\textwidth}
						\centering
						\includegraphics[width=.99\linewidth]{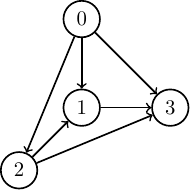}
						\caption{} 
					\end{subfigure}%
					\hspace{0.2cm}
					\begin{subfigure}{.17\textwidth}
						\centering
						\includegraphics[width=.99\linewidth]{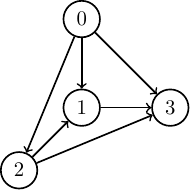}
						\caption{} 
					\end{subfigure}%
					
					\begin{subfigure}{.17\textwidth}
						\centering
						\includegraphics[width=.99\linewidth]{graph9.pdf}
						\caption{}
					\end{subfigure}%
					\hspace{0.2cm}
					\begin{subfigure}{.17\textwidth}
						\centering
						\includegraphics[width=.99\linewidth]{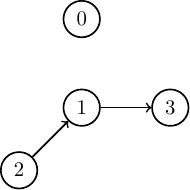}
						\caption{}
					\end{subfigure}%
					\hspace{0.2cm}
					\begin{subfigure}{.17\textwidth}
						\centering
						\includegraphics[width=.99\linewidth]{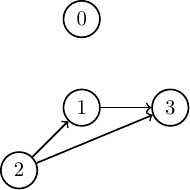}
						\caption{}
					\end{subfigure}%
					\hspace{0.2cm}
					\begin{subfigure}{.17\textwidth}
						\centering
						\includegraphics[width=.99\linewidth]{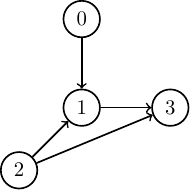}
						\caption{}
					\end{subfigure}%
					\hspace{0.2cm}
					\begin{subfigure}{.17\textwidth}
						\centering
						\includegraphics[width=.99\linewidth]{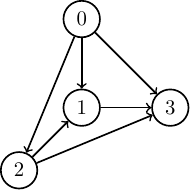}
						\caption{}
					\end{subfigure}%
					\caption{Two filtrations of the same digraph to demonstrate the sensitivity of persistent path Dirac operators to filtrations. { (A)}-{(E)} is the first filtration, and {(F)}-{(J)} is the second one.}
					\label{digraph_filtration:1}
				\end{figure}
				
				
				As a result, all the operators yield a value of zero, and consequently, we have $\beta^{(1,1)}_0=4$, $\beta^{(1,1)}_1=0$, leading to $\eta(D_1^{(1,1)})=4+0+0=4$.
				Moving forward, when two edges are added, it is observed that their images in the space formed by the vertices are linearly independent. Consequently, we have $\beta^{(1,2)}_0=4-2=2$, $\beta^{(1,2)}_1=0$, resulting in a second-degree component of the persistent complex with zero dimension. Therefore, $\eta(D_1^{(1,2)})=2+0+0=2$.
				Now, when $(n,m)=(1,3)$, a triangle, essentially a cycle, emerges when considering the induced hypergraph. According to Proposition \ref{prop:5} and Equation \ref{eq:3}, the second-degree component of the persistent complex is generated by this triangle. Consequently, we have $\beta^{(1,3)}_0=2$, $\beta^{(1,3)}_1=1$, and $\eta(\mathbf{d_2^C})=0$. Following equation \ref{eq:2}, this leads to $\eta(D_1^{(1,3)})=2+1+0=3$.
				Three adjacent triangles are formed when all edges are added, resulting in $\eta(\mathbf{d_2^C})=1$. The images of these three triangles are linearly independent, and with six edges in total, we have $\beta^{(1,5)}_1=6-3=3$. Consequently, $\eta(D_1^{(1,5)})=1+3+1=5$.
				
				\begin{figure}[t!]
					\centering
					\begin{subfigure}{.32\textwidth}
						\centering
						\includegraphics[width=.99\linewidth]{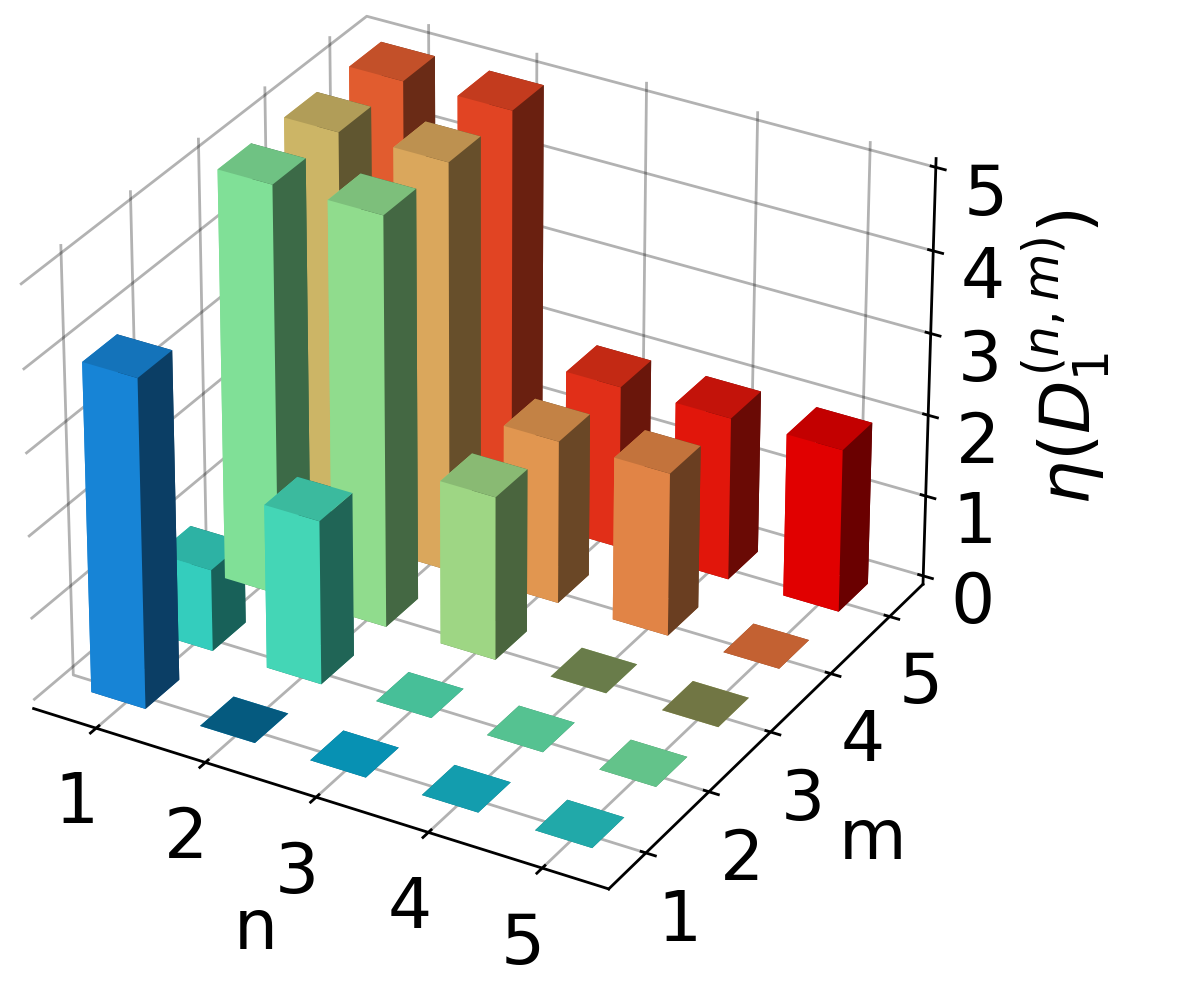}
						\caption{}
						\label{digraph_filtration_features:1a}
					\end{subfigure}%
					\hspace{0.1cm}
					\begin{subfigure}{.32\textwidth}
						\centering
						\includegraphics[width=.99\linewidth]{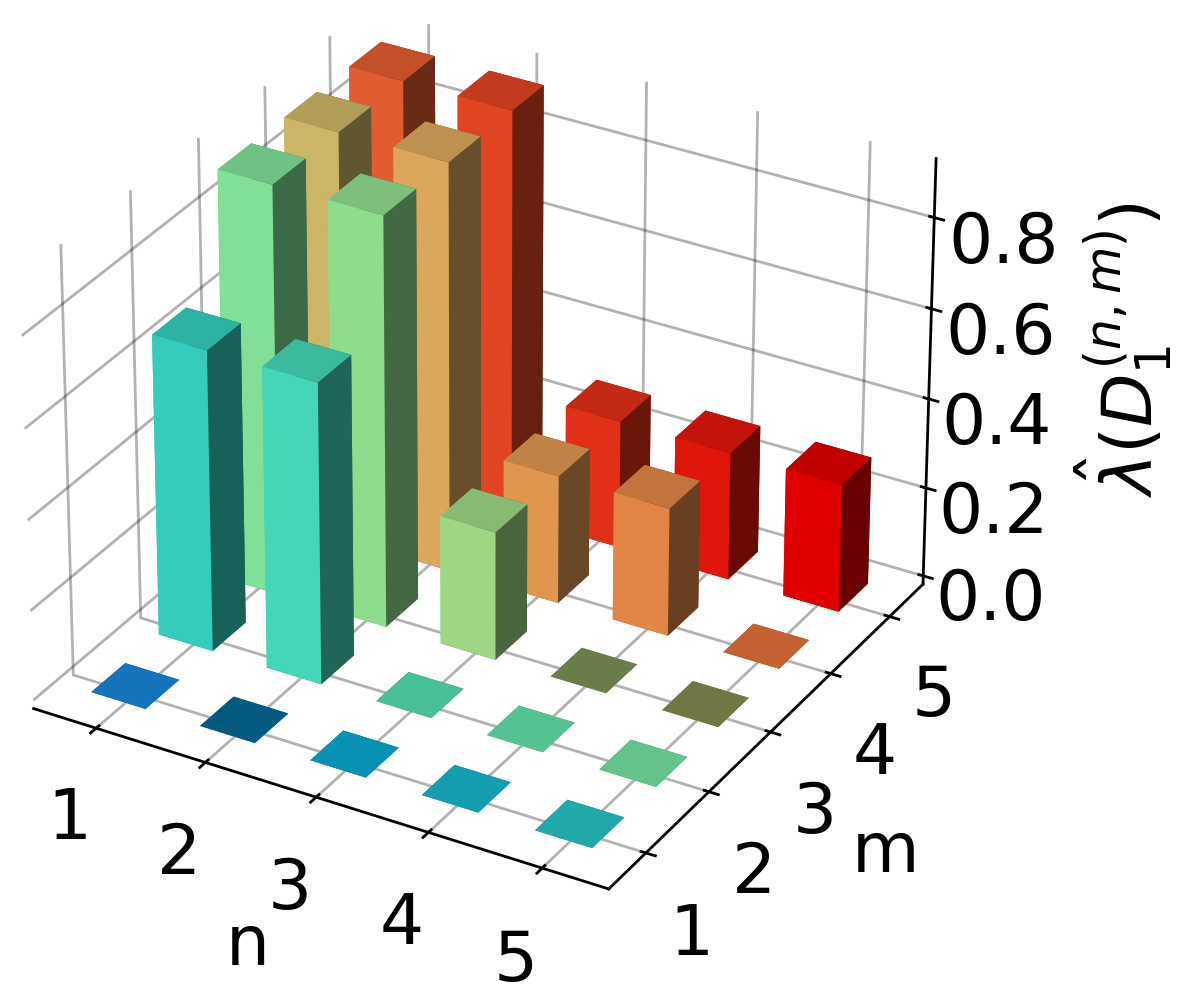}
						\caption{}
						\label{digraph_filtration_features:1b}
					\end{subfigure}%
					\hspace{0.1cm}
					\begin{subfigure}{.32\textwidth}
						\centering
						\includegraphics[width=.99\linewidth]{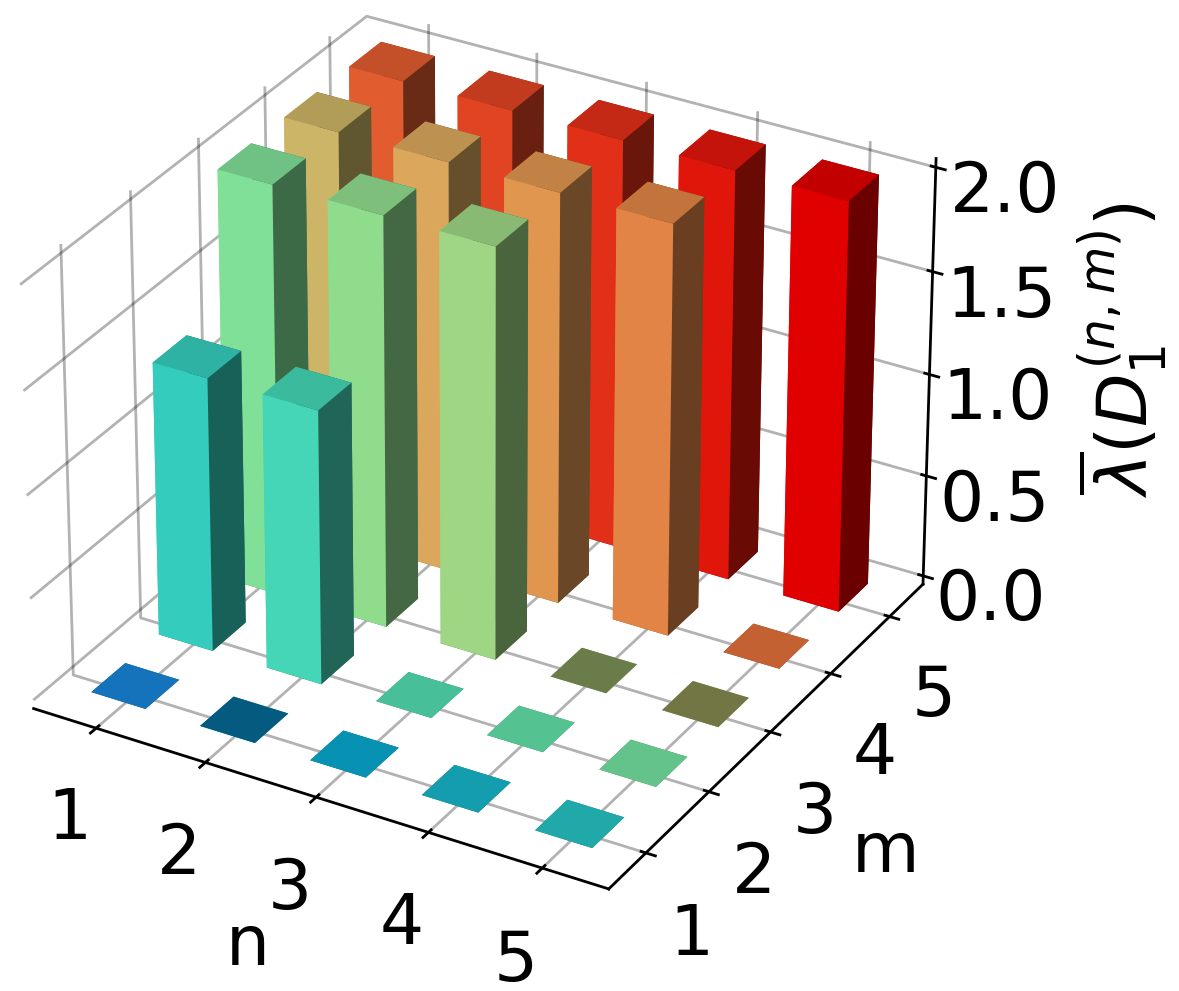}
						\caption{}
						\label{digraph_filtration_features:1c}
					\end{subfigure}
					
					\begin{subfigure}{.32\textwidth}
						\centering
						\includegraphics[width=.99\linewidth]{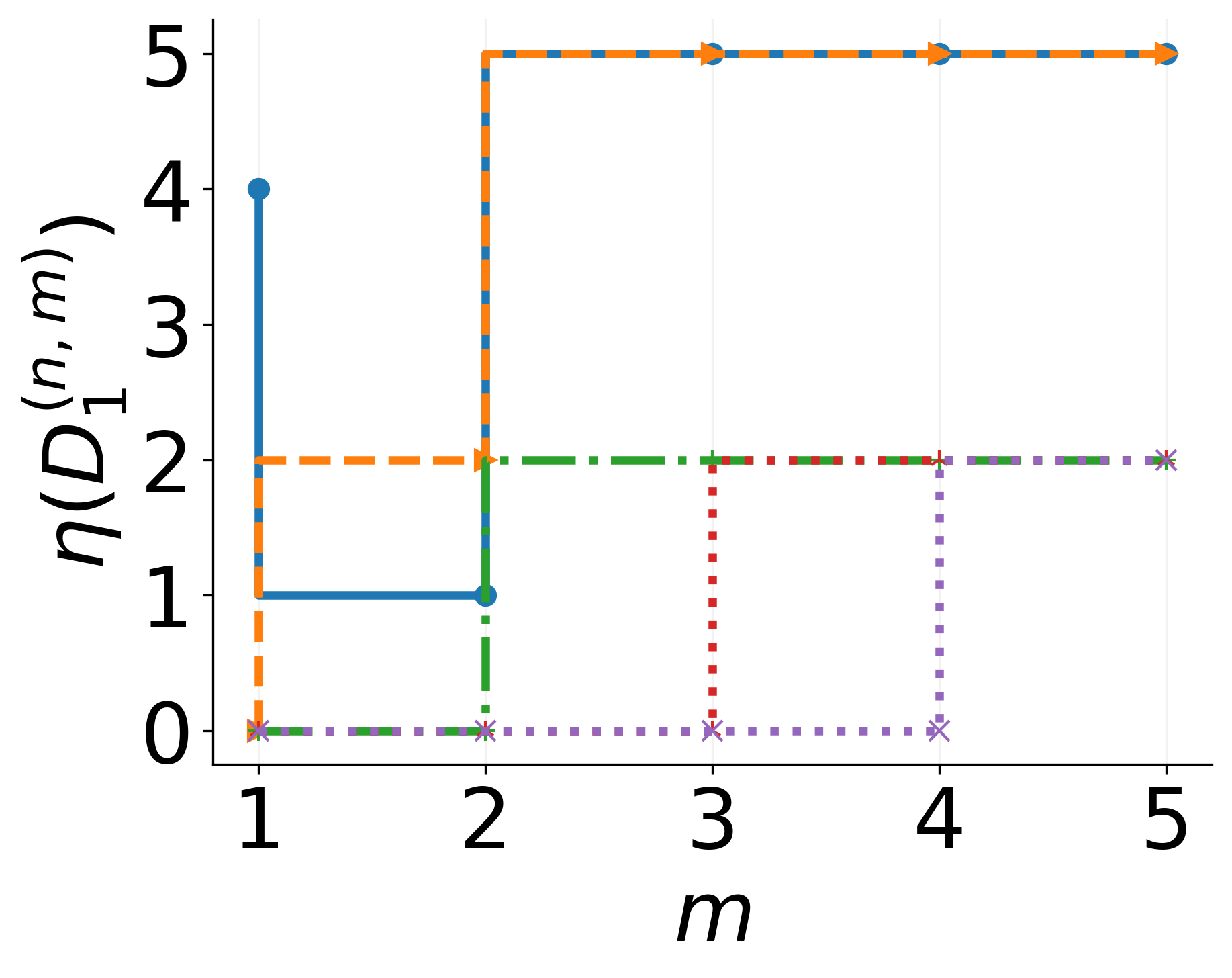}
						\caption{}
						\label{digraph_filtration_features:1d}
					\end{subfigure}%
					\hspace{0.1cm}
					\begin{subfigure}{.32\textwidth}
						\centering
						\includegraphics[width=.99\linewidth]{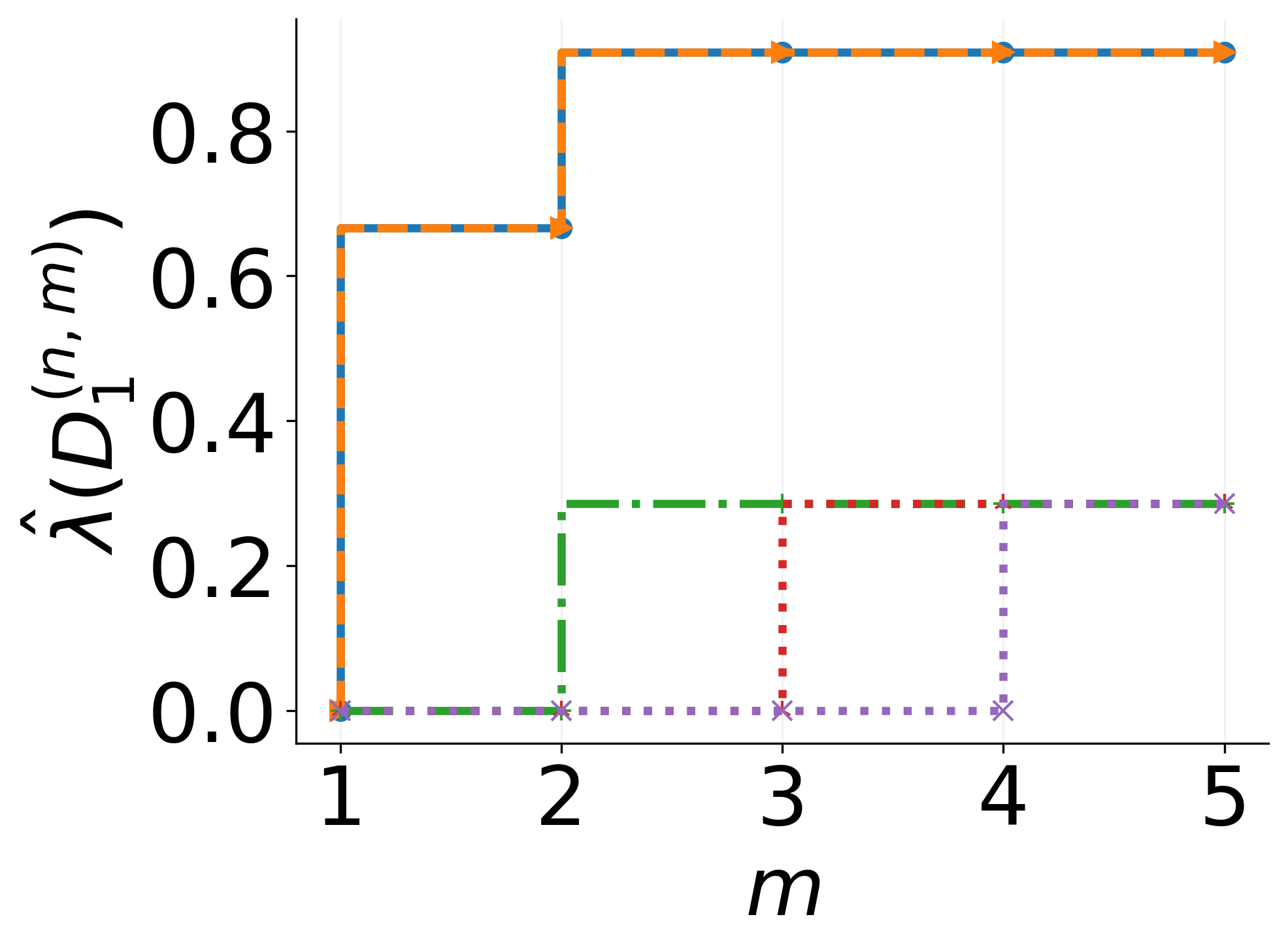}
						\caption{}
						\label{digraph_filtration_features:1e}
					\end{subfigure}%
					\hspace{0.1cm}
					\begin{subfigure}{.32\textwidth}
						\centering
						\includegraphics[width=.99\linewidth]{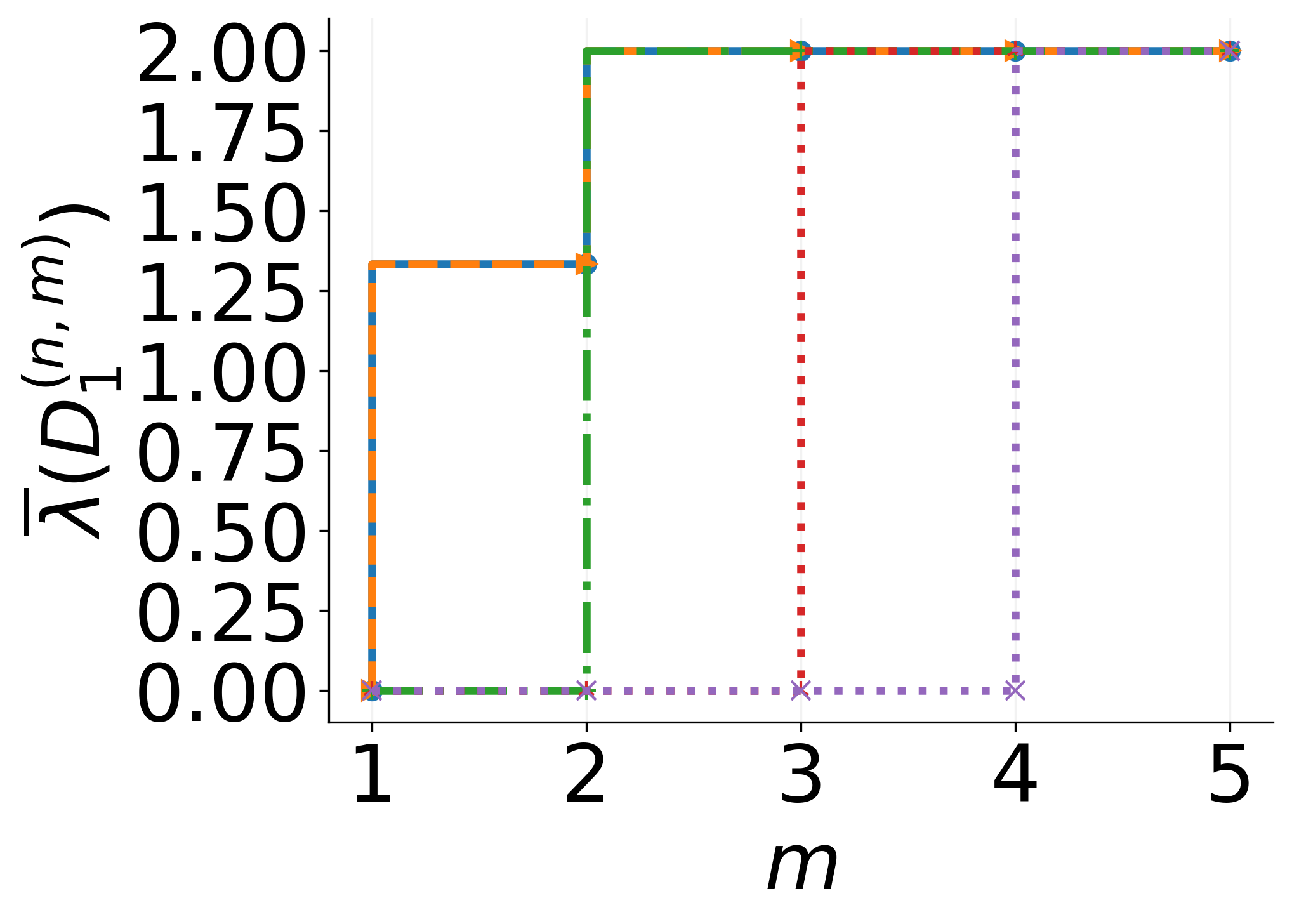}
						\caption{}
						\label{digraph_filtration_features:1f}
					\end{subfigure}
					
					\begin{subfigure}{.70\textwidth}
						\centering
						\includegraphics[width=.99\linewidth]{Legends}
					\end{subfigure}
					
					\caption{Captured features by persistent path Dirac operators induced by the filtrations in Figure \ref{digraph_filtration:1}. {(A)}, {(B)} and {
							(C)} show $\eta(D_1^{(n,m)\\
						})$, $\hat{\lambda}(D_1^{(n,m)})$ and $\overline{\lambda}(D_1^{(n,m)})$ of the first filtration in Figure \ref{digraph_filtration:1} respectively. {(D)}, {(E)}, and {(F)} show their projections on the $mz$-plane.}
					\label{digraph_filtration_features:1}
				\end{figure}

				\begin{figure}[t!]
					\centering
					\begin{subfigure}{.32\textwidth}
						\centering
						\includegraphics[width=.99\linewidth]{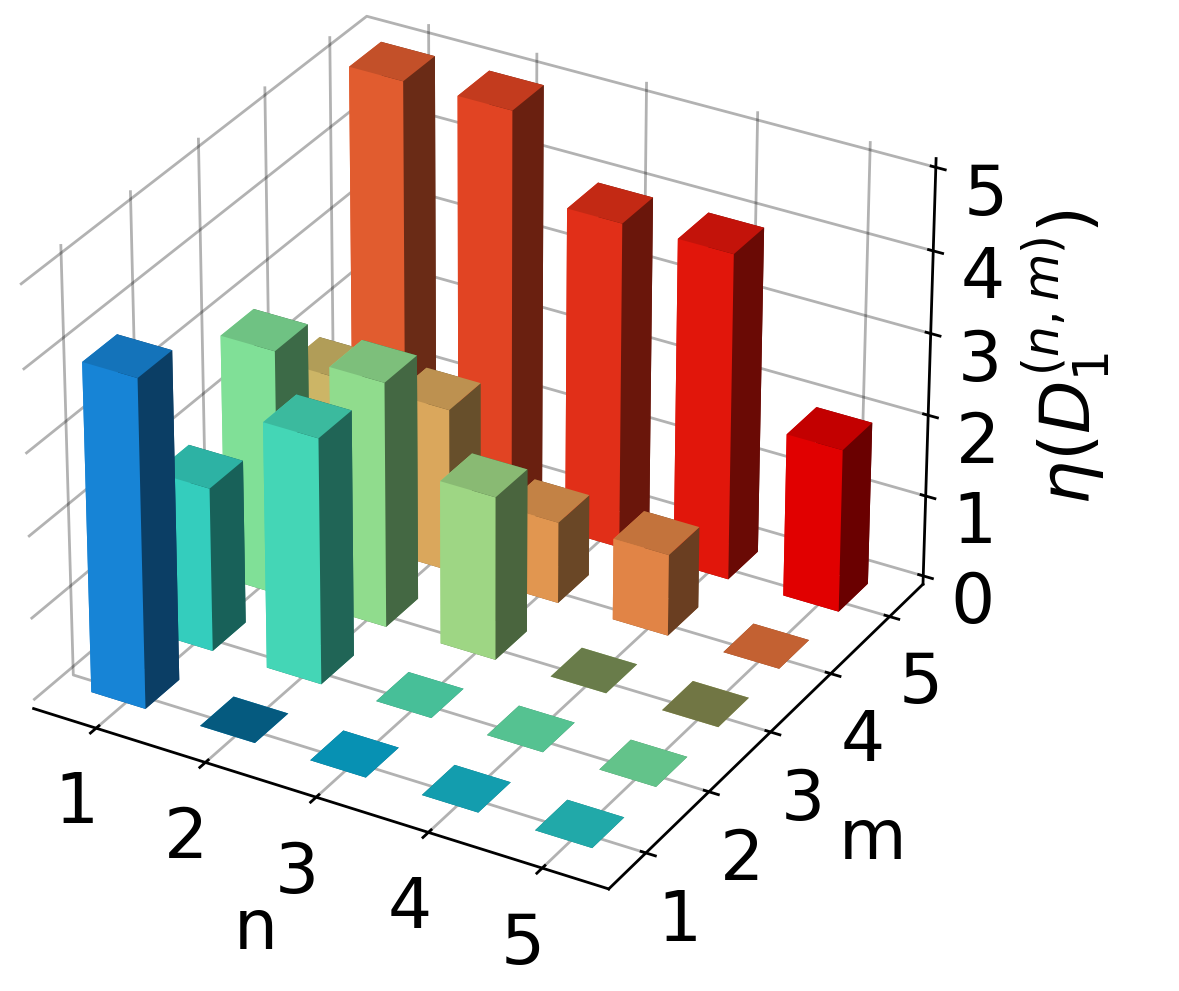}
						\caption{}
						\label{digraph_filtration_features:2a}
					\end{subfigure}%
					\hspace{0.1cm}
					\begin{subfigure}{.32\textwidth}
						\centering
						\includegraphics[width=.99\linewidth]{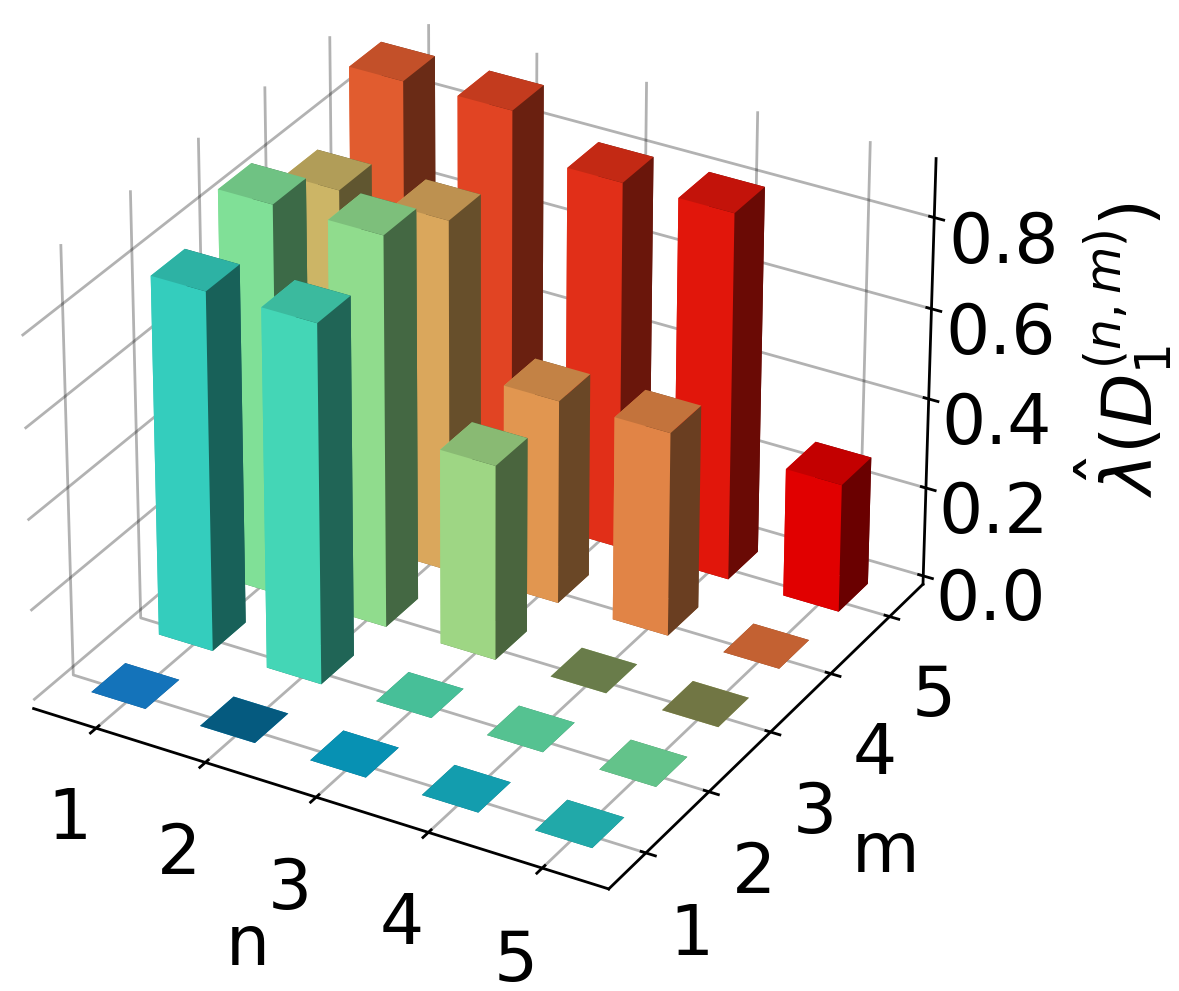}
						\caption{}
						\label{digraph_filtration_features:2b}
					\end{subfigure}%
					\hspace{0.1cm}
					\begin{subfigure}{.32\textwidth}
						\centering
						\includegraphics[width=.99\linewidth]{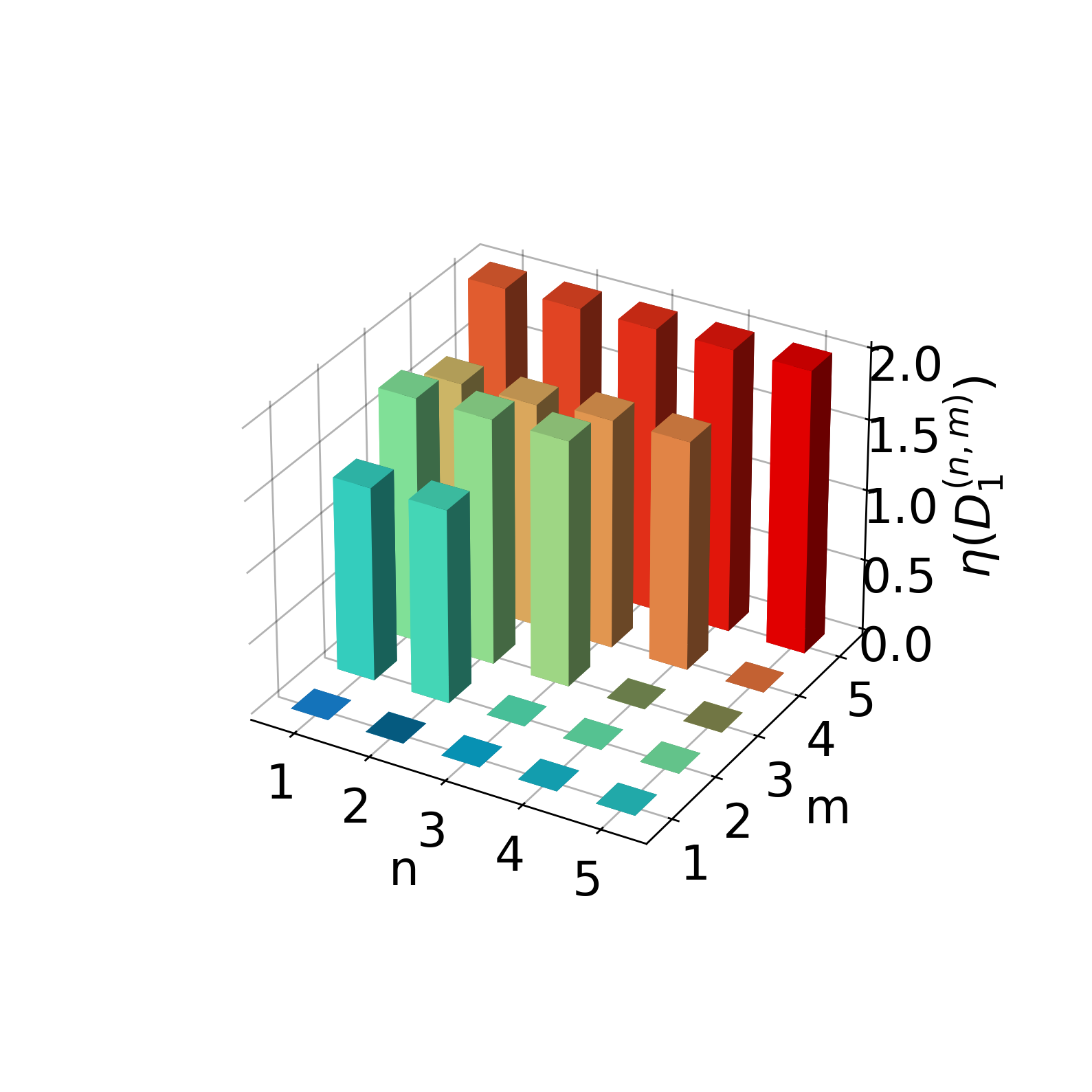}
						\caption{}
						\label{digraph_filtration_features:2c}
					\end{subfigure}
					
					\begin{subfigure}{.32\textwidth}
						\centering
						\includegraphics[width=.99\linewidth]{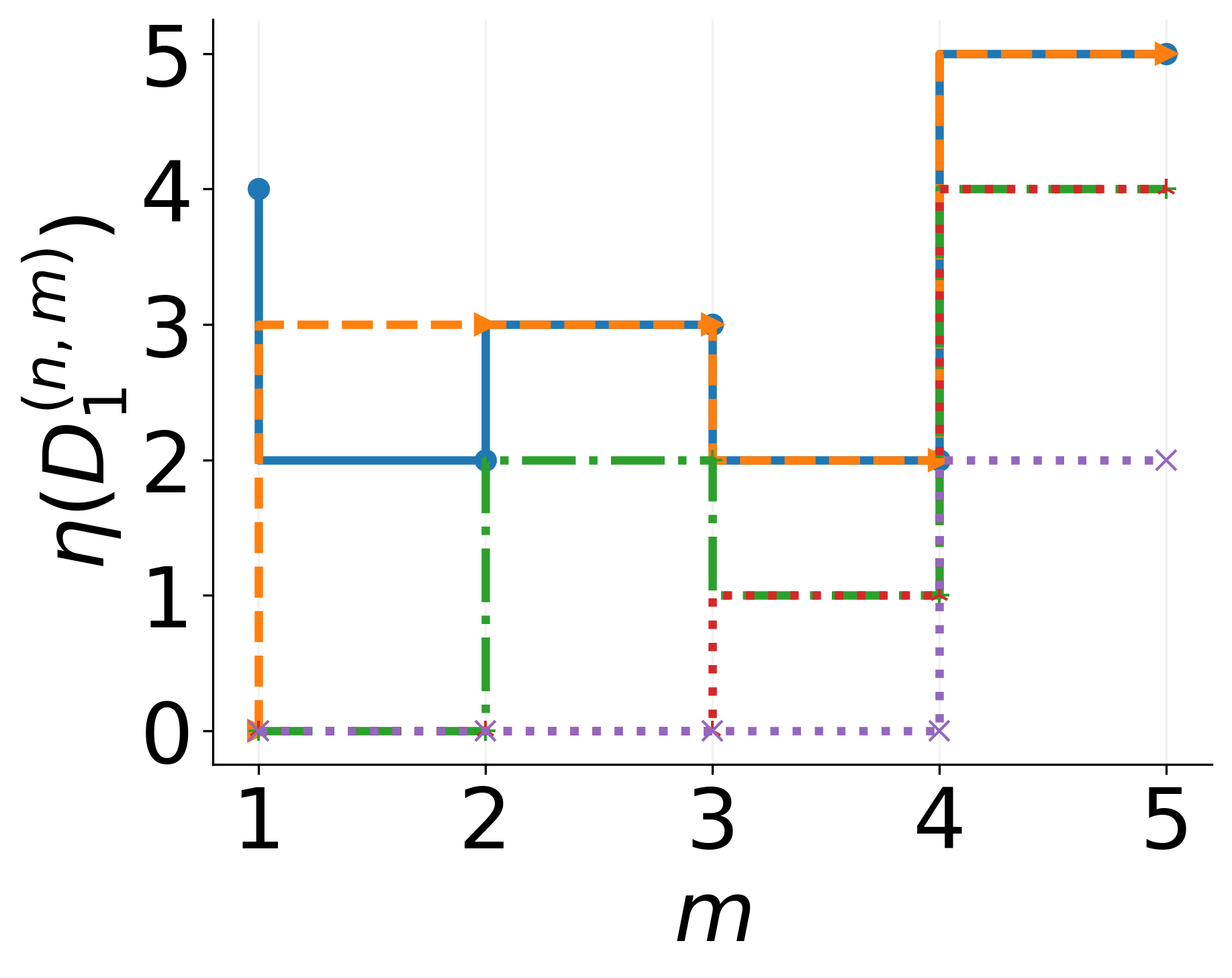}
						\caption{}
						\label{digraph_filtration_features:2d}
					\end{subfigure}%
					\hspace{0.1cm}
					\begin{subfigure}{.32\textwidth}
						\centering
						\includegraphics[width=.99\linewidth]{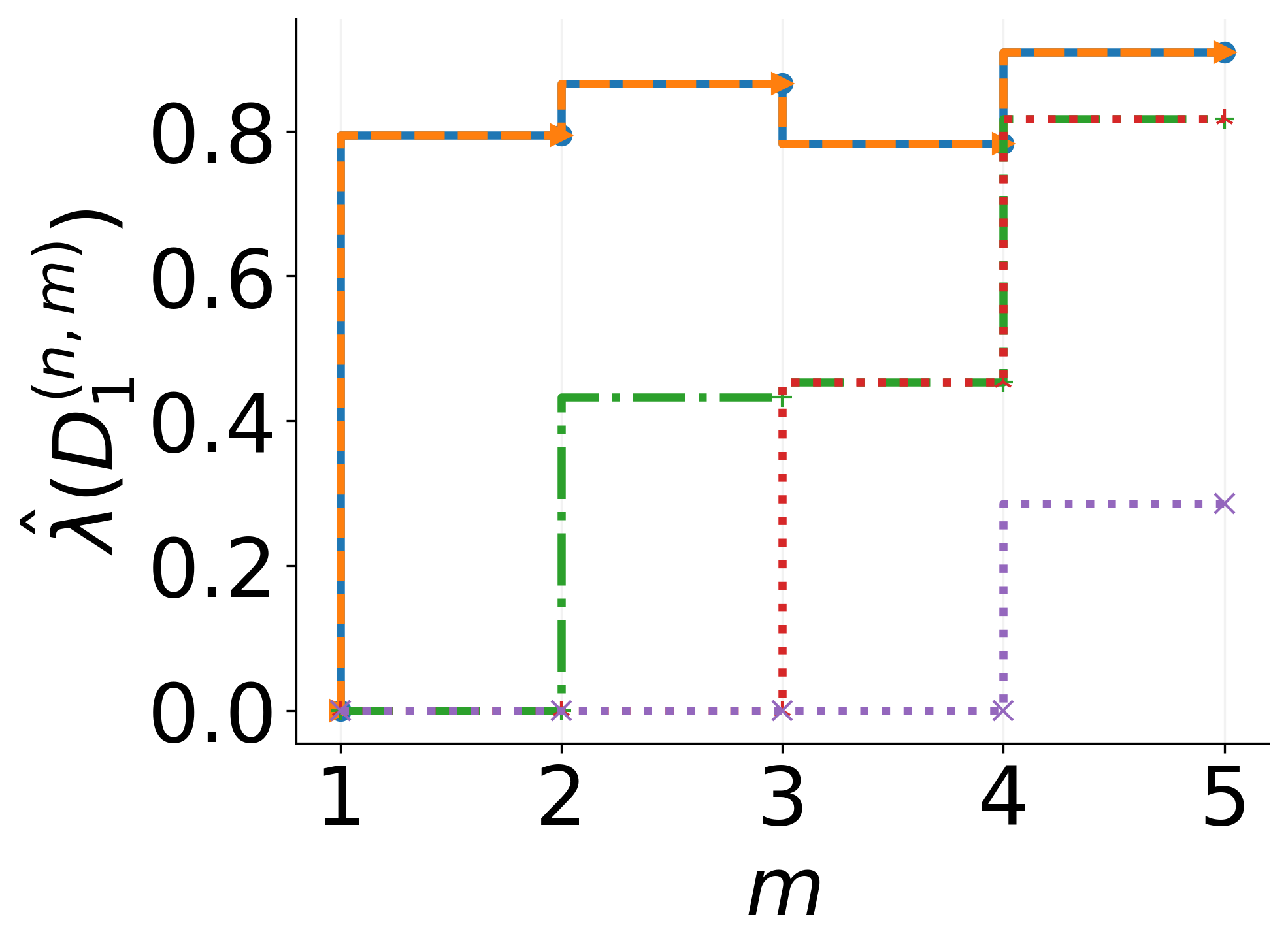}
						\caption{}
						\label{digraph_filtration_features:2e}
					\end{subfigure}%
					\hspace{0.1cm}
					\begin{subfigure}{.32\textwidth}
						\centering
						\includegraphics[width=.99\linewidth]{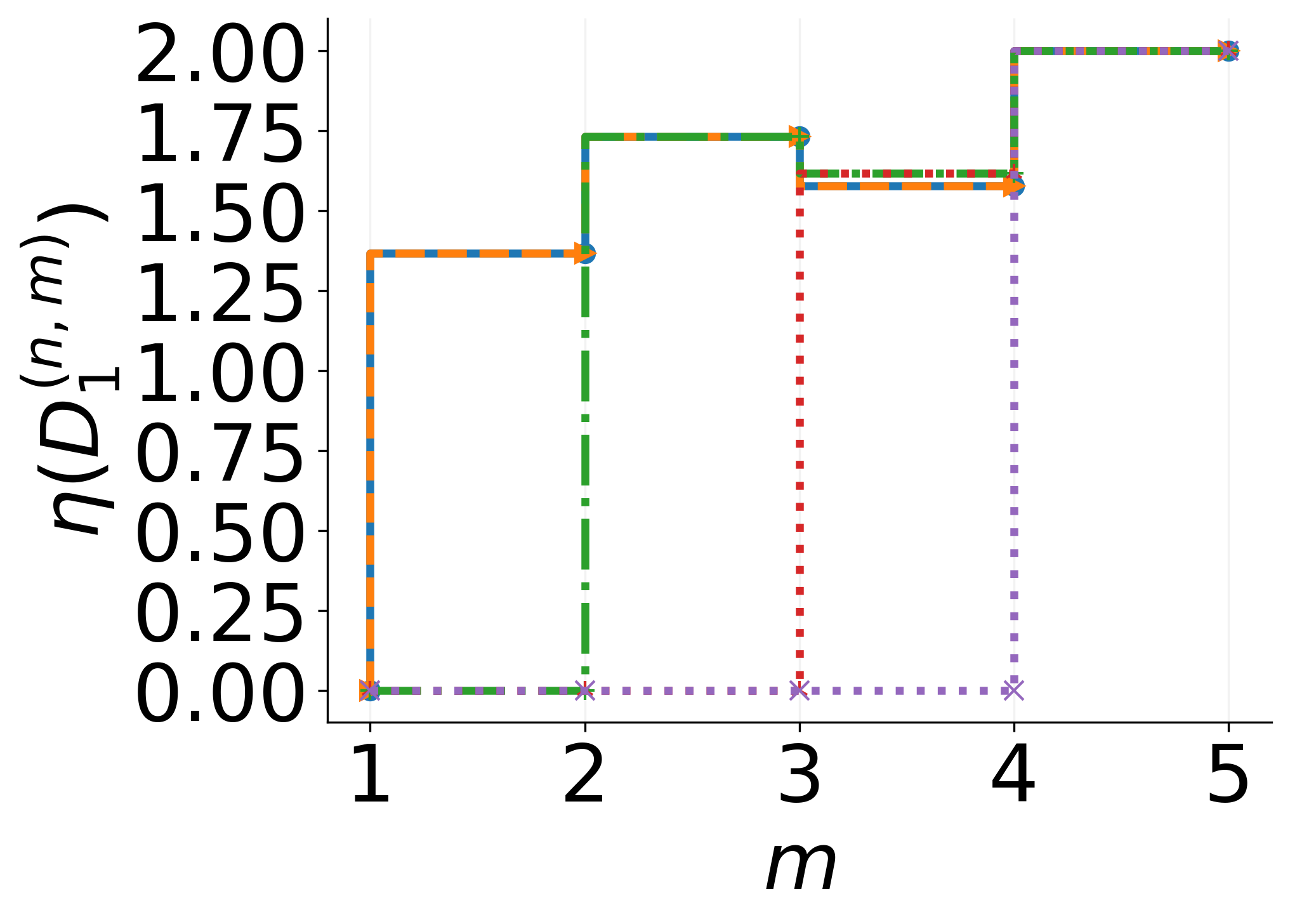}
						\caption{}
						\label{digraph_filtration_features:2f}
					\end{subfigure}
					
					\begin{subfigure}{.70\textwidth}
						\centering
						\includegraphics[width=.99\linewidth]{Legends}
					\end{subfigure}
					
					\caption{Captured features by persistent path Dirac operators induced by the filtrations in Figure \ref{digraph_filtration:1}. {(A)}, {(B)} and {
							(C)} show $\eta(D_1^{(n,m)\\
						})$, $\hat{\lambda}(D_1^{(n,m)})$ and $\overline{\lambda}(D_1^{(n,m)})$ of the second filtration in Figure \ref{digraph_filtration:1} respectively. {(D)}, {(E)}, and {(F)} show their projections on the $mz$-plane.}
					\label{digraph_filtration_features:2}
				\end{figure}
				

			\end{example}
			
			
			\subsection{Relationship of persistent path/hypergraph Dirac with persistent Laplacian}

			Several topological Laplacian operators have been constructed based on the type of objects and homologies they are used on. We list these objects, which consist of a family of discrete objects that satisfy certain properties and their corresponding topological Laplacians. Let $V$ be a nonempty finite ordered set, $P(V)$ the collection of all subsets of $V$ excluding the empty set, and $S(V)$ the set of all sequences in $V$. Recall that the usual graph homology of a graph $G=(V, E)$ with a vertex set $V$ is the homology of the chain complex $(\Omega_\bullet, \partial_\bullet)$ for which the vertices and edges generate its first two components (represented by ordered pairs $(v_0,v_1)$ based on $V$ ordering, e.g., $v_0<v_1$) respectively, and all higher degree components are set to be zero. For a hypergraph $H$, the usual hypergraph homology is just the embedded homology induced by the edges of $H$, represented by ordered tuples based on $V$ ordering, in the chain complex $(C_\bullet,\partial_\bullet)$ induced on $V$ as constructed before. However, the version we are considering is the one we have defined, which is induced by the anchor sequences in $H$. A simplicial complex $\mathcal{K}=(V, K)$ is a hypergraph with the property that if $e$ is an edge, and $\tau \subset e$, then $\tau$ is an edge in $K$. Its homology is the usual hypergraph homology. For a digraph $G$, its usual homology is the homology of the chain complex $(\Omega_\bullet, \partial_\bullet)$ for which its first two components are generated by the vertices and edges of $G$, and the higher components are set to be zero. However, in our case, we consider a natural extension of this complex and its homology by path complexes. A path complex $P$ is a collection of elements of $S(V)$ such that if $v=(v_0,v_1,\cdots,v_{n-1},v_n)\in P$, then $(v_1,\cdots,v_{n-1},v_n)$ and $(v_0,v_1,\cdots,v_{n-1})$ are in $P$. Its homology is the constructed one in the Path Dirac section.
			
			A tightly related topological Laplacian to persistent path Dirac is the persistent path Laplacian \cite{wang2022persistent}. Though both operate on digraphs as constructed in this article, they exhibit a significant distinction, as they encode information about distinct complexes and homologies. Precisely, the $n$-th persistent path Laplacian preserves details related to the $n$-th homology of the chain complex
			\begin{equation}
				\xymatrix@C=3em{
					\cdots \ar[r] & C_{n+1}	\ar[r]^{d^{B,A}_{n+1}} &A_{n}\ar[r]^{d^{A}_{n}}  & A_{n-1} \ar[r]^{d^{A}_{n-1}}& \cdots 
				}
			\end{equation}
			and is given by
			\begin{equation}
				\Delta_n^{B,A}=\Delta_n = (d^A_n)^*d^A_n+d^{B,A}_{n+1}(d^{B,A}_{n+1})^*.
			\end{equation}
			On the other hand, the n-th persistent path Dirac operator conveys information pertaining to the homologies of degrees up to and including n within the auxiliary chain complex $(C_\bullet, d_\bullet)$, as illustrated in the diagram

			\begin{equation}
				\xymatrix@C=3em{
					\cdots \ar[r] &A_{n}\ar[r]^{d^{A}_{n}} \ar[d] & A_{n-1} \ar[r]^{d^{A}_{n-1}}\ar[d]& \cdots \ar[r]^{d^{A}_{2}}\ar[d]& A_{1} \ar[r]^{d^{A}_{1}}\ar[d]& A_0 \ar[r]\ar[d]& 0\\
					\cdots \ar[r]&  C_{n} \ar[r]^{d^{C}_{n}}\ar[ur]^{d^{B,A}_{n}}\ar@{_{(}.>}[d] & C_{n-1}\ar[r]^{d^{C}_{n-1}}\ar[ur]^{d^{B,A}_{n-1}}\ar@{_{(}.>}[d] & \cdots \ar[r]^{d^{C}_{2}} \ar[ur]^{d^{B,A}_{2}}\ar@{_{(}.>}[d] & C_1\ar[r]^{d^{C}_{1}} \ar[ur]^{d^{B,A}_{1}}\ar@{_{(}.>}[d] & C_0 \ar@{_{(}.>}[d] \ar[r] & 0 \\
					\cdots \ar[r] & B_{n}\ar[r]^{d^{B}_{n}} & B_{n-1} \ar[r]^{d^{B}_{n-1}}& \cdots \ar[r]^{d^{B}_{2}}& B_{1} \ar[r]^{d^{B}_{1}}& B_0 \ar[r]& 0 	
				}
			\end{equation}
			
			However, when properly managed, both sequences of persistent path Laplacians $$(\Delta_0^{B,A},\ \Delta_1^{B,A},\ \cdots,\ \Delta_n^{B,A})$$ and persistent path Dirac operators $$(D_0^{B,\ A},\ D_1^{B,\ A},\ \cdots,\ D_n^{B,\ A})$$ have the same computational complexities. In general, the behavior of these two types of operators differs significantly, making it unreasonable to compare them against each other. One may be more suitable for analyzing a particular situation than the other. Furthermore, non-persistent Dirac operators encompass information from non-persistent Laplacians and more, rendering them legitimately comparable. Consequently, both types of persistent operators are consequential, and they are better employed in conjunction to capture the complete persistence information. Moreover, possible Dirac operators can be constructed for graph, simplicial complex and hyperdigraph homologies as discussed in the Concluding Remarks.

			
			\begin{example}
				According to the comments after Proposition \ref{prop:2}, we can infer the eigenvalues of a Dirac operator from its corresponding Laplacian operator. However, this is not generally true for persistent path Dirac and Laplacian operators. In this example, we show that persistent path Dirac gives rise to eigenvalues that cannot be inferred by the corresponding persistent path Laplacians. Consider the digraph filtration depicted in Figure \ref{digraph_filtration_features:4}. Table \ref{table:6} shows the calculated eigenvalues of the first-degree persistent path Dirac operators and their corresponding first-degree persistent path Laplacian operators. We notice that $L_1^{(0,1)}$ has an eigenvalue 2, while as $D_1^{(0,1)}$ does not have $\sqrt{2}$. Similarly, $L_1^{(0,2)}$ has an eigenvalue 2, and $D_1^{(0,2)}$ does not have $\sqrt{2}$.
				

				\begin{figure}[t!]
					\centering
					\begin{subfigure}{.20\textwidth}
						\centering
						\includegraphics[width=.99\linewidth]{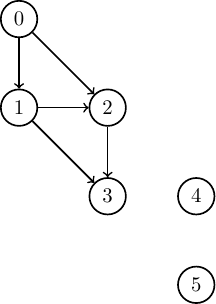}
						\caption{}
					\end{subfigure}%
					\hspace{0.2cm}
					\begin{subfigure}{.20\textwidth}
						\centering
						\includegraphics[width=.99\linewidth]{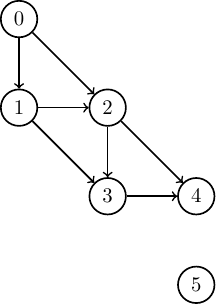}
						\caption{} 
					\end{subfigure}%
					\hspace{0.2cm}
					\begin{subfigure}{.20\textwidth}
						\centering
						\includegraphics[width=.99\linewidth]{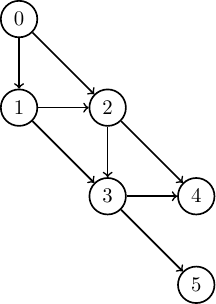}
						\caption{} 
					\end{subfigure}%
					\hspace{0.2cm}
					\begin{subfigure}{.20\textwidth}
						\centering
						\includegraphics[width=.99\linewidth]{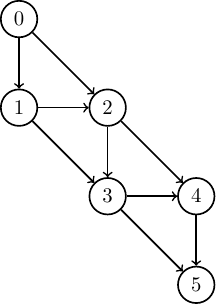}
						\caption{} 
					\end{subfigure}%
					
					\caption{A filtration of a digraph to compare the eigenvalues of persistent path Laplacian and persistent path Dirac operators.}
					\label{digraph_filtration_features:4}
					
				\end{figure}
				
				\begin{table}[t!] 
					\centering
					\begin{tabular}{@{}llll}
						\hline
						$\text{Spec}(L^{(0,1)}_1)=$ & $\{2,4,4,4,5\}$  &  $\text{Spec}(D^{(0,1)}_1)=$ & $\{0 (\times 5), \pm\sqrt{1.59},\pm\sqrt{3},\pm 2,$  \\ [0.5ex] 
						& & & $\pm\sqrt{4.41},\pm\sqrt{5},\pm\sqrt{5}\}$ \\
						$\text{Spec}(L^{(0,2)}_1)=$ & $\{2,4,4,4,5\}$  &  $\text{Spec}(D^{(0,2)}_1)=$ & $\{0 (\times 4), \pm\sqrt{0.89},\pm\sqrt{1.7},\pm 2,$  \\ [0.5ex] 
						& & & $\pm\sqrt{3.25},\pm\sqrt{4.86},\pm\sqrt{5},\pm\sqrt{5.3}\}$ \\
						$\text{Spec}(L^{(0,3)}_1)=$ & $\{2,4,4,4,5.09\}$  &  $\text{Spec}(D^{(0,3)}_1)=$ & $\{0 (\times 6), \pm\sqrt{1.19},\pm\sqrt{3},\pm 2,$  \\ [0.5ex] 
						& & & $\pm\sqrt{3.47},\pm\sqrt{5},\pm\sqrt{5.09},\pm\sqrt{5.34}\}$ \\
						\hline
					\end{tabular}
					\caption{Eigenvalues of persistent path Laplacian and persistent path Dirac operators induced by the filtration in Figure \ref{digraph_filtration_features:4}.} 
					\label{table:6}
				\end{table}
				

			\end{example}
			
			
			\section{Applications}
			
			In this section, we explore the applications of path and hypergraph complexes and persistent path Dirac and persistent hypergraph Dirac operators in the field of biology, specifically focusing on their utilization with molecules. The approach we use \cite{wang2022persistent, chen2023persistent} involves enhancing molecules with path and hypergraph complexes derived from the constituent atoms' inherent characteristics. We then apply a filtration process, which can be based on distance or other pertinent features of interest. Atoms possess various distinguishing attributes, including atomic number, atomic mass, and covalent radius. In this section, our primary focus is examining the impact of the Pauling Scale of electronegativity on molecules.
			
			A preorder on a nonempty set $P$ is a relation that satisfies the following:
			\begin{enumerate}
				\item Reflexivity: $a\leq a$ for all $a\in P$.
				\item Transitivity: For all $a,b,c\in P$, if $a\leq b$, $b\leq c$, then $a\leq c$.
			\end{enumerate} 
			
			Preorders naturally give rise to digraphs by their inherent relations. These digraphs tend to induce rich path and hypergraph complexes, as they often encompass the essential subgraphs needed to induce higher-degree path and hypergraph complex components. To avoid including loops, we specifically consider a strict preorder, which means we do not permit reflexive relations.
			
			Assigning real values to atoms produces a natural preorder on atoms, where for any two atoms A and B, we say A is less than or equal to B (A $\leq$ B) if and only if the assigned value to A is less than or equal to the value assigned to B. The periodic table then can be visualized as stacks of atoms on a line, and atoms of a molecule can be visualized the same. Our proposed application involves assigning atoms electronegativity values based on the Pauling Scale. This assignment establishes a natural preorder for both individual atoms and molecules. It results in induced digraphs that can capture certain aspects of molecular nature and geometric behavior. We can then apply filtration techniques and compute the resulting path or hypergraph complexes and their corresponding Dirac operators. These calculations yield numerical features that can be leveraged in various learning processes and analyses.
			
			We illustrate our application through two specific examples. First, we equip a Glycogen molecule with a directed graph reflecting the bonds between its constituent atoms only. We will consider the subdirected graph induced by the bonds between the atoms instead of the whole directed graph induced by the preorder. We determine the direction of these bonds based on Pauling electronegativity values as described in \cite{pauling1932nature}. The Glycogen molecule is comprised of oxygen (O), hydrogen (H), and carbon (C) atoms, which possess electronegativity values of 3.44, 2.2, and 2.55, respectively. As a result, the assignment of directions is detailed in Figure \ref{application:1:digraph}.
			\begin{figure}[t!]
				\centering
				\begin{subfigure}{.25\textwidth}
					\centering
					\includegraphics[width=.99\linewidth]{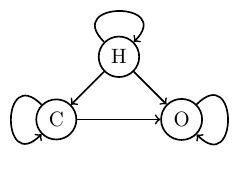}
				\end{subfigure}
				\caption{A figure that demonstrates the allowed digraph edges equipped on the Glycogen molecule, where the loops say two atoms of the same type are connected, and does not mean a directed edge from an atom to itself is allowed for no loops are allowed.} 
				\label{application:1:digraph}
			\end{figure}
			Note that we do not permit loops in our construction. The loops depicted in Figure \ref{application:1:digraph} indicate that we allow arrows between atoms of the same type but do not allow arrows from an atom to itself. To determine the distances between atoms in Glycogen, we have computed them using the MolView website \cite{MolView}, and the results are presented in Table \ref{table:5}.
			\begin{table}[t!] 
				\centering
				\begin{tabular}{@{}lllll}
						\hline
						& H$\leftrightarrow$ O  &  C$\leftrightarrow$ O& C$\leftrightarrow$ C & C$\leftrightarrow$ H \\ [0.5ex] 
						\hline
						Distance & 0.97 {\AA}& 1.43 {\AA}   & 1.53 {\AA} & 1.1 {\AA} \\ 
						
						\hline
					\end{tabular}
					\caption{The distances between every two atoms in Glycogen molecule that share a bond as calculated by MolView.}
					\label{table:5}
				\end{table}
				\begin{figure}[t!]
					\centering
					\begin{subfigure}{.4\textwidth}
						\includegraphics[width=.99\linewidth]{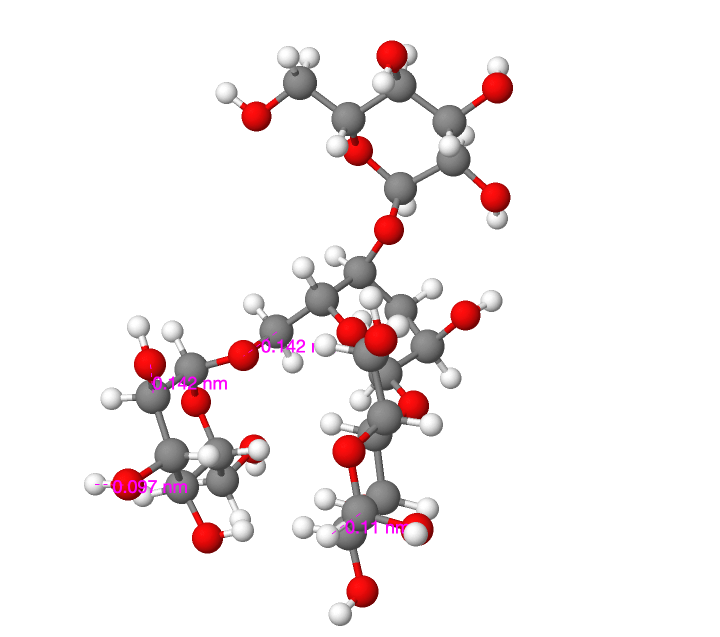}
						\caption{}
					\end{subfigure}%
					\begin{subfigure}{.35\textwidth}
						\includegraphics[width=.99\linewidth]{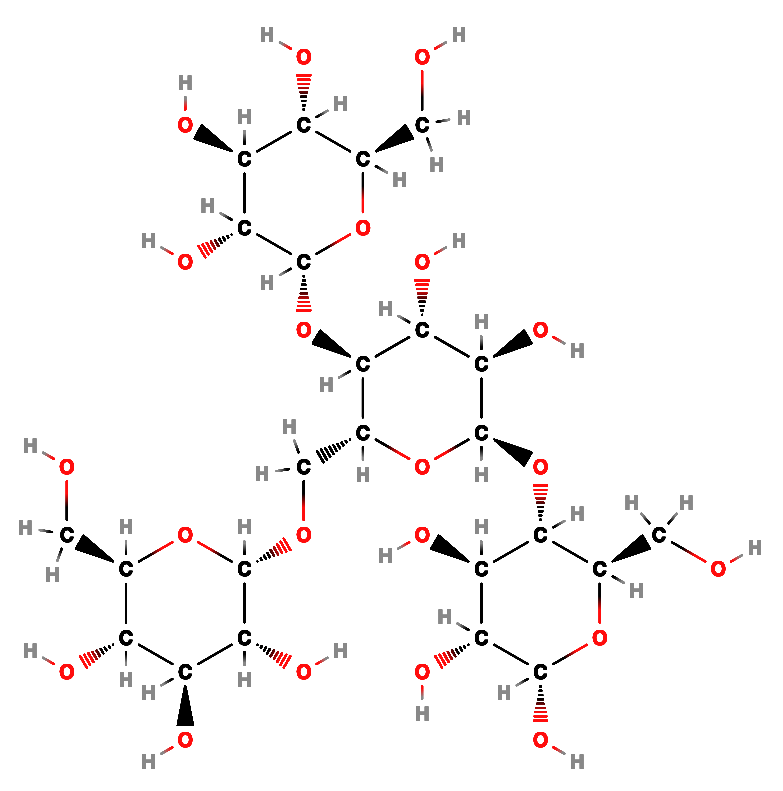}
						\caption{}
					\end{subfigure}
					\caption{Glycogen molecule. {(A)} Glycogen structure. {(B)} Molecular formula.}
					\label{Glycogen}
				\end{figure}
				\begin{figure}[t!]
					\centering
					\begin{subfigure}{.32\textwidth}
						\centering
						\includegraphics[width=.99\linewidth]{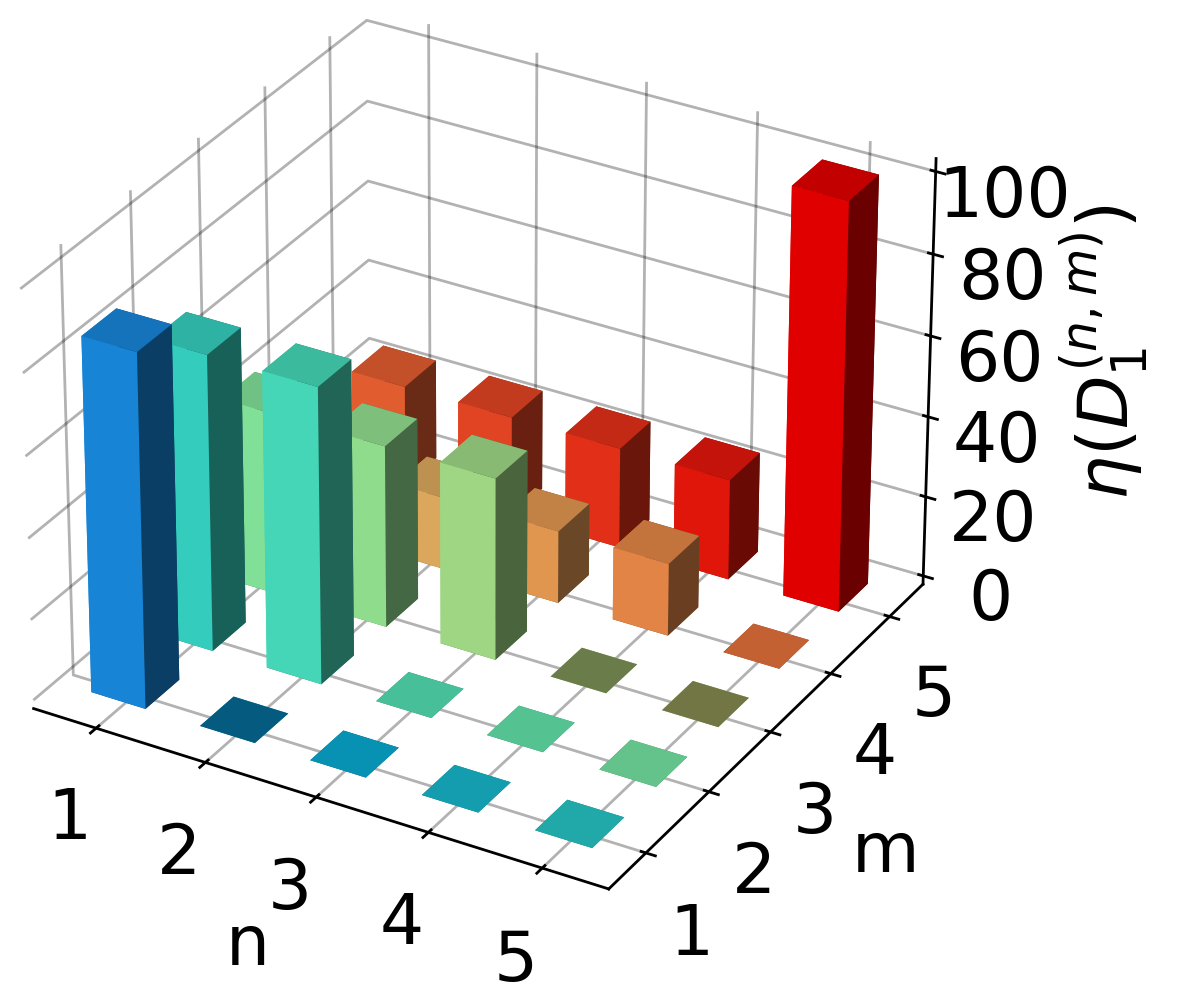}
						\caption{}
						\label{application:1:Features:a}
					\end{subfigure}%
					\hspace{0.1cm}
					\begin{subfigure}{.32\textwidth}
						\centering
						\includegraphics[width=.99\linewidth]{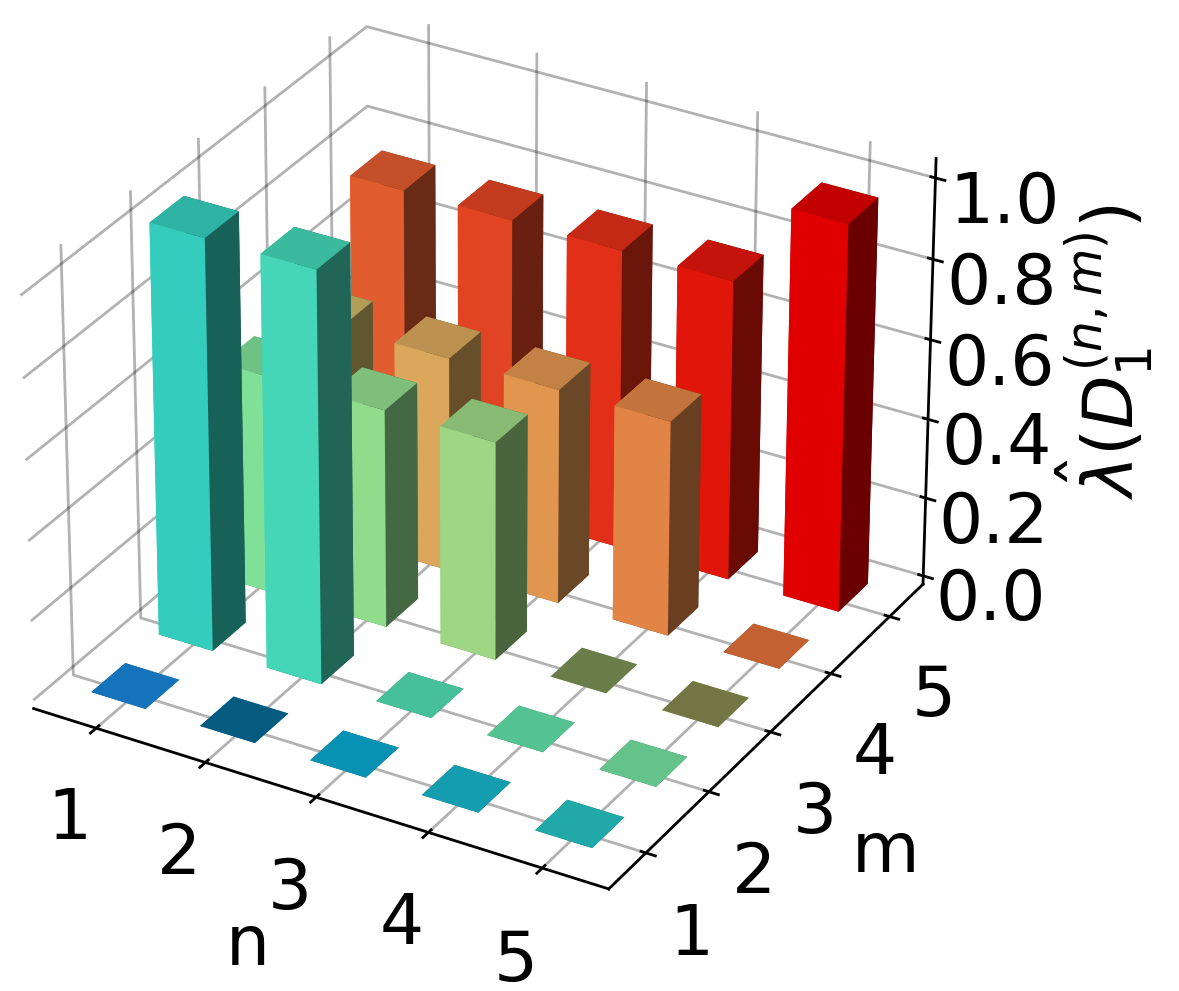}
						\caption{}
						\label{application:1:Features:b}
					\end{subfigure}%
					\hspace{0.1cm}
					\begin{subfigure}{.32\textwidth}
						\centering
						\includegraphics[width=.99\linewidth]{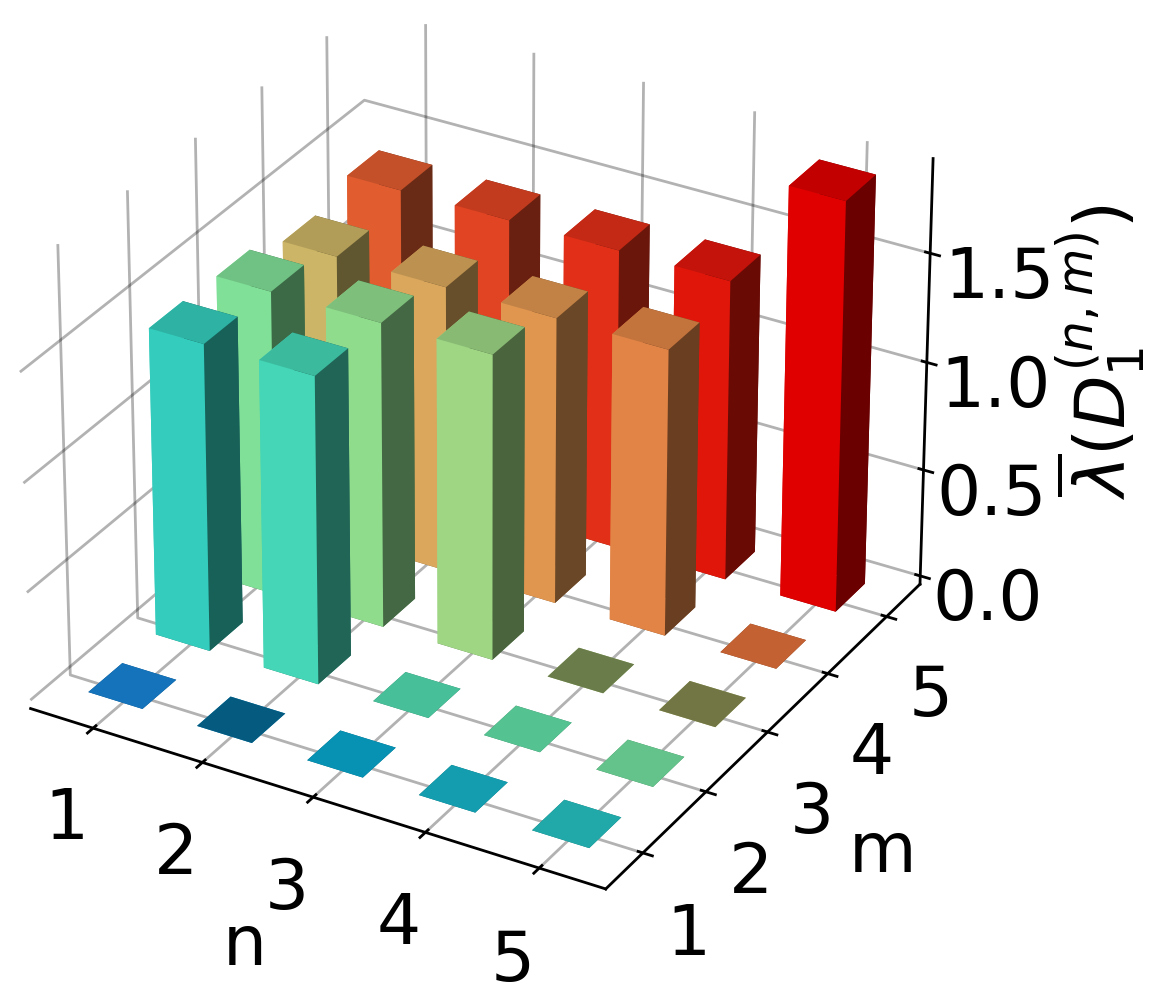}
						\caption{}
						\label{application:1:Features:c}
					\end{subfigure}
					
					\begin{subfigure}{.32\textwidth}
						\centering
						\includegraphics[width=.99\linewidth]{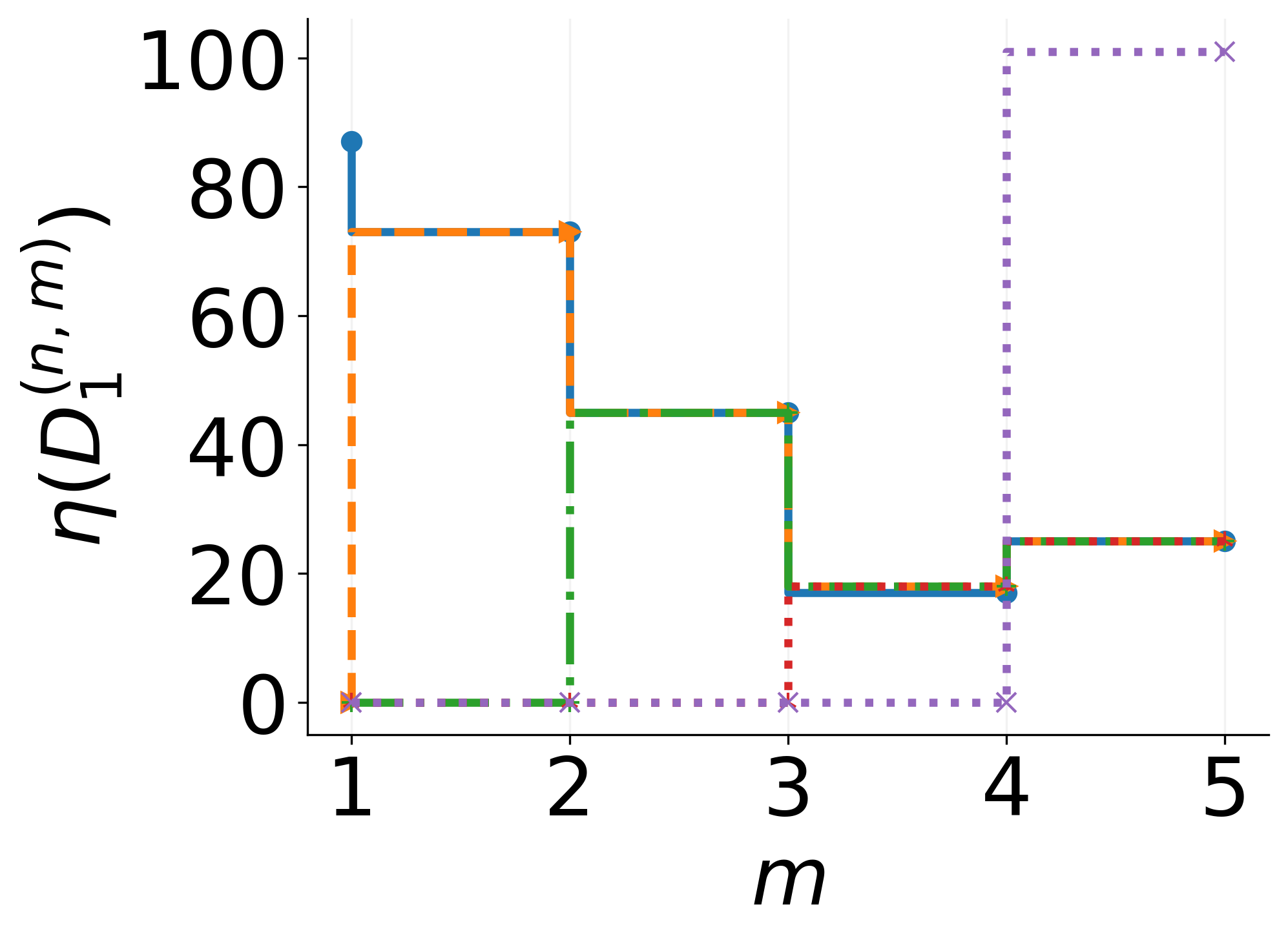}
						\caption{}
						\label{application:1:Features:d}
					\end{subfigure}%
					\hspace{0.1cm}
					\begin{subfigure}{.32\textwidth}
						\centering
						\includegraphics[width=.99\linewidth]{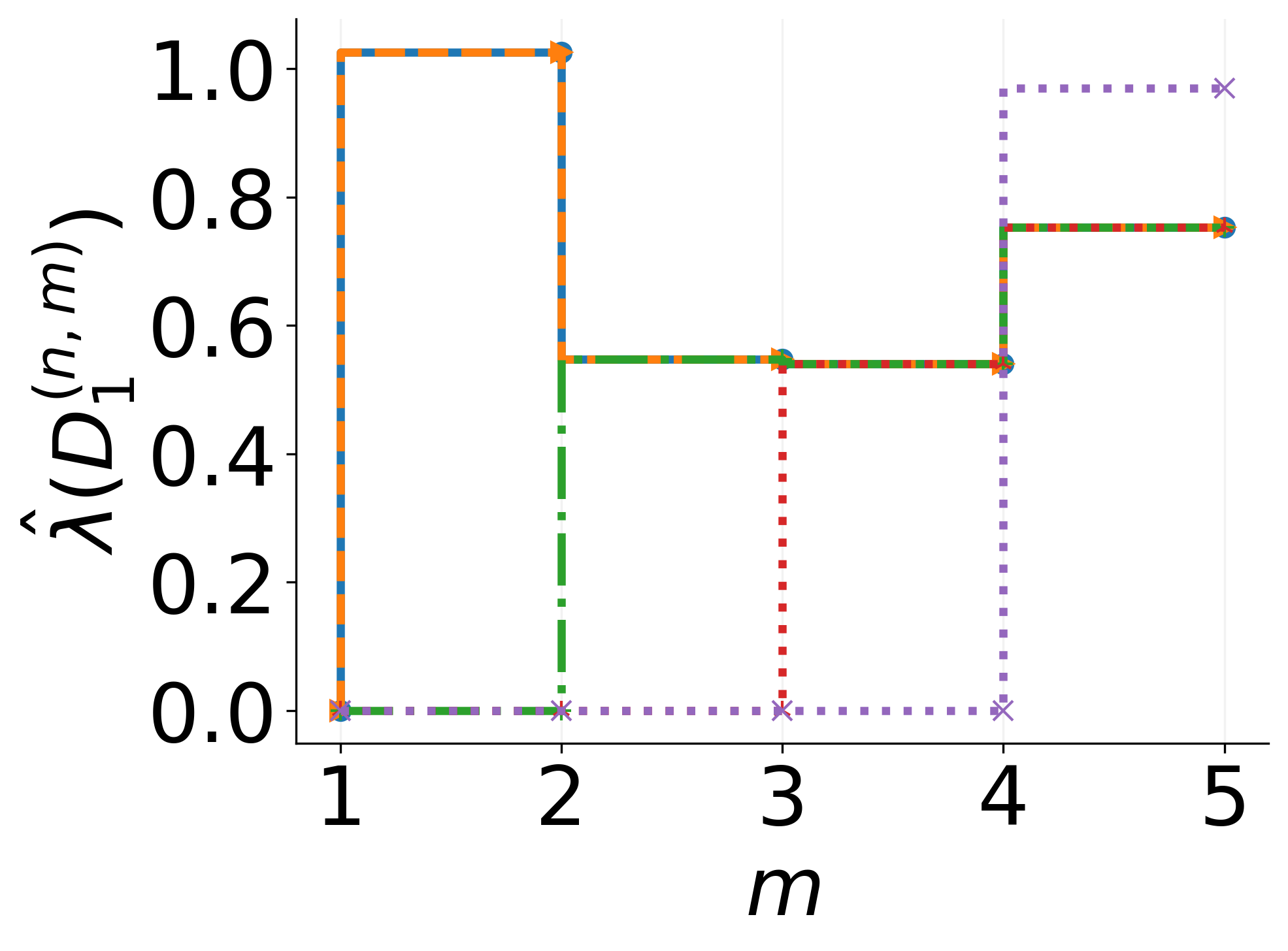}
						\caption{}
						\label{application:1:Features:e}
					\end{subfigure}%
					\hspace{0.1cm}
					\begin{subfigure}{.32\textwidth}
						\centering
						\includegraphics[width=.99\linewidth]{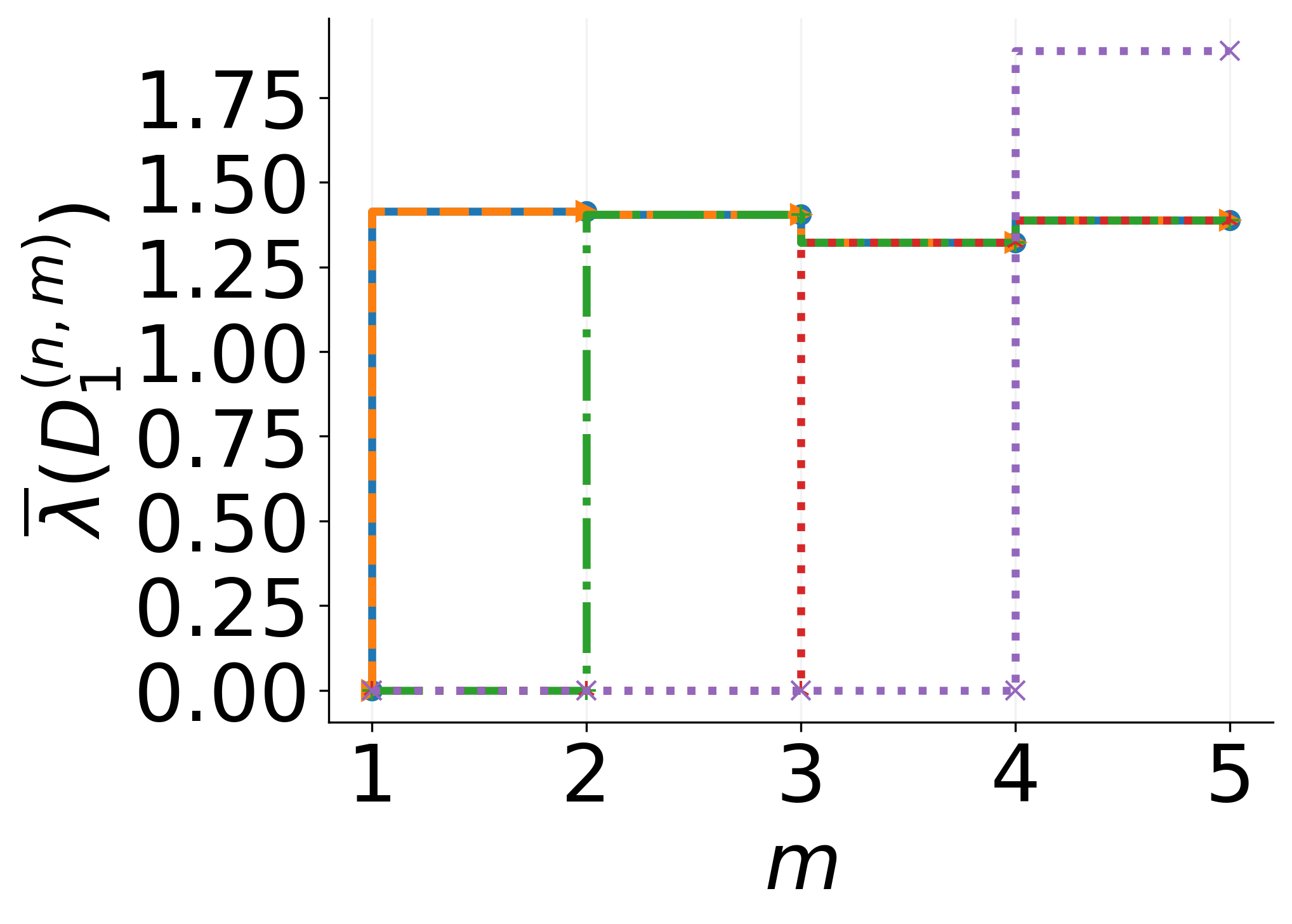}
						\caption{}
						\label{application:1:Features:f}
					\end{subfigure}
					
					\begin{subfigure}{.70\textwidth}
						\centering
						\includegraphics[width=.99\linewidth]{Legends}
					\end{subfigure}
					
					\caption{Captured features by persistent path Dirac operators of the filtration on Glycogen molecule in Figure \ref{Glycogen}. {(A)}, {(B)} and {
							(C)} show $\eta(D_1^{(n,m)\\
						})$, $\hat{\lambda}(D_1^{(n,m)})$ and $\overline{\lambda}(D_1^{(n,m)})$ of the operator respectively. {(D)}, {(E)}, and {(F)} show their projections on the $mz$-plane.}
					\label{application:1:Features}
				\end{figure}
				
				\begin{figure}[t!]
					\centering
					\begin{subfigure}{.32\textwidth}
						\centering
						\includegraphics[width=.99\linewidth]{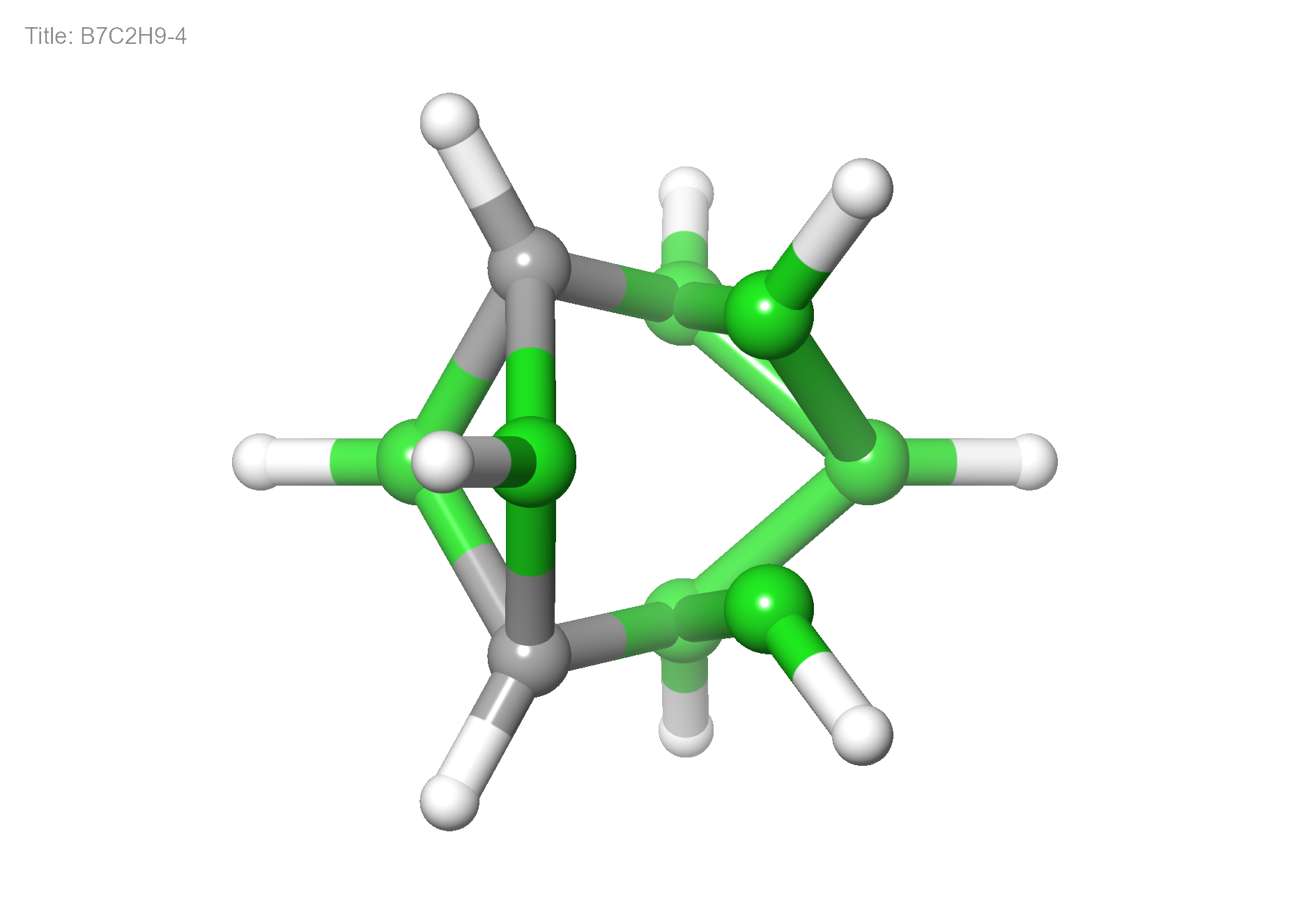}
						\caption{}
						\label{application:2:Features:a}
					\end{subfigure}%
					\hspace{0.1cm}
					\begin{subfigure}{.32\textwidth}
						\centering
						\includegraphics[width=.99\linewidth]{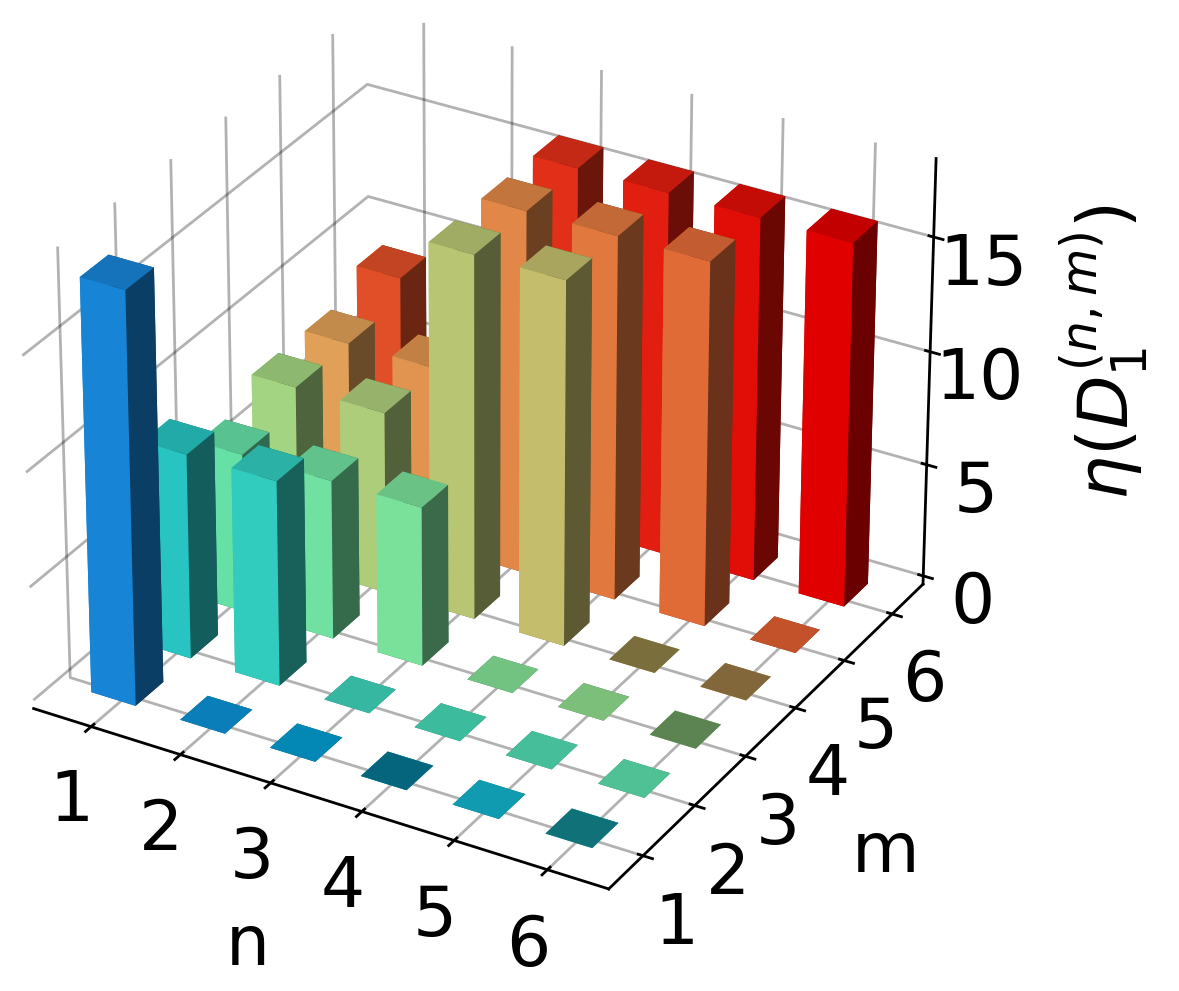}
						\caption{}
						\label{application:2:Features:b}
					\end{subfigure}%
					\hspace{0.1cm}
					\begin{subfigure}{.32\textwidth}
						\centering
						\includegraphics[width=.99\linewidth]{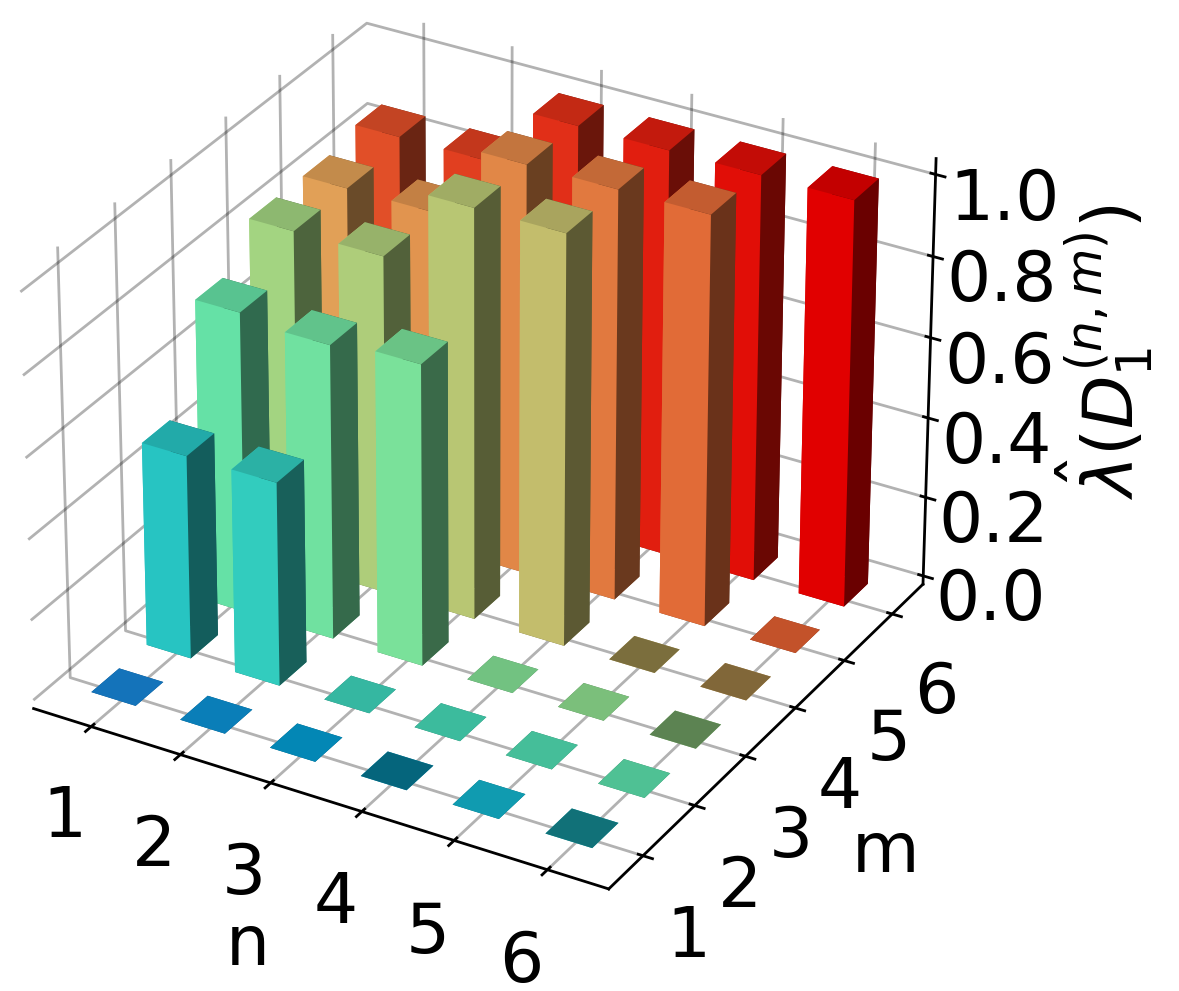}
						\caption{}
						\label{application:2:Features:c}
					\end{subfigure}

					\begin{subfigure}{.32\textwidth}
						\centering
						\includegraphics[width=.99\linewidth]{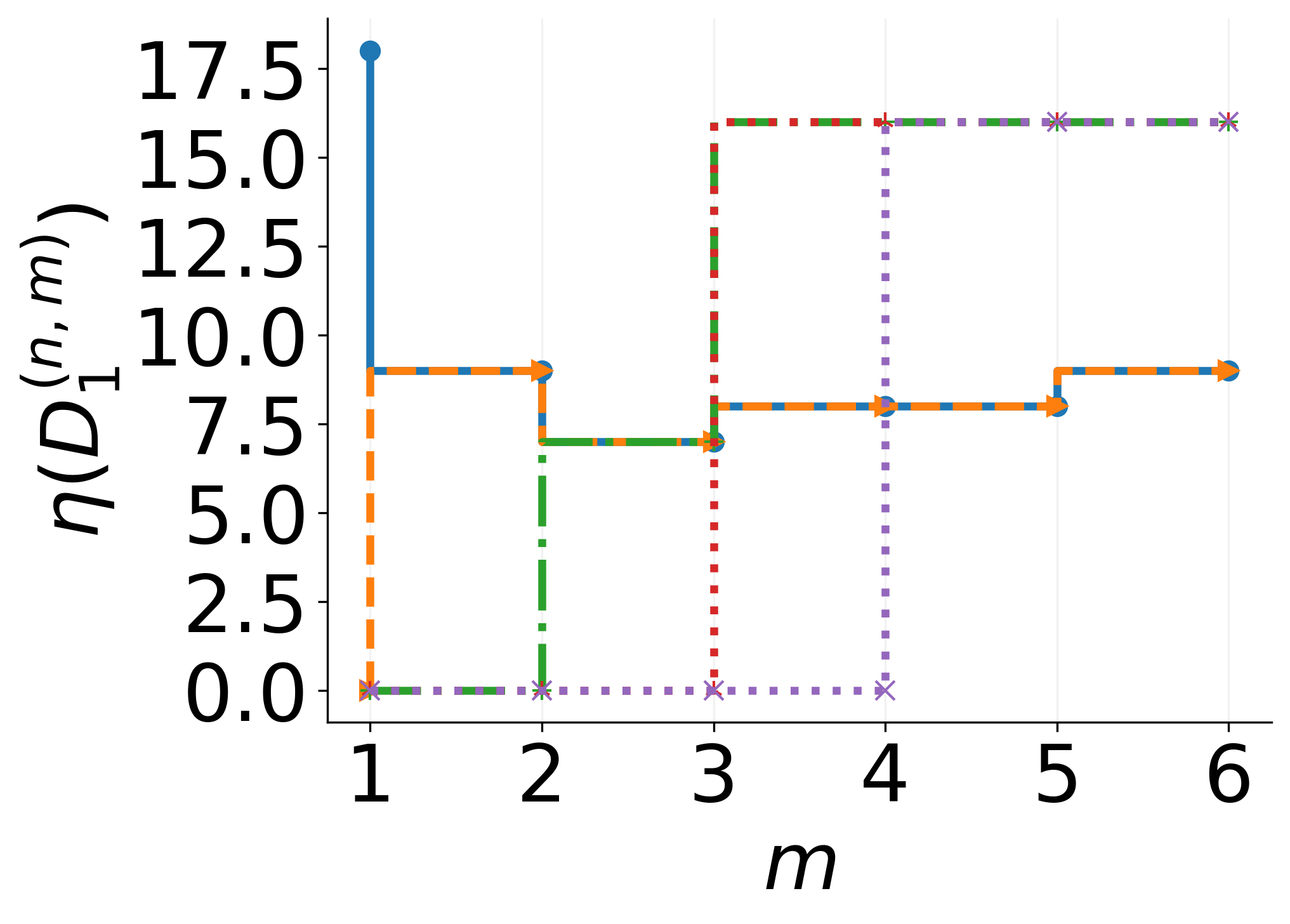}
						\caption{}
						\label{application:2:Features:e}
					\end{subfigure}%
					\hspace{0.1cm}
					\begin{subfigure}{.32\textwidth}
						\centering
						\includegraphics[width=.99\linewidth]{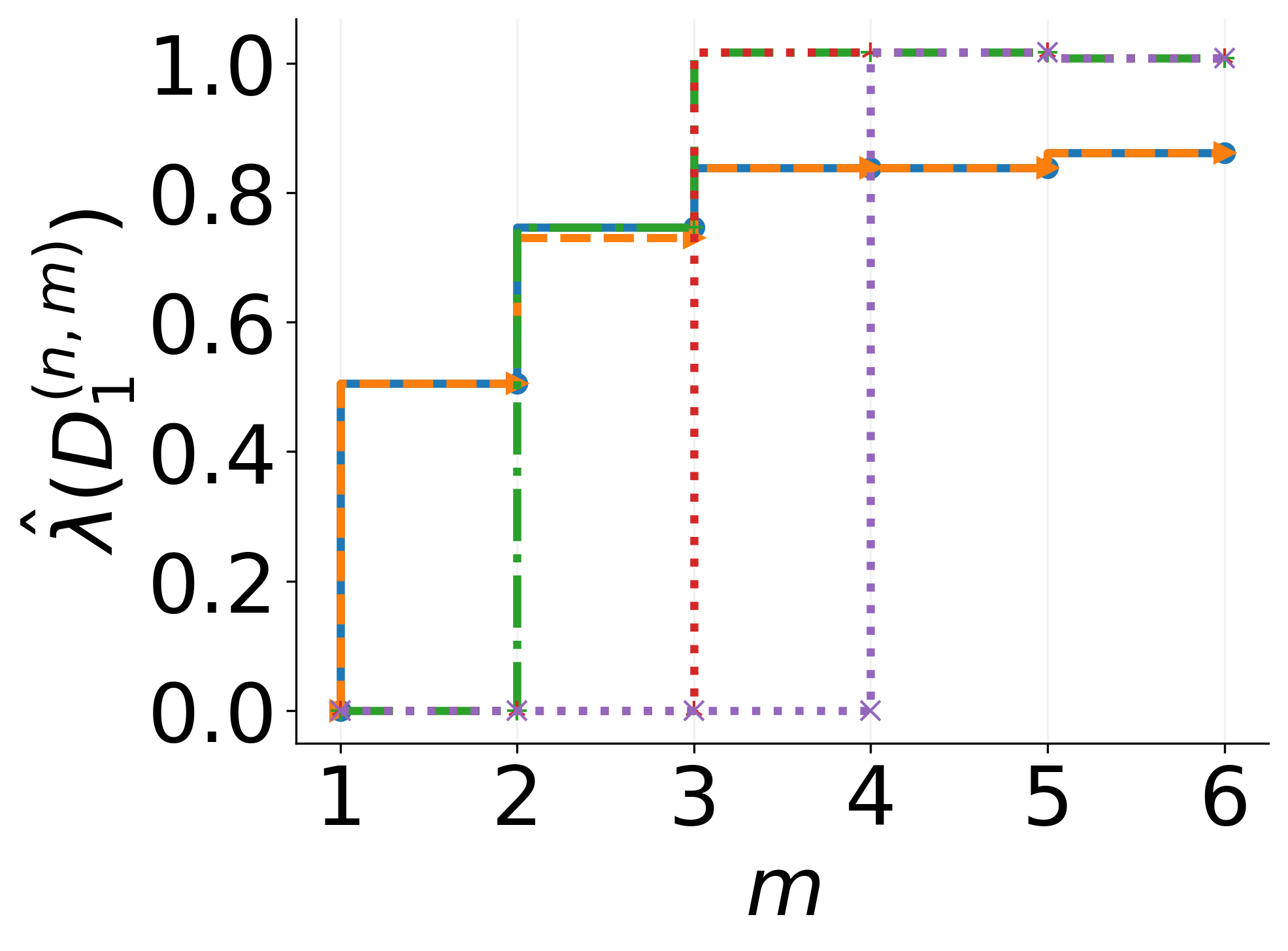}
						\caption{}
						\label{application:2:Features:f}
					\end{subfigure}
					
					\begin{subfigure}{.70\textwidth}
						\centering
						\includegraphics[width=.99\linewidth]{Legends}
					\end{subfigure}

					\caption{$\text{B}_7\text{C}_2\text{H}_9$ first isomer and its captured features. {(A)} shows the obtained protein-ligand structure of the first isomer. {(B)} and {(C)} show $\eta(D_1^{(n,m)})$ and $\hat{\lambda}(D_1^{(n,m)})$ obtained from the first isomer. {(D)} and {(E)} show their projections on the $mz$-plane.}
					\label{application:2:Features}
				\end{figure}

				\begin{figure}[t!]
					\centering
					\begin{subfigure}{.32\textwidth}
						\centering
						\includegraphics[width=.99\linewidth]{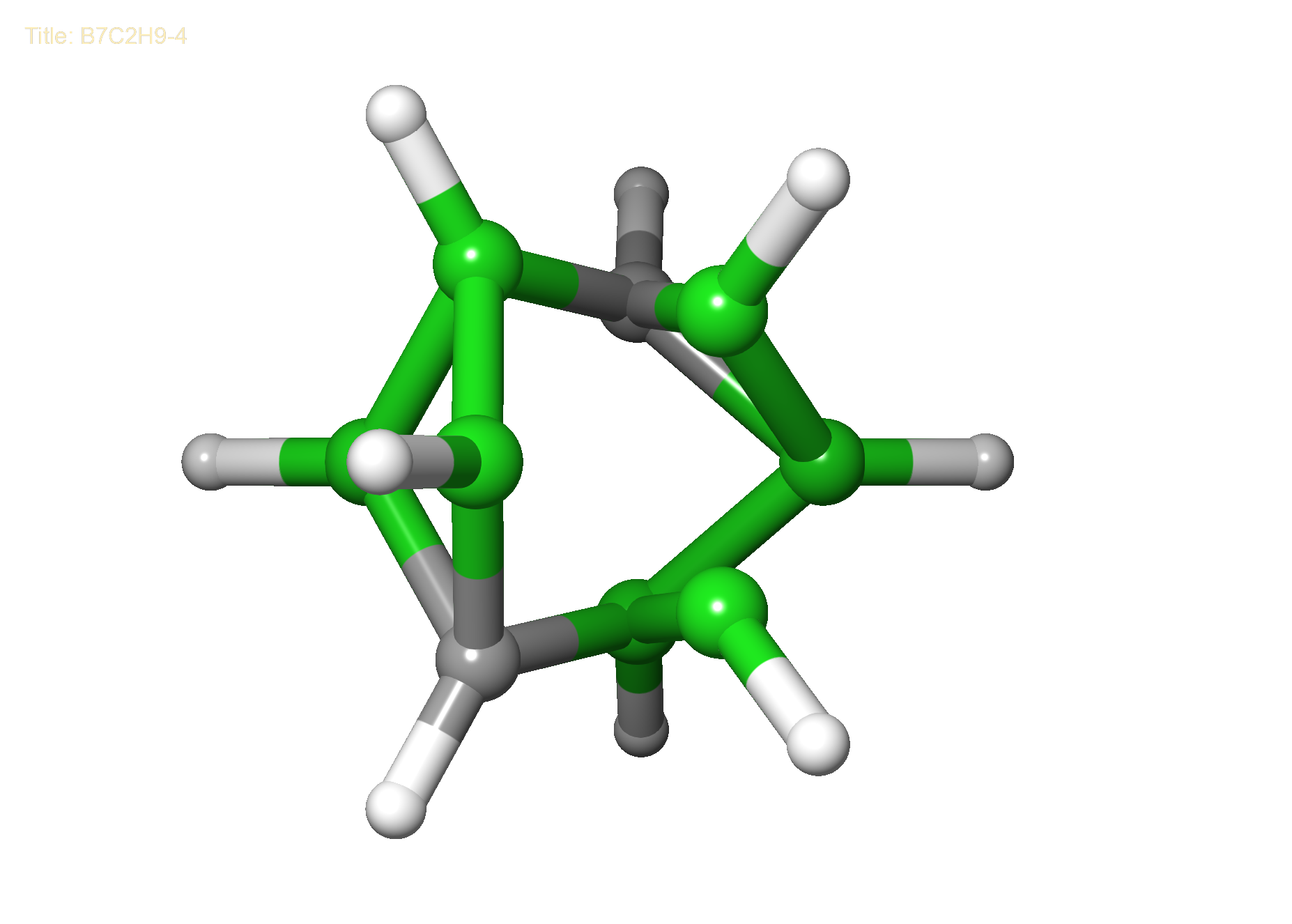}
						\caption{}
						\label{application:3:Features:a}
					\end{subfigure}%
					\hspace{0.1cm}
					\begin{subfigure}{.32\textwidth}
						\centering
						\includegraphics[width=.99\linewidth]{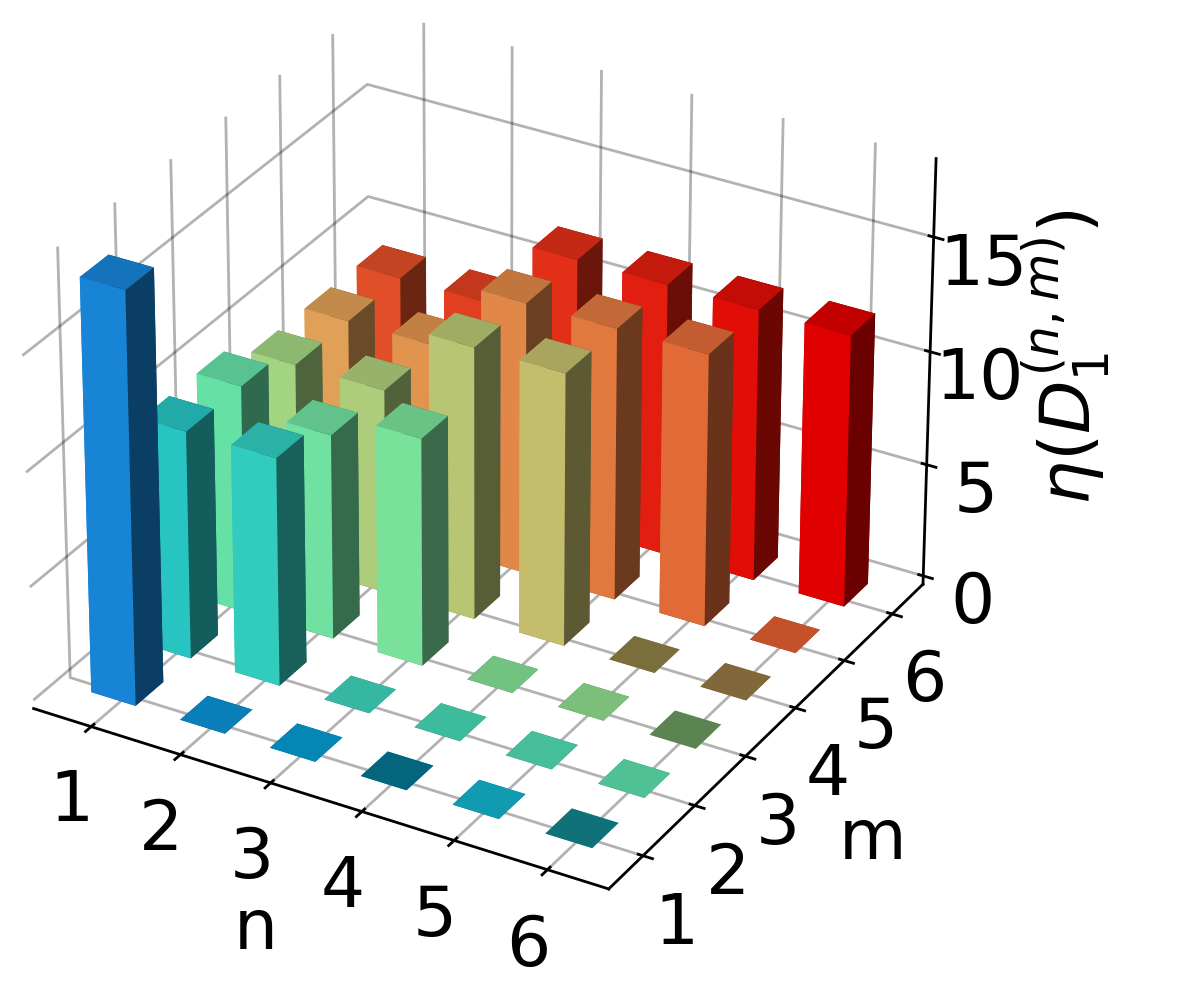}
						\caption{}
						\label{application:3:Features:b}
					\end{subfigure}%
					\hspace{0.1cm}
					\begin{subfigure}{.32\textwidth}
						\centering
						\includegraphics[width=.99\linewidth]{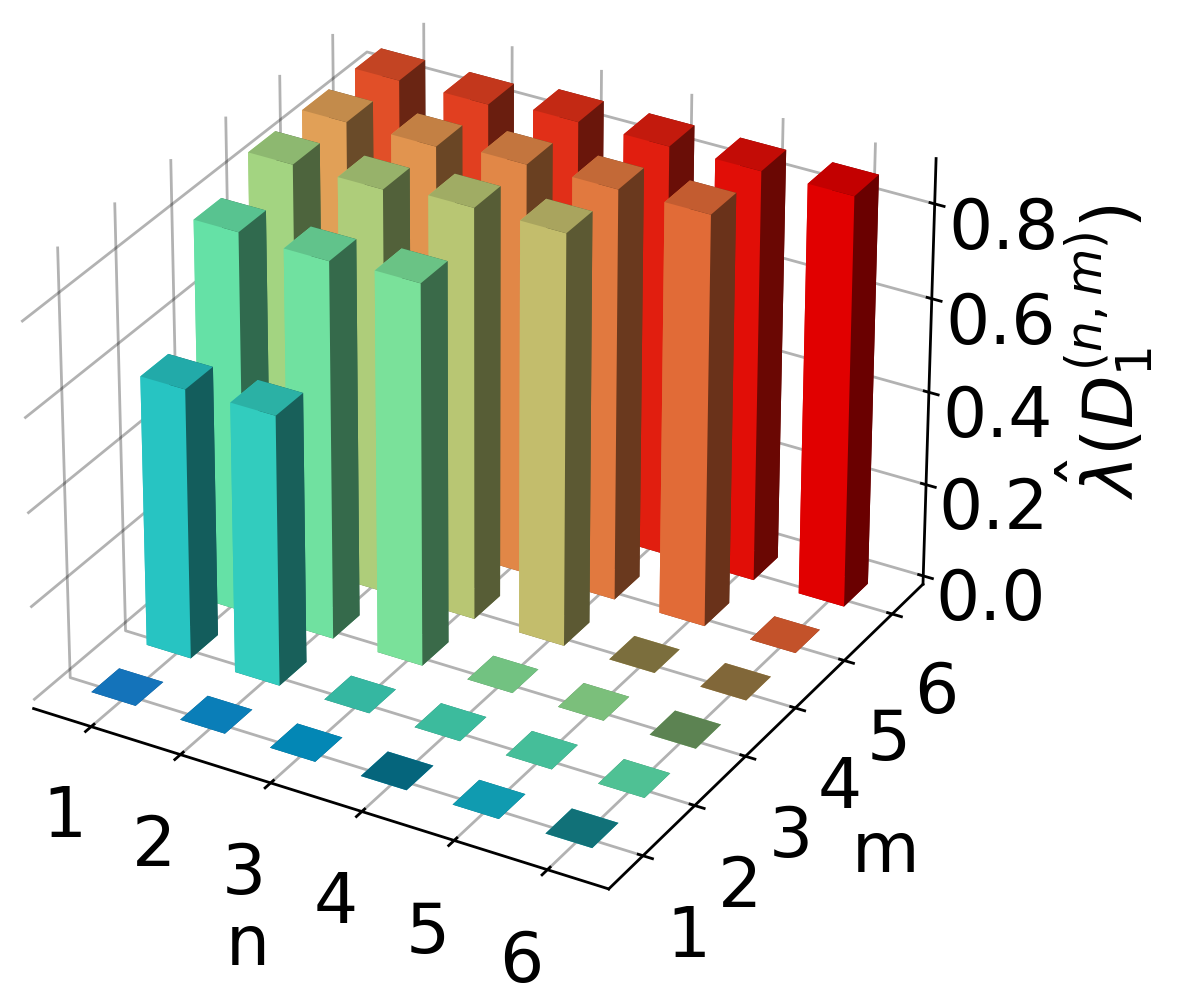}
						\caption{}
						\label{application:3:Features:c}
					\end{subfigure}
					
					\begin{subfigure}{.32\textwidth}
						\centering
						\includegraphics[width=.99\linewidth]{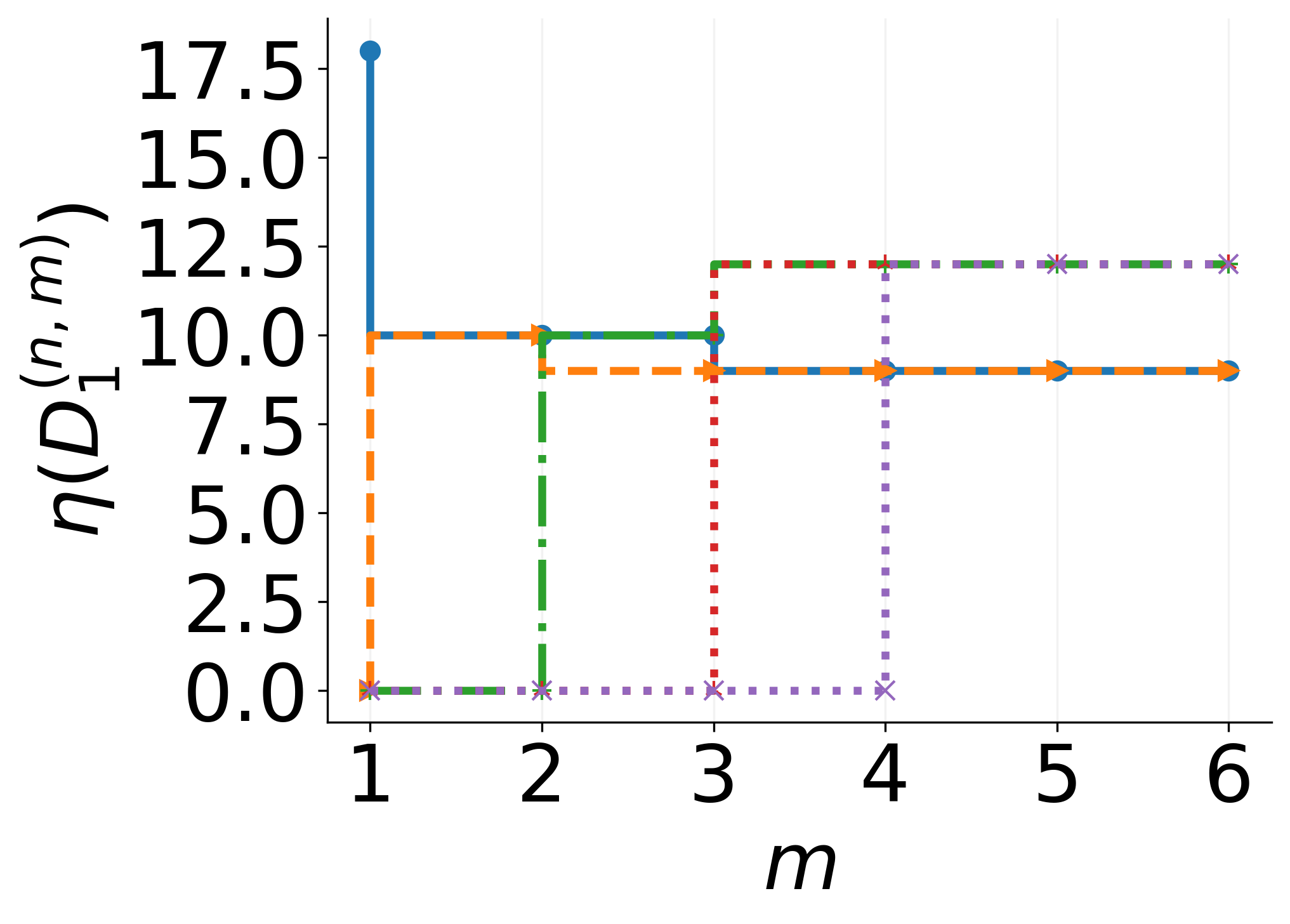}
						\caption{}
						\label{application:3:Features:e}
					\end{subfigure}%
					\hspace{0.1cm}
					\begin{subfigure}{.32\textwidth}
						\centering
						\includegraphics[width=.99\linewidth]{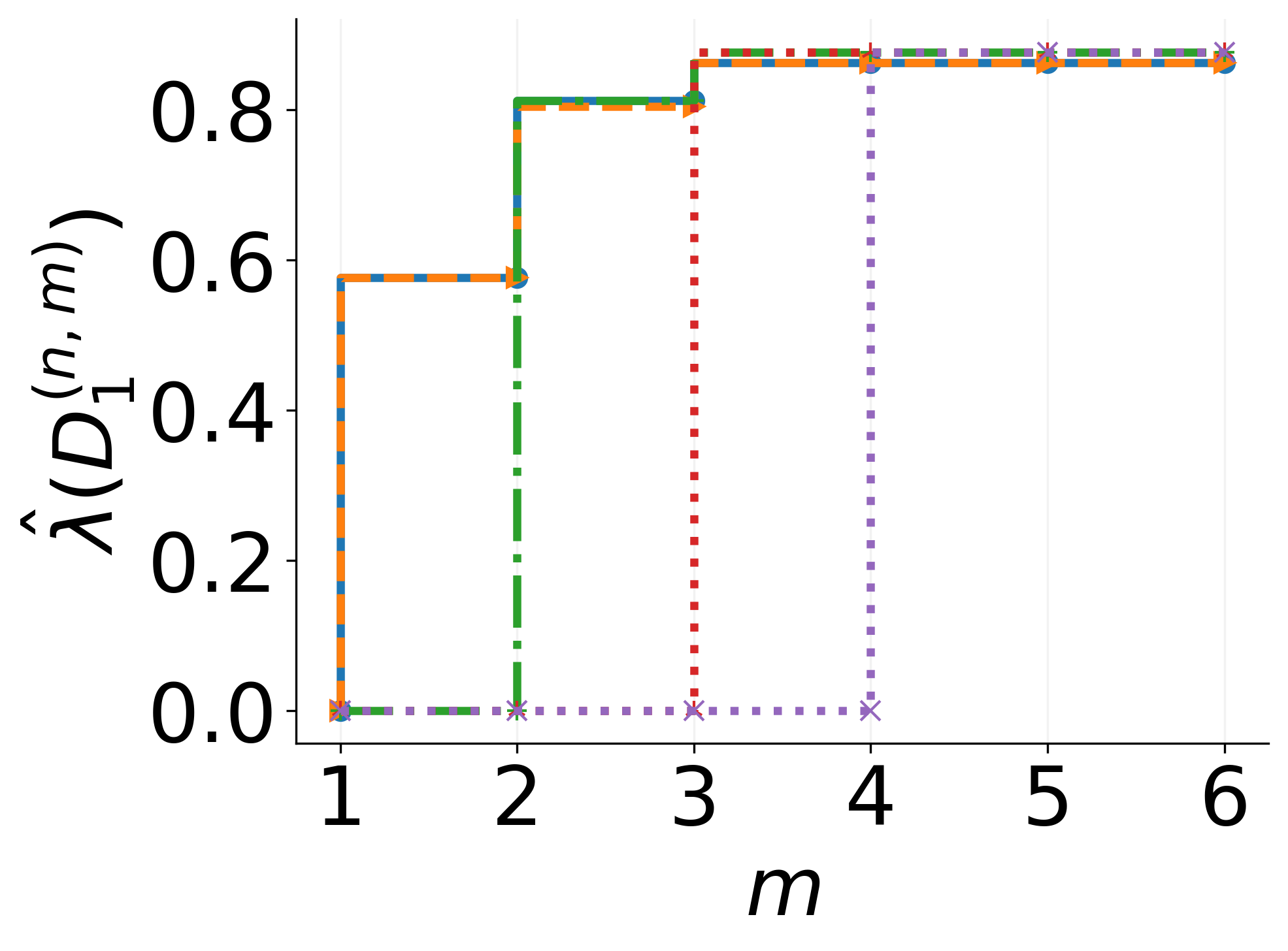}
						\caption{}
						\label{application:3:Features:f}
					\end{subfigure}
					
					\begin{subfigure}{.70\textwidth}
						\centering
						\includegraphics[width=.99\linewidth]{Legends}
					\end{subfigure}

					\caption{$\text{B}_7\text{C}_2\text{H}_9$ second isomer and its captured features.  {(A)} shows the obtained protein-ligand structure of the second isomer. {(B)} and {(C)} show $\eta(D_1^{(n,m)})$ and $\hat{\lambda}(D_1^{(n,m)})$ obtained from the second isomer. {(D)} and {(E)} show their projections on the $mz$-plane.}
					\label{application:3:Features}
				\end{figure}

				With four distinct distances at our disposal, our filtration encompasses five different digraphs, namely $H^1$, $H^2$, $H^3$, $H^4$, and $H^5$. The construction of these digraphs is based on varying filtration parameters. For instance, when the parameter is set to zero, leading to $H^1$, no connections exist between atoms, resulting in a wholly disconnected digraph. As the parameter increases to  0.97{\AA}\  and  1.1{\AA}\, $H^2$ emerges, where only hydrogen and oxygen atoms are connected, maintaining a disconnected digraph. Continuing this trend, $H^3$ emerges when the parameter falls between  1.1{\AA}\ and  1.43{\AA}\, encompassing connections between carbon and hydrogen atoms and progressively increasing the digraph's connectivity. Finally, $H^5$ represents a fully connected digraph when the parameter surpasses or equals  1.53{\AA}. The features extracted from these digraphs are visualized in Figures \ref{application:1:Features:a}-\ref{application:1:Features:c}.

				In line with the examples provided, it is evident that the nullities of persistent path Dirac operators convey significant information as anticipated. Moreover, in contrast to particular examples where the generalized mean displayed lower sensitivity than the arithmetic mean, the generalized mean for Glycogen exhibited a higher degree of data variability than the arithmetic mean. We will elucidate Figure \ref{application:1:Features:a}.  According to equations \ref{eq:1} and \ref{eqn:dirac},
				\begin{equation}
					\eta(D_1^{(n,m)}) = \beta^{(n,m)}_0+\beta^{(n,m)}_1 + \eta( \mathbf{d_2^C} ).
				\end{equation}

				In our specific case, similar to numerous other directed graph instances, we observe that $\beta^{(n,m)}_1$ and $\eta( \mathbf{d_2^C} )$ are zero, resulting in the equation:
				\begin{equation} \eta(D_1^{(n,m)}) =\beta^{(n,m)}_0=\beta_0^{m}. 
				\end{equation}
				This relationship holds for various value combinations for $n$ and $m$. Here, $\beta_0^{m}$ represents the dimension of the zeroth homology of $H^m$, which diminishes as more edges are introduced. This phenomenon elucidates a significant portion of the observed figure behavior. The peak at $(n,m)=(5,5)$ can also be accounted for. The elevated value in this instance results from adding additional loops due to the presence of carbon atoms, causing $\beta^{(n,m)}_1$ to deviate from zero and become substantial. This application underscores the capacity of persistent path Dirac operators to unravel the evolution of a molecule's shape, topological attributes, and persistence.

				In the second example of applying our analysis, we utilize two isomeric structures of $\text{B}_7\text{C}_2\text{H}_9$ and their complexes with proteins to showcase the power of persistent path Dirac operators in distinguishing between them. These structures were obtained using Schrodinger's Maestro software and are visualized in Figure \ref{application:2:Features:a} and \ref{application:3:Features:a}. To create a directed graph based on atom electronegativity, we establish directed edges following the sequence $\text{B}\rightarrow \text{H}\rightarrow \text{C}$. We also permit connections between two atoms of the same type but prohibit self-connections. We only include directed edges between atoms that share a chemical bond. Consequently, hydrogen atoms do not contribute to essential subgraphs that would generate path complexes of degrees higher than one. This simplifies the digraph when computing higher-degree components.
				The analysis reveals distinctive features captured by both isomers in Figures \ref{application:2:Features} and \ref{application:3:Features}. Notably, for large values of both $n$ and $m$, the first isomer exhibits higher nullity values than the second. This discrepancy is straightforward to explain since the first isomer forms a triangular graph, resulting in non-trivial higher-degree path complex components with more generators than the second isomer. For instance, while there is only one square-shaped graph like \ref{fig:1a}, the graph in the shape of \ref{fig:1e} lacks sufficient generators for higher-degree path complex components.

				\section{Concluding remarks}
				Path and hypergraph homologies and complexes are crucial mathematical tools for analyzing the complexities and inherent characteristics of hypergraphs and digraphs, enabling us to distinguish between these objects. Many physical entities and datasets can be endowed with innate hypergraphs and digraphs that arise from their intrinsic properties. Path and hypergraph complexes offer a valuable means to delve deeper into the intricacies of physical entities, such as molecules, and extract more comprehensive insights from associated datasets.
				
				This article presents path Dirac and hypergraph Dirac operators and persistent path Dirac and persistent hypergraph operators to extract information from path and hypergraph complexes. These operators can encompass harmonic and non-harmonic spectra and offer insights into subcomplexes within path and hypergraph complexes.
				First, we explore the homological aspects of path and hypergraph complexes and examine the fundamental subgraphs that generate higher-degree components, mainly focusing on the initial three. This analysis allows us to expedite path and hypergraph complex computations and better understand their construction. Furthermore, we demonstrate that these problems can be simplified by reducing them to standard graphs rather than dealing with complicated hypergraphs or, conversely, elevating graphs to hypergraphs with significantly fewer hyperedges. We also derive a formula for the first-degree Betti number, providing an upper bound for its value.
				
				Subsequently, we introduce persistent path Dirac and hypergraph Dirac operators and illustrate their behavior through various examples in diverse scenarios. We explore the sensitivity of these operators to filtrations, revealing their responsiveness to topological changes. Additionally, we corroborate the upper bound formula presented in Proposition \ref{prop:4} using illustrative examples.
				
				Finally, we investigate several applications of persistent path and hypergraph Dirac operators and persistent path and hypergraph homologies in biology, particularly in studying molecules. We introduce naturally induced strict preorders based on molecules, which, in turn, induce hypergraphs and digraphs featuring intricate path and hypergraph complexes. The richness of information within these path and hypergraph complexes corresponds to the complexity of distinct classes within the preorders, making them powerful tools for topological data analysis.
				
				Like homology and topological 	Laplacians, topological Diracs are a class of distinct mathematical formulations. Since persistent Laplacians have been shown to outperform persistent homology in a collection of 34 datasets in protein engineering \cite{qiu2023persistent}, persistent Diracs are expected to be powerful new tools for topological data analysis (TDA) as well.

				Table \ref{table:graph_Lap}  presents a list of (persistent) topological operators (i.e., homology, Laplacian, and Dirac) on various mathematical objects. 
				The persistent homologies and persistent Laplacians of a simplicial complex, digraph, directed flag,  path complex, 	hypergraph, and hyperdigraph have been constructed. 
				Persistent Dirac, persistent path Dirac, and persistent hypergraph Dirac have been constructed. A few more formulations, such as persistent Dirac of digraph, flag, and hyperdigraph, can be similarly developed. 
				
				\begin{table}[h!]
					\centering
					\scalebox{0.99}[1.09]{
						\begin{tabular}{c|c|c|c}
							\hline
							Topological Structure & Homologies  & Topological Laplacians & Topological Diracs \\
							\hline
							simplicial complex & simplicial homology  & Laplacian &  simplicial Dirac \\
							digraph  & homology of digraphs & digraph Laplacian & digraph Dirac**  \\
							directed flag  & directed flag homology & directed flag Laplacian &  directed flag Dirac** \\
							path complex  & path homology & path Laplacian & path Dirac*  \\
							hypergraph & homology of hypergraphs &hypergraph Laplacian &  hypergraph Dirac*\\
							hyperdigraph  & hyperdigraph homology & hyperdigraph Laplacian &  hyperdigraph Dirac** \\
							\hline
					\end{tabular}}
					\caption{A list of (persistent) topological operators (i.e., homology, Laplacian, and Dirac) on various mathematical objects, such as simplicial complex, digraph, directed flag,  path complex, 	hypergraph, and hyperdigraph. Starred objects are constructed in this article. Double-starred objects can be constructed similarly. }
					\label{table:graph_Lap}
				\end{table}

				In general, Dirac operators tend to have large matrix sizes, leading to a challenge in dealing with large data. Nonetheless, their distinctive structure, comprising blocks of zeroes and boundary matrices, allows for their representation through various specialized methods. Calculating these operators necessitates initially determining an orthonormal basis and identifying the matrices representing the boundary map. Additionally, constructing the complex necessary for persistent path Dirac operators involves identifying the inverse image. This process aligns with solving linear equations, for which effective techniques are available.


				\subsection*{Code Availability}
				
				The codes are available at
				
				\url{https://github.com/FaisalSuwayyid/persistent-path-hypergraph-dirac/}.
				\section*{Acknowledgments}
				This work was supported in part by NIH grants  R01GM126189, R01AI164266, and R01AI146210, NSF grants DMS-2052983,  DMS-1761320, and IIS-1900473,  NASA grant 80NSSC21M0023,  MSU Foundation,  Bristol-Myers Squibb 65109, and Pfizer.
				F.S. is grateful for financial support from King Fahd University of Petroleum and Minerals and thanks Dr Dong Chen for technical assistance.


				\bibliographystyle{unsrt}
				\bibliography{biblio}

\begin{thebibliography}{10}

\bibitem{kaczynski2004computational}
Tomasz Kaczynski, Konstantin~Michael Mischaikow, and Marian Mrozek.
\newblock {\em Computational homology}, volume~3.
\newblock Springer, 2004.

\bibitem{cang2017topologynet}
Zixuan Cang and Guo-Wei Wei.
\newblock {TopologyNet}: Topology based deep convolutional and multi-task
  neural networks for biomolecular property predictions.
\newblock {\em PLoS computational biology}, 13(7):e1005690, 2017.

\bibitem{grbic2022aspects}
Jelena Grbi{\'c}, Jie Wu, Kelin Xia, and Guo-Wei Wei.
\newblock Aspects of topological approaches for data science.
\newblock {\em Foundations of data science (Springfield, Mo.)}, 4(2):165, 2022.

\bibitem{chen2023persistent}
Dong Chen, Jian Liu, Jie Wu, and Guo-Wei Wei.
\newblock Persistent hyperdigraph homology and persistent hyperdigraph
  {Laplacians}.
\newblock {\em Foundation of Data Science}, 5:558--588, 2023.

\bibitem{zomorodian2004computing}
Afra Zomorodian and Gunnar Carlsson.
\newblock Computing persistent homology.
\newblock In {\em Proceedings of the twentieth annual symposium on
  Computational geometry}, pages 347--356, 2004.

\bibitem{edelsbrunner2008persistent}
Herbert Edelsbrunner, John Harer, et~al.
\newblock Persistent homology-a survey.
\newblock {\em Contemporary mathematics}, 453(26):257--282, 2008.

\bibitem{bubenik2017persistence}
Peter Bubenik and Pawe{\l} D{\l}otko.
\newblock A persistence landscapes toolbox for topological statistics.
\newblock {\em Journal of Symbolic Computation}, 78:91--114, 2017.

\bibitem{dodziuk1977rham}
Jozef Dodziuk.
\newblock {De Rham-Hodge} theory for l 2-cohomology of infinite coverings.
\newblock {\em Topology}, 16(2):157--165, 1977.

\bibitem{chen2021evolutionary}
Jiahui Chen, Rundong Zhao, Yiying Tong, and Guo-Wei Wei.
\newblock Evolutionary {de Rham-Hodge} method.
\newblock {\em Discrete and continuous dynamical systems. Series B},
  26(7):3785, 2021.

\bibitem{wang2020persistent}
Rui Wang, Duc~Duy Nguyen, and Guo-Wei Wei.
\newblock Persistent spectral graph.
\newblock {\em International journal for numerical methods in biomedical
  engineering}, 36(9):e3376, 2020.

\bibitem{memoli2022persistent}
Facundo M{\'e}moli, Zhengchao Wan, and Yusu Wang.
\newblock {Persistent Laplacians}: Properties, algorithms and implications.
\newblock {\em SIAM Journal on Mathematics of Data Science}, 4(2):858--884,
  2022.

\bibitem{qiu2023persistent}
Yuchi Qiu and Guo-Wei Wei.
\newblock Persistent spectral theory-guided protein engineering.
\newblock {\em Nature Computational Science}, 3(2):149--163, 2023.

\bibitem{chen2022persistent}
Jiahui Chen, Yuchi Qiu, Rui Wang, and Guo-Wei Wei.
\newblock Persistent {Laplacian projected Omicron BA.4 and BA.5} to become new
  dominating variants.
\newblock {\em Computers in Biology and Medicine}, 151:106262, 2022.

\bibitem{meng2021persistent}
Zhenyu Meng and Kelin Xia.
\newblock Persistent spectral--based machine learning {(PerSpect ML)} for
  protein-ligand binding affinity prediction.
\newblock {\em Science advances}, 7(19):eabc5329, 2021.

\bibitem{wang2021hermes}
Rui Wang, Rundong Zhao, Emily Ribando-Gros, Jiahui Chen, Yiying Tong, and
  Guo-Wei Wei.
\newblock {HERMES:} persistent spectral graph software.
\newblock {\em Foundations of data science (Springfield, Mo.)}, 3(1):67, 2021.

\bibitem{wei2021persistent}
Xiaoqi Wei and Guo-Wei Wei.
\newblock Persistent sheaf {Laplacians}.
\newblock {\em arXiv preprint arXiv:2112.10906}, 2021.

\bibitem{shepard1985cellular}
Allen~Dudley Shepard.
\newblock {\em A cellular description of the derived category of a stratified
  space}.
\newblock Brown University, 1985.

\bibitem{hansen2019toward}
Jakob Hansen and Robert Ghrist.
\newblock Toward a spectral theory of cellular sheaves.
\newblock {\em Journal of Applied and Computational Topology}, 3:315--358,
  2019.

\bibitem{grigor2012homologies}
Alexander Grigor'yan, Yong Lin, Yuri Muranov, and Shing-Tung Yau.
\newblock Homologies of path complexes and digraphs.
\newblock {\em arXiv preprint arXiv:1207.2834}, 2012.

\bibitem{wang2022persistent}
Rui Wang and Guo-Wei Wei.
\newblock Persistent path laplacian.
\newblock {\em Foundations of Data Science}, 5:26--55, 2023.

\bibitem{ameneyro2022quantums}
Bernardo Ameneyro, Vasileios Maroulas, and George Siopsis.
\newblock Quantum persistent homology.
\newblock {\em arXiv preprint arXiv:2202.12965}, 2022.

\bibitem{ameneyro2022quantum}
Bernardo Ameneyro, George Siopsis, and Vasileios Maroulas.
\newblock Quantum persistent homology for time series.
\newblock In {\em 2022 IEEE/ACM 7th Symposium on Edge Computing (SEC)}, pages
  387--392. IEEE, 2022.

\bibitem{calmon2022dirac}
Lucille Calmon, Juan~G Restrepo, Joaqu{\'\i}n~J Torres, and Ginestra Bianconi.
\newblock Dirac synchronization is rhythmic and explosive.
\newblock {\em Communications Physics}, 5(1):253, 2022.

\bibitem{calmon2023dirac}
Lucille Calmon, Michael~T Schaub, and Ginestra Bianconi.
\newblock Dirac signal processing of higher-order topological signals.
\newblock {\em arXiv preprint arXiv:2301.10137}, 2023.

\bibitem{wee2023persistent}
JunJie Wee, Ginestra Bianconi, and Kelin Xia.
\newblock {Persistent Dirac} for molecular representation.
\newblock {\em arXiv preprint arXiv:2302.02386}, 2023.

\bibitem{bressan2016embedded}
Stephane Bressan, Jingyan Li, Shiquan Ren, and Jie Wu.
\newblock The embedded homology of hypergraphs and applications.
\newblock {\em arXiv preprint arXiv:1610.00890}, 2016.

\bibitem{grigor2020path}
Alexander~A Grigor’yan, Yong Lin, Yu~V Muranov, and Shing-Tung Yau.
\newblock Path complexes and their homologies.
\newblock {\em Journal of Mathematical Sciences}, 248:564--599, 2020.

\bibitem{consonni2008new}
Viviana Consonni and Roberto Todeschini.
\newblock New spectral indices for molecule description.
\newblock {\em Match}, 1:2, 2008.

\bibitem{alhevaz2020generalized}
Abdollah Alhevaz, Maryam Baghipur, Hilal~A Ganie, and Yilun Shang.
\newblock On the generalized distance energy of graphs.
\newblock {\em Mathematics}, 8(1):17, 2020.

\bibitem{nguyen2019agl}
Duc~Duy Nguyen and Guo-Wei Wei.
\newblock {AGL}-score: algebraic graph learning score for protein--ligand
  binding scoring, ranking, docking, and screening.
\newblock {\em Journal of chemical information and modeling}, 59(7):3291--3304,
  2019.

\bibitem{pauling1932nature}
Linus Pauling.
\newblock The nature of the chemical bond. iv. the energy of single bonds and
  the relative electronegativity of atoms.
\newblock {\em Journal of the American Chemical Society}, 54(9):3570--3582,
  1932.

\bibitem{MolView}
Herman Bergwerf.
\newblock Glycogen.
\newblock \url{https://molview.org/}.

\end{thebibliography}
				
				\medskip
			\end{document}